\documentclass[a4paper]{amsart}

\usepackage[english]{babel}
\usepackage[utf8x]{inputenc}
\usepackage[T1]{fontenc}
\usepackage{physics}

\usepackage[a4paper,top=3cm,bottom=2cm,left=3cm,right=3cm,marginparwidth=1.75cm]{geometry}

\usepackage{amsmath}
\usepackage{amssymb}
\usepackage{amsthm}
\usepackage{graphicx}
\usepackage{subcaption}
\usepackage{comment}
\usepackage[numbers]{natbib}
\usepackage[colorinlistoftodos]{todonotes}
\usepackage[colorlinks=true, allcolors=blue]{hyperref}
\usepackage{tikz-cd}

\newtheorem{theorem}{Theorem}[section]
\newtheorem{corollary}[theorem]{Corollary}

\newtheorem{proposition}[theorem]{Proposition}
\newtheorem{lemma}[theorem]{Lemma}

\newtheorem*{theorem*}{Theorem}
\newtheorem{definition}[theorem]{Definition}
\newtheorem{example}[theorem]{Example}

\DeclareMathOperator{\bB}{{\mathbb B}}
\DeclareMathOperator{\bK}{{\mathbb K}}

\DeclareMathOperator{\Cdb}{{\mathbb C}}
\DeclareMathOperator{\Rdb}{{\mathbb R}}

\DeclareMathOperator{\bF}{{\mathbb F}} 
\DeclareMathOperator{\Tdb}{{\mathbb T}}
\DeclareMathOperator{\Ndb}{{\mathbb N}}

\DeclareMathOperator{\cW}{{\mathcal W}}

\DeclareMathOperator{\cC}{{\mathcal C}}
\DeclareMathOperator{\cO}{{\mathcal O}}

\DeclareMathOperator{\cR}{{\mathcal R}}

\DeclareMathOperator{\cT}{{\mathcal T}}

\DeclareMathOperator{\cS}{{\mathcal S}}

\DeclareMathOperator{\bH}{{\mathbb H}}
\DeclareMathOperator{\bC}{{\mathbb C}}
\DeclareMathOperator{\bN}{{\mathbb N}}
\DeclareMathOperator{\bR}{{\mathbb R}}
\DeclareMathOperator{\bT}{{\mathbb T}}
\DeclareMathOperator{\bQ}{{\mathbb Q}}
\DeclareMathOperator{\bZ}{{\mathbb Z}}
\DeclareMathOperator{\maxten}{\otimes_{\rm max}}
\DeclareMathOperator{\minten}{\otimes_{\rm min}}
\DeclareMathOperator{\cten}{\otimes_{\mathfrak c}}

\DeclareMathOperator{\erten}{\otimes_{\rm er}}
\DeclareMathOperator{\elten}{\otimes_{\rm el}}

\title{Real operator systems}
\author{David P. Blecher}
\address{Department of Mathematics, University of Houston, Houston, TX
77204-3008, USA}
\email[David P. Blecher]{dpbleche@central.uh.edu}

\author{Travis B. Russell}
\address{Department of Mathematics, Texas Christian University, PO BOX 298900,
Fort Worth, Texas 76129}
\email[Travis B. Russell]{travis.b.russell@tcu.edu}
\date{Final revision of September, 2025.  To appear J. Functional Analysis} 
\begin{document}

\begin{abstract}  Operator systems are the unital self-adjoint subspaces of the bounded operators on a Hilbert space.  Complex operator systems are an important category containing the C$^*$-algebras and von Neumann algebras, which is increasingly of interest in modern analysis and also in modern quantum physics (such as quantum information theory). They have an extensive theory, and have very important applications in all of these subjects. We present here the real case of the theory of (complex) operator systems, and also the real case of their remarkable tensor product theory, due in the complex case to Paulsen and his coauthors and students (such as Kavruk), building on pioneering earlier work of Kirchberg and others. We uncover several notable differences between the real and complex theory, including the absence of minimal and maximal functors in the category of real operator systems.   We also develop very many foundational structural results for real operator systems, and elucidate how the complexification interacts with the basic constructions in the subject.
In the final two sections of our paper we study real analogues  of the Kirchberg conjectures (and of several important related problems that have attracted much interest recently), and study the deep relationships between them.   \end{abstract}

\maketitle
\tableofcontents

\bigskip

\section{Introduction}

Complex operator systems are an important 
category containing the C$^*$-algebras and  von Neumann algebras, which is increasingly of interest in modern analysis (see e.g.\ \cite{CvS}) and also in modern quantum physics (such as quantum information theory).   They  have an extensive theory
(see e.g.\ the texts \cite{Pnbook,P} and the very many papers by Paulsen and his students and collaborators, and many others), and have very important applications in all of these subjects. 

 Real structure occurs naturally and crucially
 in very many areas of mathematics and mathematical physics.  In several deep mathematical theories, at some point a crucial advance has been made by switching to the real case (e.g.\ in $K$-theory and the Baum-Connes conjecture, see for example the very recent survey \cite{BSc}, or see \cite{Ros} for some more examples).   This is sometimes because  the real category is bigger, as it is in our case, and hence allows more freedom. 
 Ruan initiated the study of real operator spaces in \cite{ROnr,RComp}, and this study was continued in 
\cite{Sharma, BT,BReal,BCK}.   Indeed \cite{BReal}, as well as providing some basic structural results for real operator spaces 
and investigation how the 
facts and structures there interact with the complexification, 
completed the verification that a large portion of the theory of complex 
operator spaces and operator algebras as represented by the text {\rm \cite{BLM}} for specificity, 
transfers to the real case.     
Here we do the same thing but for real  operator systems, that is real unital selfadjoint spaces  $\cS$ of operators on a real Hilbert space.
These of course are operator spaces.  There are many motivations for such a task, some mentioned in some of the above cited works.  For example,
 the sequels \cite{BMcI, BMcII} to  the present  paper  develop
a real noncommutative convexity, the real case of the recent and profound complex theory  developed by  Davidson and Kennedy which is based on complex operator systems. 
Classical convexity, is essentially a `real theory', as an inspection  of the texts in that subject will show.  Also, historically, 
much of the motivation for noncommutative convexity is based on real 
examples as in the work of the Helton school for example (e.g.\ see references in \cite{BMcI, BMcII}) 
so it makes sense to develop a noncommutative real convexity based on real operator systems (and hence on the present paper).

We develop here the theory of  real operator systems.  One goal of the present paper is to create for the fast growing 
community of operator systems users a resource or repository 
for the real systems theory, in addition to providing some basic structural results.   Since this is a daunting and  not very well defined task, we restrict 
ourselves  mostly to the more modest target 
of checking  the real case of what seems to be some of the most important results in the theory, and how the 
facts and structures there interact with the complexification
$\cS_c$, which is sometimes quite nontrivial. 
Much of the last third of our paper checks 
the real case of  the remarkable theory of operator system tensor products due to Paulsen and his coauthors
and students, such as Kavruk, and the exciting applications 
surrounding the Connes-Kirchberg problem \cite{P}.  This is an extensive theory, building in part on the earlier C$^*$-algebra and operator space tensor product theories (see e.g.\ \cite{Kir,ER,Pisbk,P,Pnbook,BLM,BP}).  Thus again we 
focus on what seems to be some of  the most important parts. 
For example we establish the real version of the delicate implications between the Kirchberg-Connes, and Tsirelson conjectures, etc, and prove the equivalence with the complex case.  (That the real version of the Connes' embedding
problem was equivalent to the complex was already known, see \cite{Oz,BDKS}.) We also check many other facts in the ``nuclearity/exactness/Kirchberg problem theory'' in the real case, and their relationships with the complexification.

  As we have alluded to earlier, not only do we want to check that 
the real versions of the complex theory  work, 
 we also want to know what the complexifications are 
of standard constructions, and this is often as important but much less obvious.   
For example, it is important to know that the complexification of a particular operator system tensor product is 
a particular  tensor product of the complexifications.
More generally, it is important to know for which `constructions' $F$ in the theory we have $F(X)_c = F(X_c)$ canonically completely isometrically.  In some cases 
one has to be careful with the identifications.  

Some of the topics that work out rather differently in the real case concern the relationship between the positive elements of $V$ and the positive elements of $M_n(V)$ for $n > 1$ when $V$ is a real operator system. To summarize these differences: a real matrix ordering $\{C_n\}$ with a proper cone $C_1$ may not be a proper matrix ordering; an order unit for a real matrix ordered space may fail to be a matrix order unit; the kernel of a ucp map is generally not an intersection of kernels of states; a real operator system can be isometrically completely order isomorphic to its dual  (although not completely isometrically); and there is often no maximal or minimal operator system structure extending the positive cone on a given real operator system. In connection with the last item: unlike the case for real and complex Banach spaces (see e.g.\ \cite[Chapter 3]{Pisbk}, \cite[Chapter 1]{BLM}), and unlike the case for complex operator systems viewed as ordered Banach spaces (or more accurately, as archimedean order unit $*$-vector spaces \cite{PTT}),  there is no MIN nor MAX functor for general real operator systems.   
This is related to the fact that in the real case archimedean order unit $*$-vector spaces are not usually function spaces.
Nonetheless we will prove for this class an analogue of Kadison's characterization of real archimedean order unit spaces.   
These results complement known results concerning differences between real quantum physics and complex quantum physics which have recently appeared in the literature (e.g. \cite{CDPR2}).

We do not discuss here nonunital operator systems, matrix convexity and boundary representations and the Choquet boundary and Choquet simplices, coproducts and universal systems, or many of the important applications to quantum information theory, such as 
nonlocal games, etc. We also do not attempt here to solve the  real analog of the Smith-Ward conjecture and its connection  to the Riesz separation property.  We hope to discuss these and other topics in the future, with some of these projects involving students.

To not try the readers patience we have attempted to be brief. 
Our proofs are often deceptively short, but are usually shorthand for, or are referencing deep results and arguments from, the earlier complex case which had to be checked in the real case. Indeed generally in our paper we are able to include a huge number of results in a relatively short manuscript since many proofs are similar to their complex counterparts, or follow by complexification, and thus we often need only discuss the new points that arise. 

We will begin our paper with several sections containing fundamental structural results about real operator systems. Section \ref{CE} contains a proof of the Choi-Effros-Ozawa abstract characterization theorem for real operator systems, which is slightly different from Ozawa's. Ozawa has some theory of real operator systems in the remarkable paper \cite{Oz}. As we will remark at some point in this section, our proof also gives a useful complexification for general matrix ordered spaces, and this will be important later. 
 We remark that real variants of the abstract metric-linear characterizations of operator systems due to the first author and Neal were given in  \cite[Section 2]{BReal}.  
  Section \ref{rcast} begins the real case of the basic theory of unital operator systems as well as more advanced constructions such as real operator $A$-systems and real dual operator systems. 
Section \ref{cx}
 begins with a discussion of complete $M$-projections and the real operator space centralizer space $Z(V)$  in real operator systems.  Then we characterize when a real operator space may be given a 
complex structure, and we classify such structures.
In Section \ref{csc} real C$^*$-covers are studied, and their relation with complexification.
Section \ref{archs} studies archimedeanization as well as the relationship between the matrix ordering at the first and second levels with the rest of the matrix ordering. We point out a few fundamental differences between the real and complex case here.
Section \ref{quot} is devoted to real operator system quotients and kernels, and their relation with the complexification. Section \ref{Duality} investigates duality for 
finite dimensional real operator systems, and its relation with the complexification.  Unlike in the complex case a finite dimensional real operator system can be isometrically  order isomorphic to its dual. Section \ref{nmin} considers real $n$-minimal and $n$-maximal operator system structures.  Several new points appear in the real theory here, for example we show that there is no MIN and MAX functor for general real operator systems.  
We also prove in Theorem \ref{hasn} a quite nontrivial characterization of real archimedean order unit $*$-vector spaces, the analogue of Kadison's famous characterization of function spaces (although in our case they are not necessarily function spaces).

Section \ref{tens} reviews quickly the basic operator system tensor products and their relation with the complexification. 
  Section \ref{stabil} considers the real versions of  the five main `nuclearity' related properties, and their stability under complexification. 
  A real operator system is `nuclear' (that is, has the Completely Positive Factorization Property (CPFP)) (resp.\ is exact, has the local lifting property (SLLP), has the weak expectation property (WEP), has the double-commutant expectation property (DCEP)) if and only if its  complexification has these properties. These properties are explicitly defined in Section \ref{stabil}. 
 (Ruan already pointed out some very small part of this in 2003 in the operator space setting, without details.)
 We also establish many other results in the real case, and discuss many examples. 
 
 In the Sections \ref{kk} and \ref{aptp} we find the real versions of the Kirchberg conjectures, and their relation to Tsirelson's problem.  The approach to the latter uses the second authors theory with Araiza of abstract projections in operator systems.  We also consider the  finite representability conjecture.  That these conjectures are all false gives an enormous amount of new structural information  
 about operator systems, since so much of the theory can be connected to some form of one of these conjectures. 
Also there is always the hope with regard to these and other  problems which require similar techniques, 
that it may be easier to find explicit counterexamples in the real case.
For example the open problem as to whether all three dimensional operator systems  are exact only involves three real variables in the real case.  
Also examining the real versions of these well known problems  will serve in some measure as a guide through the real theory: in order to reach our conclusions we will check the real case of very many important positive results en route.   Research is often driven by specific problems.  Moreover, 
 it will hopefully serve as a showcase of tools and techniques 
 that may be helpful in the future in proving real operator system results.  
 It also will not take us very long.  Indeed towards the end of the paper some of our proofs become briefer, because hopefully the reader has by this point become comfortable with the basic tricks and principles of `checking the real case' in the present endeavor. 

Many of our results would apply immediately to unital operator spaces $X$, via Arveson's $\cS = X + X^*$ trick (note that $(X + X^*)_c = X_c + (X_c)^*$, 
\cite{BT}).  Similarly one may obtain applications to 
general operator spaces $X$, via the Paulsen system $\cS(X)$  trick and 
the canonical complete isometry, hence  complete order isomorphism, $\cS_{\bR}(X)_c = \cS_{\bC}(X_c)$ shown in the proof of Proposition 2.10 in \cite{BT}. 
We do not usually take the time to point out such applications.

We now turn to notation.   
The reader will need to be familiar with the  basics of complex operator spaces and systems and von Neumann algebras,  
as may be found in early chapters of \cite{BLM,ER,Pnbook, Pisbk}, and e.g.\  \cite{P}. 
It would be helpful to also browse the  existing real operator space theory  \cite{ROnr,RComp,Sharma,BT,BReal,BCK}.  Some basic 
real $C^*$- and von Neumann algebra theory may be found in \cite{Li} or \cite{ARU,Good}.  We write $M_n(\bR)$ for the real $n \times n$ matrices, or sometimes simply $M_n$ when the context is clear.  Similarly in the complex case. 
We will often use the quaternions $\bH$ as an example: this is simultaneously  a real operator system, a real Hilbert space, and a real C$^*$-algebra, usually thought of as a real $*$-subalgebra of $M_4(\bR)$ or $M_2(\bC)$. It is also a complex Banach space,  but is not a complex operator space, and its complexification is $M_2(\bC)$.  
 The letters $H, K$ are usually reserved for real or complex Hilbert spaces.  Every complex Hilbert space $H$ is a real Hilbert space with the `real part' of the inner product, this is sometimes called the {\em realification} of $H$.   I.e. we forget the complex structure.   More generally we write $X_r$ 
 for a complex Banach space regarded as a real Banach space. 
The identity mapping $U_r: H \to H_r$ is isometric. However, this mapping is only real linear and not `unitary' from the complex to the real Hilbert space.  Indeed, such notion of `unitary' does not make sense using the original inner product here (note $\langle U_r(ih), U_r h \rangle_r = \Re{i\|h\|^2} = 0$).  

\begin{proposition} \label{concon} Let $H$ be a complex Hilbert space, and let  $H_r$ be its realification.  Then the identity inclusion $B_{\bC}(H) \to B_{\bR}(H_r)$ is a real unital (selfadjoint) complete order embedding.  That is, a selfadjoint 
matrix in $M_n(B_{\bC}(H))$ is in the canonical 
positive cone  $M_n(B_{\bC}(H))^+$ if and only if 
it is in the canonical 
positive cone  $M_n(B_{\bR}(H_r))^+$.
\end{proposition}

\begin{proof} This is clear from an abstract result but we display the  direct calculation.
    Let $\widetilde{\pi}$ be this identity inclusion, and $U_r: H \to H_r$ be as above, so that $\widetilde{\pi}(T) = U_r T U_r^{-1}$. 
    Suppose that $x \in B_{\bC}(H)_{\rm sa}$. Then if $h, k \in H_r$, \[ \langle \widetilde{\pi}(x) h,k \rangle_r = \Re \,     \langle x h, k \rangle  = \Re \, \langle  h, x k \rangle  = \langle h, \widetilde{\pi}(x) k \rangle_r .\] So  $\widetilde{\pi}(x)^* = \widetilde{\pi}(x)$. Now let $h \in H_r^n$ and $x \in M_n(B_{\bC}(H)_{\rm sa})$. Then 
    \begin{eqnarray}
        \langle \widetilde{\pi}^{(n)}(x) h, h \rangle_r & = & \sum_{ij} \, \Re \, \langle x_{ij} h_j, h_i \rangle \nonumber \\
        & = & \Re \,  \sum_{ij} \langle x_{ij} h_j, h_i \rangle  \nonumber \\
        & = & \Re \,  \langle x h, h \rangle  \nonumber \\
        & = & \langle x h, h \rangle. \nonumber
    \end{eqnarray}
    using that $\langle x h, h \rangle \in \mathbb{R}$. It follows that $x \geq 0$ if and only if $\widetilde{\pi}^{(n)}(x) \geq 0$.
\end{proof}  

We shall only need this once, but related to the last result, it is easy to see that the real bicommutant of a subsystem $\cS$ of  $B_{\bR}(H_r)$ is the complex bicommutant in $B_{\bC}(H).$

 For us a {\em projection}  in an algebra 
is always an orthogonal projection (so $p = p^2 = p^*$).    A  normed algebra $A$  is {\em unital} if it has an identity $1$ of norm $1$, 
and a map $T$ 
is unital if $T(1) = 1$. 
 We write $X_{\rm sa}$ for the selfadjoint elements  
 in a $*$-vector space $X$.  In the complex case it is easy to see that
 $M_n(X)_{\rm sa} \cong (M_n)_{\rm sa} \otimes X_{\rm sa}$, but a dimension count shows that this identity can easily fail for real spaces $X$ if $X_{\rm sa}$ is small.   For example $M_2(\bH)_{\rm sa}$ has dimension 6, while $M_2(\bR)_{\rm sa}$ has dimension 3 and $\bH_{\rm sa}$ has dimension 1.
 
A real operator space may either be viewed as a real subspace of $B(H)$ for a real Hilbert space $H$, or abstractly as 
a vector space with a norm $\| \cdot \|_n$ on $M_n(X)$ for each $n \in \Ndb$ satisfying  the conditions of
Ruan's characterization in  \cite{ROnr}. 
 Sometimes the sequence of norms $(\| \cdot \|_n)$ 
 is called the {\em operator space structure}. 
If $T : X \to Y$ we write $T^{(n)}$ for the canonical `entrywise' amplification taking $M_n(X)$ to $M_n(Y)$.   
The completely bounded norm is $\| T \|_{\rm cb} = \sup_n \, \| T^{(n)} \|$, and $T$ is 
completely  contractive if  $\| T \|_{\rm cb}  \leq 1$. 
 Of course $T$ is 
{\em selfadjoint} if $T(x^*) = T(x)^*$ for $x \in X$.  
A map $T$ is said to be {\em positive} if it is selfadjoint (the reader should note this carefully)
 and takes  positive elements to positive elements, and  {\em 
completely positive} if $T^{(n)}$ is  positive for all $n \in \Ndb$. A {\rm ucp map} is  unital, linear, and completely positive. Every completely positive map between real or complex
 operator systems   is automatically selfadjoint  \cite[Lemma 2.3]{BT}.   
As we said earlier, the real case of most operator space results from \cite{BLM} have been checked, for example  in \cite{BReal}.

 A {\em concrete real operator system} is a unital selfadjoint subspace of $B(H_r)$, where $H_r$ is a real Hilbert space. 
 Likewise, a {\em concrete  complex operator system} is a unital selfadjoint subspace $V$ 
 of $B(H_c)$ where $H_c$ is a complex Hilbert space.   An operator system $V$ has a canonical positive cone in $M_n(V)$ for all integers $n$.   We write this cone as 
 $M_n(V)^+$ or $C_n$.  
We will give more basic information about this in Section 2.  We recall our convention that 
 a {\em complete order embedding} (resp.\ {\em complete order  isomorphism}) is a {\em selfadjoint} completely positive linear injective  map (resp.\ bijective  map) $u : V \to W$ whose inverse is completely positive as a map on 
 $u(V)$. 
 Unital 
complete isometries (resp.\ unital surjective complete isometries) between operator systems are complete order embeddings (resp.\ complete order  isomorphisms).  

An    {\em (abstract) real  unital operator space} (resp.\ {\em (abstract) real operator system}) may be defined to be a  operator space  $X$ with a distinguished element $u \in X$ (resp.\ and also with an involution $*$) such that there exists a  unital (so  $T(u) = I$)
 real complete isometry $T : X \to B(H)$ 
 (resp.\   which is also selfadjoint). 
 In the operator system case this embedding induces a  positive cone 
 $M_n(X)^+$ in $M_n(X)$ for all integers $n$.   Since unital 
surjective complete isometries between concrete operator systems are complete order isomorphisms it follows that the 
induced matrix cones $M_n(X)^+$ are independent of the particular $H$ on which $X$ is represented as a real operator system.  
The main result in Section 2 characterizes abstract real operator system in terms of these cones.
 Some ``metric-linear'' characterizations of real operator systems are discussed  in  \cite[Section 2]{BReal}.

 Every complex operator system is a real operator system.   Since there are several ways to see this we leave this to the reader. 
 A {\em subsystem} of an operator system $V$ is a selfadjoint subspace of $V$ containing $1_V$. 
As a real operator system the quaternions $\bH$ has trivial `first level' 
(sometimes called `base level') cone'(indeed $\bH_{\rm sa} = \bR 1$), but it `makes up for this' at the higher matrix levels. 
Other common examples in this paper include $\cS_n$, the span of the identity and the first $n$ generators of the group C$^*$-algebra $C^*(\bF_n)$ and their adjoints.  We use this notation in both the real and complex case, sometimes if we wish to distinguish we will say: real $\cS_n$.
Then $\cC = C^*(\bF_\infty)$, the full group $C^*$-algebra of the free group $F_\infty$,
 or if we wish to indicate the real case will write e.g.\  $\cC_{\bR}$ or $C^*_{\bR}(\bF_\infty)$.
 We sometimes write $\bB$ for $B(\ell^2_{\bR})$, with the compacts denoted by $\bK$.

If  $T : X \to Y$ is a surjective unital complete isometry between real unital operator spaces with
$I_H \in X \subseteq B(H)$ and $I_K \in Y \subseteq B(K)$, then the canonical extension $\tilde{T}: X + X^*  \to B(K) : x + y^*  \mapsto T (x) + T (y)^*$ is well defined
for $x,y \in X$, is 
selfadjoint and is a completely isometric complete order embedding onto $Y + Y^*$ \cite{BT}.

We use $X_c$ to denote the complexification of a real vector space $X$. 
An {\em operator space complexification} of a real operator space $X$ 
is a pair $(X_c, \kappa)$ consisting of a complex operator space $X_c$ and a real linear complete isometry $\kappa : X \to X_c$ 
such that $X_c = \kappa(X) \oplus i \, \kappa(X)$ as a vector space.   For simplicity we usually identify $X$ and $\kappa(X)$ and write $X_c = X + i \, X$.
We say that the complexification is {\em reasonable} if the map $\theta_X(x+iy) = x - iy$ on $X_c$ (that is
$\kappa(x) + i \kappa(y) \mapsto  \kappa(x) - i \kappa(y)$ for $x, y \in X$), is 
a complete isometry.  Ruan proved that a  real operator space has a unique reasonable operator space complexification
$X_c = X + i X$ up to complete isometry (see \cite{RComp}, or  \cite[Theorem 2.2]{BCK}
for a  simple proof). 
We will use the notation $\theta_X$ repeatedly. 
We recall that $B_{\bR}(H)_c \cong B_{\bC}(H_c)$ for a real Hilbert space $H$, and that every complex Hilbert space is unitarily isomorphic to $H_c$ for such $H$. 

Ruan showed in \cite{RComp} that 
$(X_c)^* \cong (X^*)_c$ completely  isometrically.   For linear functionals we have 
$\| \varphi \| = \| \varphi \|_{cb}$, as follows from \cite[Lemma 5.2]{ROnr}.
  It is pointed out however in \cite[Proposition 2.8]{Sharma} that $(X_r)^* \neq (X^*)_r$  real completely isometrically
for a complex operator space 
(the mistakes in the proof of that Proposition are easily fixed).    Here $X_r$ is the space regarded as a real operator space.   For quotient operator spaces, we have $X_c/Y_c \cong (X/Y)_c$ completely isometrically.

The  complexification of a real  operator space (resp.\ system) may be identified up to real  complete isometry with the operator subspace 
(resp.\ subsystem) $V_X$ of $M_2(X)$ 
consisting of matrices of the form 
\begin{equation} \label{ofr} c(x,y) \; = \; \begin{bmatrix}
       x    & -y \\
       y   & x
    \end{bmatrix}
    \end{equation} 
    for $x, y \in X$.   We identify $c(x,y)$ with $x + iy \in X_c$. 
    It is evident from this that the canonical projection Re $: X_c \to X$ is completely contractive (resp.\ ucp).
    Note that $V_X$ may be made into a complex operator space with multiplication by $i$ implemented by  multiplication by
    $$u = \begin{bmatrix}
       0   & -1 \\
       1   & 0
    \end{bmatrix}  .$$
    Also, $H_c \cong H^{(2)}$ as real Hilbert spaces
and we have 
$$B_{\Cdb}(H_c) \subset B_{\Rdb}(H_c) \cong B(H^{(2)}) \cong M_2(B(H)).$$ These identifications are easily checked to be real complete isometries.
The canonical embedding $\kappa : B(H) \to B_{\Cdb}(H_c)$ above and the associated embedding $B(H) + i B(H) \to \kappa(B(H)) + i \kappa(B(H)) = B_{\Cdb}(H_c)$, 
when viewed as a map into $M_2(B(H))$ by the identifications above, correspond to 
the map $$x + iy  \mapsto c(x,y) \; = \; \begin{bmatrix}
       x    & -y \\
       y   & x
    \end{bmatrix} \in V_{B(H)} \subseteq M_2(B(H)) \; , \qquad x, y \in B(H).$$
Thus $B_{\Cdb}(H_c)$ is real completely isometric to $V_{B(H)}$.
 The operator $i I$ in $B_{\Cdb}(H_c)$
corresponds to the matrix $u$ above (considered as a matrix 
    in $M_2(B(H))$).
    Then  $\theta_{B(H)}(x + iy) = x -iy$ for $x, y \in B(H)$,     corresponds to the completely isometric 
    operation $z \mapsto u^* z u = - u zu$ on $V_{B(H)}$.  
     Restricting to a subspace (resp.\ subsystem)  $X \subseteq B(H)$, we see that  $\theta_{X}$ is a complete isometry (resp.\ a unital complete order isomorphism), and $X_c$ 
    is real completely isometric (resp.\  unitally completely order isomorphic) to $V_X$.

\section{The Choi-Effros-Ozawa characterization of real operator systems} \label{CE}

To motivate for nonexperts 
the definitions in this section, 
let $V \subseteq B(H_r)$ be a concrete real operator system. 
Given $z \in V$, we may write $z = x + y$, where $x \in V_{\rm sa}$ and $y \in V_{\rm as}$, by setting $x = \frac{1}{2}(z + z^*)$ and $y = \frac{1}{2}(z - z^*)$. We say that $x \in V$ is {\em positive} if $x \in V_{\rm sa}$ and $\langle x h, h \rangle \geq 0$ for all $h \in H_r$. We let $C \subseteq V_{\rm sa}$ denote the set of positive elements. If $x \in V_{\rm sa}$, then setting $t = \|x\| \geq 0$ we see that for any $h \in H_r$,
\[ \langle (tI + x)h, h \rangle = t\|h\|^2 + \langle xh, h \rangle \geq 0 \]
since $\|x\| \|h\|^2 \geq | \langle xh, h \rangle|$ by Cauchy-Schwarz. Thus, for every $x \in V_{\rm sa}$, there exists $t \geq 0$ such that $tI + x \in C$. Also, if $tI + x \in C$ for every $t > 0$, then for any $h \in H_r$ with $\|h\|=1$,
\[  0 \leq \langle (tI + x)h, h \rangle = t + \langle xh, h \rangle \]
and hence $\langle xh, h \rangle \geq 0$. So $x \in C$. Finally, suppose that both $x$ and $-x$ are positive. Then for every $h \in H_r$, $\pm \langle xh, h \rangle \geq 0$ and hence $\langle xh, h \rangle = 0$. It follows that for every $h,k \in H_r$ we have \[ \langle x (h+k), h+k \rangle = \langle x h, h \rangle + \langle x h, k \rangle + \langle x k, h \rangle + \langle x k, k \rangle = 2 \langle x h, k \rangle \] since $x=x^*$. So $\langle x h, k \rangle = 0$ for every $h, k \in H_r$ and therefore $x = 0$.

For each $n \in \mathbb{N}$, we can identify $M_n(V)$ with operators on $B(H_r^n)$ by setting $[ x_{ij}] (h_k)_{k=1}^n = (\sum_{k=1}^n x_{ik}h_k)_{i=1}^n \in H_r^n$ for each $(h_k)_{k=1}^n \in H_r^n$. With these identifications, $M_n(V)$ is also a real operator system. 
We let $C_n \subseteq M_n(V)_{\rm sa}$ denote the positive elements. If $x \in C_n$ and $y \in C_m$, then $x \oplus y \in C_{n+m}$ where
\[ x \oplus y := \begin{bmatrix} x & 0 \\ 0 & y \end{bmatrix} \]
since $(x \oplus y)^* = x \oplus y$ and, for any $h \in H_r^n$ and $k \in H_r^m$, $$\langle (x \oplus y) (h \oplus k), (h \oplus k) \rangle = \langle x h, h \rangle + \langle y k, k \rangle \geq 0. $$ Finally, let $\alpha \in M_{n,m}(\mathbb{R})$. We may regard $\alpha$ as an element of $B(H_r^m, H_r^n)$ by setting $\alpha (h_k)_{k=1}^m = (\sum_{k=1}^m \alpha_{ik} h_k)_{i=1}^n \in H_r^n$. Then for any $x \in C_n$ and $h \in H_r^m$, $\langle \alpha^T x \alpha h,h \rangle = \langle x (\alpha h), (\alpha h) \rangle \geq 0$.

We now extend the features observed above to an abstract setting.  Let $V$ be a real vector space. We call a map $*:V \to V$, denoted as $x \mapsto x^*$, an {\em involution} provided that for any $x,y \in V$ and $t \in \mathbb{R}$, $(x+ty)^* = x^* + ty^*$ and $(x^*)^* = x$. Given a real vector space $V$ with involution $*$, we may extend the involution to $M_n(V)$ by setting the $ij$ entry of $x^*$ equal to $x_{ji}^*$ for any $x \in M_n(V)$. We let $M_n(V)_{\rm sa} := \{ x \in M_n(V) : x = x^*\}$ and we let $M_n(V)_{\rm as} := \{x \in M_n(V) : x = -x^*\}$. Since $x = \frac{1}{2}(x+x^*) + \frac{1}{2}(x-x^*)$ for any $x \in M_n(V)$, we see that $M_n(V) = M_n(V)_{\rm sa} + M_n(V)_{\rm as}$ for each $n \in \mathbb{N}$. We call a sequence of subsets $\{ C_n \subseteq M_n(V)_{\rm sa} \}_{n=1}^\infty$ a {\em real matrix ordering} if for every $n,m \in \mathbb{N}$ and $\alpha \in M_{n,m}(\mathbb{R})$ we have $C_n \oplus C_m \subseteq C_{n+m}$ and $\alpha^T C_n \alpha \subseteq C_m$. A real matrix ordering $\{C_n\}$ is called {\em proper} if $\pm x \in C_n$ implies that $x = 0$, i.e. $C_n \cap -C_n = \{0\}$. An element $e \in V_{\rm sa}$ is called a {\em matrix order unit} if for every $x \in M_n(V)_{\rm sa}$ 
there exists $t > 0$ such that $t e_n + x \in C_n$ where 
\[ e_n := e \otimes I_n = \begin{bmatrix} e & & \\ & \ddots & \\ & & e \end{bmatrix}. \]
We say that a matrix order unit is {\em archimedean} if whenever 
$x \in M_n(V)_{\rm sa}$ and $te_n + x \in C_n$ for every $t > 0$ then $x \in C_n$.

\begin{definition}[Abstract real operator system]
    Let $V$ be a vector space over $\mathbb{R}$, $*: V \to V$ an involution, $\{C_n\}_{n=1}^\infty$ a proper real matrix ordering and $e \in V_{\rm sa}$ an archimedean matrix order unit for $\{C_n\}_{n=1}^\infty$. Then the tuple $(V,*,\{C_n\},e)$ is called an {\em abstract real operator system}.
\end{definition}

(As a consequence of the main result of this Section it will follow that this definition is equivalent to the one given in the introduction.) 
In the text preceeding the definitions, we have proven the following.

\begin{proposition}
    Let $H_r$ be a real Hilbert space and $V \subseteq B(H_r)$ be a real operator system. Then $(V,*,\{C_n\},I)$ is an abstract operator system, where $*$ denotes the operator adjoint and $C_n \subseteq M_n(V)_{\rm sa}$ denotes the positive operators.
\end{proposition}

In Section 17 of \cite{Oz}, Ozawa gives a slightly different definition of a real operator system, which includes the assumption that the set 
\[ \{x \in V :  \begin{bmatrix} \epsilon I & x \\ x^* & \epsilon I \end{bmatrix} \geq 0 \text{ for all } \epsilon > 0 \} \]
is $(0)$.   In our case (of real operator systems), one can show that this condition is automatic.  With these axioms, Ozawa constructs the universal C*-cover as a real C*-algebra and demonstrates a complete order embedding into that C*-algebra. We will recover Ozawa's results by different methods, proceeding by complexification.

Given a abstract real operator system, we wish to define a corresponding abstract complex operator system. To motivate the definition, we consider the concrete case.  Some of what follows is well known, but we shall need some of the formal calculations anyway in our next proof.

Let $H_r$ be a real Hilbert space and let $V \subseteq B(H_r)$ be a real operator system. We define the complex Hilbert space $H_c = H_r + iH_r$ to be the complexification of the vector space $H_r$. Given $x,y,z,w \in H_r$, we define $\langle x+iy, z+iw \rangle = (\langle x,z \rangle + \langle y,w \rangle) + i(\langle y,z \rangle - \langle x,w \rangle)$. It is well-known that the resulting inner-product space is complete and hence a complex Hilbert space. Let $\widetilde{V} = V + iV$ denote the complexification of $V$. Then we have an embedding $\widetilde{V} \subseteq B(H_c)$ given by setting $(x+iy)(h+ik) = (xh -yk) + i(xk + yh)$. It is easily verified that the operators $x+iy$ are bounded and that $(x+iy)^* = x^* - iy^*$ for every $x,y \in V$.

Now suppose that $(x+iy)^* = x+iy$. Write $x=x_{\rm sa} + x_{\rm as}$ and $y=y_{\rm sa} + y_{\rm as}$, where $x_{\rm sa}, y_{\rm sa}$ are hermitian and $x_{\rm as}, y_{\rm as}$ are skew-adjoint (antisymmetric). Then we have $x_{\rm sa} - x_{\rm as} - iy_{\rm sa} + iy_{\rm as} = x_{\rm sa} + x_{\rm as} + iy_{\rm sa} + iy_{\rm as}$. Hence $x_{\rm as} + iy_{\rm sa} = 0$. But this implies $x_{\rm as} = y_{\rm sa} = 0$ by the linear independence of the real and imaginary parts of the complexification of a vector space. Hence we have $x=x^*$ and $y=-y^*$. It is well known  that $x+iy \geq 0$ if and only if
\[ c(x,y) := \begin{bmatrix} x & -y \\ y & x \end{bmatrix} \geq 0. \] 
That is, $x+iy \geq 0$ as an operator on $B(H_c)$ if and only if $x=x^*$, $y=-y^*$, and $c(x,y) \geq 0$ as an operator on $B(H_r^2)$. This
motivates the following definition.

\begin{definition}
    Let $(V,*,\{C_n\},e)$ be an abstract real vector space. Let $\widetilde{V} = V + iV$ denote the complexification of $V$. For each $n \in \mathbb{N}$ and for each $x,y \in M_n(V)$, define $(x+iy)^* := x^* - iy^*$. For each $n \in \mathbb{N}$, let
    \[ \widetilde{C}_n := \{ x + iy : x \in M_n(V)_{\rm sa}, y \in M_n(V)_{\rm as}, c(x,y) \in C_{2n} \}. \]
    We call the tuple $(\widetilde{V}, *, \{ \widetilde{C}_n \}, e)$ the {\em complexification} of $V$.
\end{definition}

\begin{theorem} \label{Thm: Complexification is an op sys}
    Let $(V,*,\{C_n\},e)$ be an abstract real operator system. Then the complexification $(\widetilde{V}, *, \{ \widetilde{C}_n \}, e)$ is an abstract complex operator system.
\end{theorem}

\begin{proof}
    It is clear that $*$ is a conjugate-linear involution on $\widetilde{V}$. Suppose that $x \in M_n(V)_{\rm sa}$ and $y \in M_n(V)_{\rm as}$. Then $(x+iy)^* = x^*-iy^* = x+iy$. Hence $\widetilde{C}_n \subseteq M_n(\widetilde{V})_{\rm sa}$. Repeating the arguments preceding the above definition, we see that every element of $M_n(\widetilde{V})_{\rm sa}$ may be expressed uniquely in the form $x + iy$ where $x \in M_n(V)_{\rm sa}$ and $y \in M_n(V)_{\rm as}$.

    Suppose that $x+iy \in \widetilde{C}_n$ and $z+iw \in \widetilde{C}_m$ where $x,z \in M_n(V)_{\rm sa}$ and $y,w \in M_n(V)_{\rm as}$. Then $(x+iy) \oplus (z + i w) = (x \oplus z) + i (y \oplus w) \in M_{n+m}(\widetilde{V})_{\rm sa}$ with $(x \oplus z) \in 
    M_{n+m}(V)_{\rm sa}$ and $(y \oplus w) \in M_{n+m}(V)_{\rm as}$. Then
    \[ c(x \oplus z, y \oplus w) = \begin{bmatrix} x \oplus z & -(y \oplus w) \\ y \oplus w & x \oplus z \end{bmatrix} = P^T (c(x,y) \oplus c(z,w)) P \]
    where 
    \[ P = \begin{bmatrix} 1 & 0 & 0 & 0 \\ 0 & 0 & 1 & 0 \\ 0 & 1 & 0 & 0 \\ 0 & 0 & 0 & 1 \end{bmatrix}. \] 
    Since $c(x,y) \in C_{2n}$ and $c(z,w) \in C_{2m}$, we conclude that $c(x \oplus z, y \oplus w) \in C_{2(n+m)}$. Hence $(x+iy) \oplus (z + i w) \in \widetilde{C}_{n+m}$. Therefore $\widetilde{C}_n \oplus \widetilde{C}_m \subseteq \widetilde{C}_{n+m}$.

    Next suppose that $x+iy \in \widetilde{C}_n$ with $x \in M_n(V)_{\rm sa}$ and $y \in M_n(V)_{\rm as}$, and let $\alpha \in M_{n,m}(\mathbb{C})$. Breaking $\alpha$ into its real and imaginary parts entrywise, we may write $\alpha = A + iB$ where $A,B \in M_{n,m}(\mathbb{R})$. Using $\alpha^* = (A^T - iB^T)$, we have
    \[ \alpha^*(x+iy)\alpha = \hat{x} + i\hat{y} \]
    where $\hat{x} = A^TxA + B^TxB + B^TyA - A^TyB$ and $\hat{y} = A^TxB - B^TxA + A^TyA + B^TyB$. Now 
    \begin{eqnarray} c(\hat{x},\hat{y}) & = & \begin{bmatrix} \hat{x} & -\hat{y} \\ \hat{y} & \hat{x} \end{bmatrix} \nonumber \\
    & = & \begin{bmatrix} B^T & -A^T \\ A^T & B^T \end{bmatrix} \begin{bmatrix} x & -y \\ y & x \end{bmatrix} \begin{bmatrix} B & A \\ -A & B \end{bmatrix} \in C_{2m}. \nonumber
    \end{eqnarray}
    It follows that $\alpha^* \widetilde{C}_n \alpha \subseteq \widetilde{C}_m$.

    To see that each $\widetilde{C}_n$ is proper, suppose that $\pm (x + iy) \in \widetilde{C}_n$. Then \[ \pm c(x,y) = \pm \begin{bmatrix} x & -y \\ y & x \end{bmatrix} \in C_{2n}. \] Since $C_{2n}$ is proper, $c(x,y)$ must be the zero matrix and thus $x=y=0$. So $\widetilde{C}_n$ is proper.

    It remains to check that $e$ is an archimedean matrix order unit for $\{\widetilde{C}_n\}_{n=1}^\infty$. To see that $e$ is a matrix order unit, let $x \in M_n(V)_{\rm sa}$ and $y \in M_n(V)_{\rm as}$. Then there exists $t > 0$ such that $te_{2n} + c(x,y) \in C_{2n}$. Since $te_{2n} + c(x,y) = c(te_n + x, y)$, it follows that $te_n + x + iy \in \widetilde{C}_n$. So $e$ is a matrix order unit. Similarly, if $te_n + x + iy \in \widetilde{C}_n$ for all $t > 0$, then $c(te_n + x, y) = te_{2n} + c(x,y) \in C_{2n}$ for all $t > 0$. Since $e$ is an archimedean matrix order unit for $\{C_n\}_{n=1}^\infty$, $c(x,y) \in C_{2n}$ and therefore $x+iy \in \widetilde{C}_n$. This concludes the proof.
\end{proof}

Since $\widetilde{V}$ is an abstract complex operator system, there exists a unital complex linear complete order embedding $\widetilde{\pi}: \widetilde{V} \to B(H)$ for some complex Hilbert space $H$ by the 
(complex)  Choi-Effros characterization. We 
may view this as a representation of $V$ on the real Hilbert space
$H_r$, the realification of $H$ (i.e.\ we forget the complex structure).  

\begin{corollary} \label{coreisab} 
Let $(V,*,\{\widetilde{C}_n\},e)$ be a complex operator system. Then $(V,*,\{\widetilde{C}_n\},e)$, regarded as a real vector space, 
is an abstract  real operator system and may be represented via a real 
complete order embedding as a concrete  real operator system.
\end{corollary}

\begin{proof} Consider a complex operator system $V$ complex complete order embedded in $B_{\bC}(H)$.  
By Proposition \ref{concon} we have  $V$ is   real complete order embedded in 
 $B _{\bR}(H)$ as a concrete real operator  system.        \end{proof}

The following is now easy to see: 

\begin{proposition} \label{inco}
    Let $(V,*,\{C_n\},e)$ be a real operator system with complexification $(\widetilde{V},*,\{\widetilde{C}_n\},e)$. If $\widetilde{V}$ is regarded as a real operator system, then the inclusion $j: V \to \widetilde{V}$ is a  unital (selfadjoint) complete order embedding of real operator systems.
\end{proposition}

{\bf Remark.}  The previous proofs work to show that the 
     complexification $(\widetilde{V},*,\{\widetilde{C}_n\},e)$
     of a matrix ordered space $(V,*,\{C_n\},e)$ is a matrix ordered space with $V$ contained via a unital complete order embedding,
     and $e$ will  be a matrix order unit (resp.\ be archimedean) in $\widetilde{V}$ if it had that property in $V$.  

\bigskip

Combining the preceding propositions, we get the following representation theorem for real operator systems.

\begin{theorem} \label{ce} 
    Let $(V, *, \{C_n\}, e)$ be an abstract real operator system. Then there exists a real Hilbert space $H_r$ and a unital (self-adjoint) complete order embedding $\pi: V \to B(H_r)$.
\end{theorem}

We say that an operator system complexification $W$ of $V$ is {\em reasonable} if the map $x + iy \mapsto x - iy$, for $x,y \in V$, is a complete order isomorphism of $W$.  As in the operator space case the operator system complexification is the unique 
reasonable complexification:  
 
 \begin{theorem} \label{rcth} {\rm (Ruan's unique complexification theorem) } \ Let $V$ be a real operator system. Then $V$ possesses a  reasonable operator system complexification, which is unique up to unital complete order isomorphism.   That is if $Y_1, Y_2$ are complex operator systems such that for $k = 1, 2$ if $u_k : V \to Y_k$ is a real linear unital complete order embedding
with $u_k(V) + i u_k(V) = Y_k$, and such that $u_k(x) + i u_k(y) \mapsto u_k(x) - i u_k(y)$ is a complete order isomorphism, 
then there exists a unique surjective complex linear  unital  complete order isomorphism $\rho : Y_1 \to Y_2$ with $\rho \circ u_1 = u_2$.
 \end{theorem}

 \begin{proof} Since real linear unital complete order embeddings (resp.\ isomorphisms) are the same as real linear unital complete isometries (resp.\ surjective isometries), this follows from the operator space case of the theorem.   
 \end{proof}

\begin{proposition}  If $V$ is a real operator system then the complexification $\widetilde{V}$ above is reasonable and agrees with the canonical unital operator system complexification $V_c$ above up to unital linear complete isometry and 
complete order isomorphism.   
 \end{proposition}

 \begin{proof}  To see that it is reasonable, suppose $x + iy \in \widetilde{C}_n$ with $x \in M_n(V)_{\rm sa}, y \in M_n(V)_{\rm as}$, so that $c(x,y) \geq 0$.  Multiplying by unitary permutations we see $c(x,-y) \geq 0$, so that $x - iy \in \widetilde{C}_n$.
 By symmetry the map $x + iy \mapsto x - iy$, for $x,y \in V$, is a complete order isomorphism, and completely isometric isomorphism.  The result follows by Ruan's theorem and Theorem \ref{rcth}.
\end{proof}

\begin{proposition} \label{cco}  A complex operator system is the reasonable complexification of a real operator system if and only if 
 it has an  ``operator system conjugation'', that is, a conjugate linear period 2 bijective unital complete order isomorphism $\theta$.
 In this case it is the operator system complexification of the fixed points of $\theta$.
 \end{proposition}

\begin{proof} This is obvious if one has seen the operator space case, due to Ruan.   See  e.g.\  \cite[Proposition 2.1]{BCK} for the method if needed.
 \end{proof}

 {\bf Remark.}  As is well known, if we only care about the first level cone $C_1$, 
 complex operator systems are the same as function systems, and thus have a  simple ordered space characterization in terms of an archimedean order unit (see \cite[Section 4.3]{KRI}, \cite{Alfsen,PTT}, and 
 \cite[Proposition 6.1]{BMcI}).  This category has several nice features, including MIN and MAX functors into the category of complex operator systems.   The category ROSy1 of real operator systems with their  first level 
 cone 
 only
  are not the same as real function systems since the latter objects can contain no skew elements.    
 We will not use this but it is interesting that one can characterize the real 
 operator systems together with only their first cone $C_1$ as the fixed points of a period 2 conjugate linear automorphism on a complex function system. 
However in the later result Theorem \ref{hasn} we will prove an ordered space characterization of this category ROSy1 in terms of an archimedean order unit. 
 Indeed these are exactly the 
 real $*$-vector spaces $V$ with a positive cone $V^+ \subseteq V_{\rm sa}$ and archimedean order unit $V$.  
 We will also demonstrate  in the same section that there exist no MIN and MAX functors in this setting, in contrast to  the complex theory. 

 \bigskip

 Finally, we discuss the case of real operator systems 
 with trivial involution (i.e.\ the identity map $x^* = x$).  These are the real operator systems with 
 $\cS = \cS_{\rm sa}$.  It is also easy to see that they coincide with the operator systems  which are the selfadjoint part of complex operator systems.  (Or, of real operator systems.)  Thus they form an important class of real operator systems somewhat analogous to the class of JC-algebras.  
 Note however that $M_n(\cS)$ has a nontrivial involution, the transpose, and may have nontrivial skew elements. 
 
 This class may be characterized as the real vector spaces $V$ with proper cones $C_n$ of symmetric matrices (i.e.\ $x = x^T$) in $M_n(\cS)$ satisfying the usual conditions
 $C_n \oplus C_m \subseteq C_{n+m}$ and $\alpha^T C_n \alpha \subseteq C_m$ for $\alpha \in M_{n,m}(\bR)$, having an archimedean matrix order unit.  By the main theorem of this section $\cS$ is a real operator system with identity involution, and there is 
 a unital complete order embedding $u : V \to B(H)_{\rm sa}$
 onto a concrete operator subsystem of $B(H)_{\rm sa}$. 
 Note that $V = (\tilde{V})_{\rm sa}$. 
 They are also order isomorphic to real function systems.

\section{Properties of real operator systems} \label{rcast}

In this section, we check that many of basic theorems and constructions (for instance, those found in \cite{Pnbook}) for complex operator systems also hold for real operator systems. We begin with some basic and useful facts.  If $\cS$ is a real operator system then it is easy to see (e.g.\ from facts in Section 2 such as Proposition \ref{inco}) that $(\cS_c)_{\rm sa} \cap \cS = $$\cS_{\rm sa}$, and $(\cS_c)^+ \cap \cS = $$\cS^+$. 

By a (real) state on a real operator system we mean a unital real valued functional which is contractive, or equivalently selfadjoint and positive,
or equivalently ucp \cite{BT}.   A real separable operator system possesses a faithful state.
To see this notice that the generated C$^*$-algebra is separable, and so is its complexification. Choose any faithful state $\psi$ on its complexification.   The real part $g$ of its restriction to $\cS$ is faithful: if $x \in \cS^+$ satisfies 
    $${\rm Re} \, \psi(x) =\psi(x) =0 $$
    then $x=0$.   Note that the complexification $g_c$ of a faithful state $g$ is a faithful state.   Indeed if $x+iy \geq 0$ for $x,y \in \cS$, and $g_c(x+iy) =0$, then $x=0$ since $g$  is faithful.   However since $x \pm iy \geq 0$ this implies $y=0$.
The fact that a selfadjoint element $u$ is positive if $f(u)$ is positive for all states $f$ of an operator system, in the real case, can be deduced from  Proposition 5.2.6 in  \cite{Li}.  For an  operator system $\cS$ we have $\cS = \cS_{\rm sa} \oplus \cS_{\rm as}$, where $\cS_{\rm as}$ are the skew or antisymmetric elements ($x^* = -x$).  Note that  selfadjoint functionals annihilate $\cS_{\rm as}$. 
As in \cite[Lemma 5.1]{BReal}, $\cS_c \cong \cS \oplus \cS^{\circ}$ unitally complex complete order isomorphically if 
$\cS$ is a complex operator system.   Here $\cS^{\circ}$ is the `opposite
operator system', identifiable with the obvious subspace of the opposite C$^*$-algebra $B^{\circ}$ if $\cS$ is a subsystem of $B$.   As operator systems, 
$\cS^{\circ}$ is completely order isomorphic (via $\bar{x}^*$) to the {\em conjugate operator system} 
$\bar{\cS}$.  We recall that if $X$ is a complex operator space then
$\bar{x}$ is the set of formal symbols $\bar{x}$ for $x \in X$, with the operator space structure making $X \to \bar{X} : x \mapsto \bar{x}$ a  conjugate linear complete isometry. 
So $\cS_c \cong \cS \oplus^\infty \bar{\cS}$.

\begin{lemma} \label{ncomma} Let $u : X \to B$ be a real completely contractive 
(resp.\ completely positive, completely contractive homomorphism, $*$-homomorphism) map from a real operator space (resp.\ system, operator algebra, C$^*$-algebra) $X$ into a complex C$^*$-algebra $B$.
Then $u$ extends uniquely to a complex linear completely contractive 
(resp.\ completely positive, completely contractive homomorphism, $*$-homomorphism) map $X_c \to B$.
\end{lemma}

\begin{proof}    
We obtain a 
complex linear complete contraction (resp.\ completely positive, completely contractive homomorphism, $*$-homomorphism) $X_c \to B_c$.  It is shown in 
\cite{BReal} that $B^c \cong B \oplus \bar{B}$.  The map $(b, \bar{c}) \mapsto b$
is a complex linear $*$-homomorphism
$B \oplus \bar{B} \to B$.   Composing these maps gives our result. 
\end{proof} 

Any selfadjoint bounded functional $\psi$ on an operator system is the difference of two (completely) positive  functionals whose norms have sum $\|\psi\|$.   Indeed by the Hahn-Banach theorem $\psi$ can be extended to a functional on a containing C$^*$-algebra of the same norm.  We may assume that this is  selfadjoint by averaging with the involution of the functional. By the C$^*$-theory there exist positive selfadjoint linear functionals $\psi_+, \psi_-$ such that $\psi = \psi_+ - \psi_-$ and $\|\psi\| = \|\psi_+\| + \|\psi_-\|$. Restricting to the operator system the result follows. 
A selfadjoint completely bounded map $\Phi : \cS \to B(H)$ is the difference of two completely positive maps into $B(H)$.  Indeed $\Phi$ extends  as above
 to a selfadjoint completely bounded map $\hat{\Phi}$ on $C^*(\cS)$.
    In the complex case the latter map is the difference of two completely positive maps, so is dominated by the sum of these.
    In the real case the complexification of $\hat{\Phi}$ 
    is selfadjoint, so is the difference of two completely positive maps.  One then gets the same fact for $\Phi$ after composing with the projection $B(H)_c \to B(H)$. 

    In places the reader will also need to be familiar with the theory of the injective envelope $I(X)$.  In the complex case this may be found 
in e.g.\ \cite{BLM,ER,Pnbook}.   The real case was initiated in \cite{Sharma}, and continued in \cite{BCK} and  \cite{BReal} (see particularly Section 4 there). 

A systematic check of \cite{Pnbook} reveals that almost all operator system results there are valid in the real case, provided we assume that all positive  maps and functionals are taken to be selfadjoint (that is, $T(x^*) = T(x)^*$, or for real functionals $\varphi(x^*) = \varphi(x)$).  (We already said that 
completely positive maps on real operator systems are selfadjoint.)  
 We will assume this in the discussion below. 
Some of the results from \cite{Pnbook} have already appeared for the real case in e.g.\ 
\cite{ROnr,RComp,Sharma, Oz, BT, BCK,BReal,CDPR1, CDPR2}, and we will not repeat these.  For example Ruan already checked the real Stinespring and Arveson-Wittstock-Hahn-Banach theorems, etc.   
 Or, it is shown in 5.13 in \cite{BT} that Proposition 2.1 in \cite{Pnbook} is not valid in the real case.
 Indeed a positive unital selfadjoint real linear map on a real operator system need not be bounded. 
  Nor are valid  certain exercises in that text related to Proposition 2.1.  A helpful example here is the 
 real span of a spin system $(U_i)$ (see e.g.\ p.\ 174 in \cite{Pisbk}). See also \cite{CDPR1, CDPR2}.
We have not checked if positive contractive maps on a real JC$^*$-algebra extend to unital positive maps on the unitization, however certainly completely positive contractive maps  do, as in \cite[Lemma 3.9]{CEcplp} and by complexification, and the extension is ucp.  However the vast majority of results are valid in the real case, either by the analogous proof, or by complexification.  For example Lemma 2.10 there is valid,
however we recall that the spectrum in a real algebra is defined in terms of the spectrum in the complexification \cite{Li}.  
 Proposition 2.11
is not true unless $\phi$ is selfadjoint, but if selfadjoint then the same proof works \cite{BT}.
 We also mention in particular (in addition to results mentioned in the papers in the  list above): Example 1.4, Lemma 2.3 and Theorem 2.4, and Exercises 2.3--2.5 and 2.10 are all valid in the real case with the same proof.  
 Results 2.8--2.9 are false even if $\phi$ is selfadjoint (consider selfmaps of the quaternions). Similarly for 3.1--3.4, and 3.7 and 3.14 (3.5, 3.6, 3.8 and 3.9 appear e.g.\ in \cite{BT}), and Exercises 3.2--3.4, and 3.18.
Theorem 3.11 is true with the same proof, however we have not checked if the result extends to positive maps on all commutative real C$^*$-algebras.  In Chapter 6, Theorems 6.1--6.4, 
and Lemma 6.5, follow with the same proof in the real case, 
as do later results 6.6--6.10, and Exercise 6.2 which these rely on.    In the latter exercise,  note that the restriction $r$ of $P$ to 
its range $K$ is invertible there, and we may take  $R = r^{-\frac{1}{2}}$ there, and 0 on $K^\perp$.  It has been noted by others that where Paulsen invokes
 Krein-Milman in 
 6.6, he means the geometric Hahn-Banach theorem applied in $M_n(\cS)_{\rm sa}$.    In our real case,
 apply the real geometric Hahn-Banach theorem in these places
  to obtain a separating functional $\varphi : M_n(\cS) \to \bR$, and then replace it by 
 $(\varphi + \varphi^*)/2$.  
Recall that we are assuming that all positive  maps and functionals are taken to be selfadjoint.  
We will improve on 6.7 later in the finite dimensional case (see Remark 3 after Theorem \ref{fdom}). 

\begin{corollary} \label{plem} Let $\cS$ be a real operator system and $x \in  M_n(\cS)$. If $u^{(n)}(x) \geq 0$
for every 
(selfadjoint) ucp $u : \cS \to M_k(\bR)$  and every $k \in \bN$, then $x \geq 0$.
If further $x = x^*$ then the above holds with $k = n$.
\end{corollary} 

\begin{proof}  Indeed the real case of the first assertion  follows by the same proof as in complex case (see 
\cite[Lemma 4.1]{KPTT1}), or follows  the from the complex case by complexifying.  Note that any complex linear ucp $u : \cS_c \to M_k(\bR)_c = M_k(\bC)$ gives a real linear  ucp $v = u_{| \cS} : \cS \to M_{2k}(\bR)$, so that $u^{(n)}(x) = v^{(n)}(x) \geq 0$.  If $x = x^*$ we use an idea from 
 p.\ 179 of \cite{Pnbook}.  Claim: If $x \notin M_n(\cS)^+$, then there exists a cp $\varphi: \cS \to M_n$ 
 such that $\varphi^{(n)}(x) \notin M_{n^2}^+$. 
Since $K = M_n(\cS)^+$ is a closed cone there exists a linear functional $\psi : M_n(\cS) \to \bR$ with $\psi(K) \geq 0 > \psi(x)$.  
We may assume that $\psi = \psi^*$, then $\psi$  is (completely) positive.   If the Claim were false then the associated completely 
positive map $\Psi : \cS  \to M_n$ (see the start of \cite[Chapter 6]{Pnbook}) satisfies $\Psi^{(n)}(x) \geq 0$.   This implies that 
$0 \leq \sum_{i,j} \, \Psi(x_{ij})_{ij} = \sum_{i,j} \, \psi(x_{ij} \otimes e_{ij}) = \psi(x)$ by e.g.\ the `Alternatively' on p.\ 73 in \cite{Pnbook}. 
This is a contradiction. 
\end{proof} 

That is, the matrix states determine the order.   

Relating to the Exercise set for Chapter 8, 8.6 works as expected.  Regarding Haagerup's facts about 
decomposable maps in the real case and similar settings, 
and some aspects of these are quite interesting.   For example $CB(A,B(H))$ 
is not the span of the completely positive maps, unlike in the complex case.  Nonetheless nearly all of the basic results in the theory of decomposable maps 
work in the real case (see the sequel \cite{BPr}). 
The main result in Chapter 13 is considered in our Section \ref{CE} in the real case. 

Finally, the basic injectivity results in Chapter   15 of \cite{Pnbook} are mostly verified in the real case in some of the papers in the list above.  Although  the real case of Theorem 15.12 and Corollary 15.13 follows by complexification, it is worth saying a few words about this proof   and supplying a few additional details.  First, we define an operator $A$-system in the real case exactly as it is defined in the complex case 
above \cite[Theorem 15.12]{Pnbook}, 
That is, $\cS$  is a real operator system and nondegenerate bimodule over a real C$^*$-algebra $A$, such that $a 1_{\cS} = 1_{\cS} a$  for $a \in A$,
and $a^* M_n(\cS)^+ a \subseteq M_n(\cS)^+$ for  all $a \in M_n(A)$. 
 
\begin{theorem} \label{roAsy}  Let $A$ be a real (resp.\ complex) unital C$^*$-algebra.
\begin{itemize} \item [{\rm (1)}] A real (resp.\ complex) operator system  and $A$-bimodule $\cS$ is a real (resp.\ complex) operator $A$-system if and only if it
 is a nondegenerate (so $1 x = x 1 = x$ for $x \in \cS$) operator $A$-bimodule in the sense of {\rm \cite[Chapter 3]{BLM}}  or 
 {\rm  \cite[Theorem 15.14]{Pnbook}}
(resp.\ {\rm \cite[Theorem 2.4]{BReal}}) and $a \, 1_{\cS}
= 1_{\cS} \, a$ for $a \in A$. 
\item [{\rm (2)}] For a  real operator $A$-system, if we view the
real operator system injective envelope $(I(\cS),j)$ as a real C$^*$-algebra  with $\cS$  unitally completely order embedded in it via $j$ as in e.g.\ {\rm \cite[Theorem 4.2(2)]{BCK}}, then
the canonical map $a \mapsto a 1_{\cS}$ is a $*$-homomorphism $\pi : A \to I(\cS)$ satisfying
$\pi(a) j(x) = j (ax)$ for $a \in A, x \in \cS$.  \item [{\rm (3)}] If $\cS$ is a  real operator $A$-system 
there is a real Hilbert space $H$ and  unital (selfadjoint) complete order  embedding $\gamma : \cS \to B(H)$,
and a unital  $*$-representation
 $\pi : A \to B(H)$  such that $\pi(a) \gamma(x) = \gamma(ax)$ for $a \in A, x \in \cS$. 
 \end{itemize} 
\end{theorem}
 
\begin{proof}  In both the real and complex case,
if $B = (I(\cS),j)$ is a real injective envelope of $\cS$ regarded as a C$^*$-algebra, with $j$ a unital complete order embedding, then by  \cite[Proposition 4.4.13]{BLM} (the real case of which was checked in \cite{BReal}), we have $I_{11}(\cS) = I_{22}(\cS) = I(\cS) = B$. 
 If $\cS$ is a nondegenerate operator $A$-bimodule then by
the characterization of operator $A$-bimodules in \cite[Chapter 4]{BLM} or  \cite[Theorem 15.14]{Pnbook}
 (or Section 4 and Theorem 2.4 in \cite{BReal} in the real case)
there exists $*$-homomorphisms $\pi : A \to B, \theta  : A \to B$ with $j(axb) = \pi(a) j(x) \theta(b)$ for  $a, b \in A, x \in \cS$.
Thus $$\pi(a) = \pi(a) j(1_{\cS}) = j(a 1_{\cS}) = j(1_{\cS} a) = j(1_{\cS}) \theta(a) = \theta(a).$$
So $\pi = \theta$ and also $$j((ax)^*) =j(ax)^*  = (\pi(a) j(x))^* = j(x)^* \pi(a^*) = j(x^* a^*).$$
Thus $(ax)^* = x^* a^*$ for $a \in A, x \in \cS$.    Clearly $j(X)$ is a concrete $\pi(A)$-system, so that
$X$ is an abstract $A$-system.   This  proves the harder direction of (1). 
 
For the second  statement, by results about the real injective envelope from e.g.\ \cite{BCK} it is easy to see that it suffices to show that $\cS_c$ is a complex operator $A_c$-system (to which we may apply \cite[Theorem 15.12]{Pnbook}). 
This in turn follows by a standard technique:\ identifying $A_c$ and $\cS_c$ with the corresponding sets of matrices of form
$$\begin{bmatrix} x & -y \\ y & x \end{bmatrix} .$$
If $a, b \in M_n(A), s , t \in M_n(\cS)$, and
$s \pm i t \geq 0$, or equivalently $$\begin{bmatrix} s & -t \\ t &  s \end{bmatrix} \geq 0,$$  it  follows that
$$\begin{bmatrix} a & -b \\ b &  a \end{bmatrix} \begin{bmatrix} s & -t \\ t &  s \end{bmatrix} \begin{bmatrix} a^* & b^* \\ -b^* &  a^* \end{bmatrix} \geq 0.$$ Computing this formal matrix product yields a matrix of form $$\begin{bmatrix} z & -w \\ w &  z \end{bmatrix} $$ in $M_{2n}(A)$.
Thus by the correspondence in (\ref{ofr}) we have that
$z-iw \geq 0$ in $M_n(A_c)$.  The reader can verify that this is precisely saying that $$(x+iy) (s+it) (x+iy)^* \geq 0$$ in $M_n(A_c)$.
So  $\cS_c$ is an operator $A_c$-system. 
 
The
third statement follows from the 
second by the same proof as in the complex case. Indeed in both the real and complex case, once we see that $j(\cS)$ is a concrete $\pi(A)$-subbimodule of   $B$, then it is an operator $A$-bimodule, and any faithful $*$-representation of $B$ on a Hilbert space $H$ will do the trick.   
This also proves the converse to (1) (the complex case is similar).  \end{proof}

{\bf Remarks.} 1)\ Any operator system (or unital operator space) $V$ has a well-defined and canonical partially defined 
product.  Indeed if we have an injective envelope $B = (I(V),j)$ which is a unital C$^*$-algebra and $j$ is unital,  define 
$x y = j^{-1}(j(x) j(y))$ for pairs $x, y \in V$ for which this product is defined (i.e.\ 
$j(x) j(y) \in j(V)$).  
This is well-defined independently of the particular 
injective envelope C$^*$-algebra chosen.   Indeed the latter well-definedness may be taken to be one formulation of the well known Banach-Stone theorem for unital complete isometries between unital operator spaces--namely that such isometries are the restriction of  $*$-isomorphisms between their injective or C$^*$-envelopes.  
 (The real case of this topic was discussed in  \cite{BReal}.  Indeed 
 the  unital complete isometry `extends to' a surjective unital complete isometry  
between the injective envelopes, by rigidity, whereupon apply the unital case of \cite[Theorem 4.4]{RComp}.)

With this in mind, operator $A$-systems contain a unital C$^*$-subalgebra $D$ with respect to  this canonical partial product, namely $j^{-1}(\pi(A))
= A \cdot 1$ in the notation of (2) of the theorem.
Moreover, operator $A$-systems 
are precisely the operator systems $\cS$ containing a unital C$^*$-subalgebra $D$ with respect to  this canonical partial product
such that there is a $*$-epimorphism $\rho : A \to D$
with $a x = \rho(a) x$ for all $a \in A, x \in \cS$, where the latter is the partial product.
To prove the more difficult direction of this, set $\rho(a) 
= j^{-1}(\pi(a)) = a \cdot 1$ in the notation above, and the rest is an exercise. 

Note that the partial product may be viewed as  `everywhere defined', but may not always be in $\cS$.

\medskip

2)\ We say that an operator $A$-system $\cS$ is {\em faithful} if the action of $A$ on $\cS$ is faithful (that is $a = 0$ if $a \cS = (0)$).   It is obvious, and has been observed by others, that this is equivalent to $\pi$ being one-to-one in (2) or (3).   Indeed this is a basic feature of operator bimodules over C$^*$-algebras, viewing $\cS$ as an operator $A$-bimodule as in (1) of the theorem.   In this case the C$^*$-algebra $D$ in the last remark is $*$-isomorphic to $A$, and we have:

\begin{corollary} \label{fasy}
    Faithful (real or complex) operator $A$-systems 
are precisely the (real or complex) operator systems $\cS$ containing a unital (real or complex) C$^*$-subalgebra $A$ with respect to  the canonical partial product of $\cS$ in the Remarks above.
That is, if and only if there is a unital subsystem $A$ of $\cS$ 
which is a (real or complex) C$^*$-algebra in the partial product, such that 
$A \cS \subseteq \cS$ with respect to the partial product.
\end{corollary}

\begin{proof}   The one direction follows from the Theorem and Remarks, taking $A$ to be the $D$ in the Remarks.  

For the other direction, inside $B = (I(V),j)$ we have that $j_{|A} : A \to B$ is a unital $*$-homomorphism, and 
$j(\cS)$ is a nondegenerate $j(A)$-submodule of $B$ and $j(a) j(1) = j(a) = j(1) j(a)$ for $a \in A$.   So $\cS$ is
an operator $A$-system by (1) of the Theorem, and it clearly is faithful. 
\end{proof}

\begin{theorem} \label{dual}  If $X$ is a real operator system (resp.\ real unital operator space) such that $X$ is a dual as a real  operator
space, then there exists a completely isometric weak* homeomorphism  from $X$
onto a weak* closed operator subsystem 
(resp.\ unital subspace) in $B(H)$ for a real Hilbert space $H$, which maps
the identity into the identity.  In the operator system case this weak* homeomorphism is also a complete order isomorphism.
\end{theorem} 

\begin{proof}  We need only prove the unital operator space case, since the system case follows from this.
We have that $X_c$ is a dual complex  operator
space as well as a complex operator system.  So by \cite{BMag} there exists a completely isometric complex linear weak* homeomorphism $u$ from $X_c$
onto a weak* closed unital subspace of a complex, hence real, von Neumann algebra. The restriction of $u$ to $X$ is our desired weak* homeomorphism. 
\end{proof}

We call the operator systems characterized in the last result (real)  {\em dual operator systems}.    
Of course a bidual operator system is a dual operator system, and we have a unital weak* homeomorphic complete order 
isomorphism $(V_c)^{**} \cong (V^{**})_c$.

\section{Complex structure} \label{cx}

We begin with a discussion of  the real operator space centralizer space $Z(V)$  in real operator systems.  Real operator space centralizers are defined and  developed in 
Sections 3 and  4 in \cite{BReal}.  We will not review their pretty theory here, or alternative definitions, but simply say that the proof of the next result shows that in our operator system case 
the 
algebra $Z(V)$ may be defined to be 
the operators $T : V \to V$  of left multiplication by elements of the 
$*$-subalgebra $Z$ of $I(V)$ of elements that commute with $V$ and multiply $V$ into itself.  
This is a commutative C$^*$-algebra and is isometrically a subalgebra of $B(V)$ and of CB$(V)$.  It is also identifiable with an operator subsystem of $V$.  Thus the main new feature in the operator system case is that centralizers may be considered to live inside $V$.  The complete $M$-projections are the projections $P$ in $Z(V)$, and the complete $M$-summands are the ranges $P(V)$ of complete $M$-projections.

We recall  that we discussed the characterization of  real  operator $A$-systems in 
Theorem \ref{roAsy}.

\begin{proposition} \label{asyz}   A real (resp.\  complex) operator system  $V$ is a faithful real (resp.\  complex) operator $A$-system with $A$ the  real (resp.\  complex) centralizer algebra $Z(V)$. 
Thus $Z(V)$ may be regarded as a C$^*$-subalgebra $Z$ of $V$ in the sense of Corollary {\rm \ref{fasy}} or the Remarks after Theorem {\rm \ref{roAsy}}.  Moreover in the real case, $Z(V_c) = Z(V)_c$. 
 \end{proposition}

\begin{proof}  We prove only the real case of the first statement, the complex being similar. If $B = (I(V),i)$ is the real injective envelope of $V$ regarded as a real C$^*$-algebra, with $i$ a unital complete order embedding, then by the real case of \cite[Proposition 4.4.13]{BLM}, which was checked in \cite{BReal}, we have $I_{11}(V) = I_{22}(V) = I(V)$.  By \cite[Section 4]{BReal}, the left adjointable maps on $V$ are thus describable in terms of the real injective envelope of $V$, as left multiplication by elements from 
$$D = \{ x \in B : x \, i(V) + x^*  \, i(V) \subseteq i(V) \}.$$  In particular $D\subseteq i(V)$ since $x = x i(1) \in i(V)$. Taking adjoints,  $D = \{ x \in B :  i(V) \, x + i(V) \, x^* \subseteq i(V) \}$, and the right adjointable maps on $V$ are right multiplication by elements from 
$D$.   Note that $D$ is a C$^*$-subalgebra of $B$.  If $T$ is both a left and  a right adjointable map on $V$ then there exist $z, y \in D$ with $$i(T(x)) = z \, i(x) = i(x) \, y, \qquad x \in V.$$  Setting $x = 1$ we have $z = y \in i(V)$.   Conversely any $z \in D$ with $z \, i(x) = i(x) \, z$ for all $x \in V$ comes from such a $T$. Thus
$Z(V)$ may be identified with the commutative C$^*$-subalgebra $Z = D \cap i(V)'$, where the latter is the commutant of $i(V)$ in $B$.   The canonical map $Z \to Z(V)$ is a $*$-isomorphism. It is now clear that $i(V)$ is  a concrete real  operator $Z$-system, and $V$ is a real  operator $Z(V)$-system.

The last statement was proved and used in \cite[Sections 3 and 4]{BReal}.
 \end{proof}

 \begin{corollary} \label{asyzc}  A  selfadjoint real centralizer $T$ of a real operator system  $V$ is selfadjoint as a map on $V$.  The complete  $M$-projections of $V$ 
 correspond to multiplication by a projection in the real commutative C$^*$-algebra $Z$.  The complete  $M$-summands of $V$
 are selfadjoint subspaces and are abstract real operator systems.
 \end{corollary}

 \begin{proof} In the language of the last proof  
 $T$ is multiplication by an element from 
$Z_{\rm sa}$, so that $$T(x^*)^* = (z \, i(x^*))^* =z \, i(x) = T(x), \qquad x \in V.$$ 

The complete $M$-projections of $V$ correspond to projections $p$ in $Z(V) \cong Z$.  The complete  $M$-summands of $V$ correspond to 
$p \cdot i(V)$ for such $p$.  Clearly $p \cdot i(V)$ is an operator subsystem of $pBp$, with identity $p$ of course. 
\end{proof}

In \cite{BCK,BReal} we have discussed real and complex structure in operator spaces.  {\em Real structure} in a space $X$ examines the deep question of when $X$ is a complexification, or the various possible real spaces (in our case, real operator systems) that have 
$X$ as a complexification.  Even for a C$^*$-algebra this is a very deep topic \cite{C}.  For example there are  no known uncomplicated  complex examples with no real structure.  We will not say more on this besides Proposition \ref{cco}.   {\em Complex structure} on a real operator space asks for example when a real  space is a complex space, and how this can occur.
Here we offer the following two theorems.  The first says complex structure exists if and only if there is a real linear antisymmetry 
$J : V \to V$  (antisymmetry means that $J^2 = -I$)  which is also a centralizer:

\begin{theorem}  \label{idee}   A  real operator system  $V$ is a complex operator system 
if and only if it possesses a 
real operator system  centralizer $z$ with $z^2 = -1$.  This is also equivalent to $V$ being real linearly completely isometric to a complex operator space.
\end{theorem}

\begin{proof} The first statement follows from the second and \cite[Corollary 3.2]{BReal}.  The second statement follows for example from \cite[Theorem 2.3]{BReal} and its earlier complex variant from \cite{BNmetric2}. 
Indeed the centered equation in the statement of \cite[Theorem 2.3]{BReal}, which is the same equation as in the complex case from \cite{BNmetric2}, has no reference to the scalar field being used, so implies that 
the complex operator space structure on $V$ is a
complex operator system structure.  \end{proof}

{\bf Remark.} It follows that a real operator system  $V$ has  complex  operator space structure if it satisfies the conditions in \cite[Theorem 3.1]{BReal}, or if the real  operator system bidual $V^{**}$ has complex  operator system structure (\cite[Corollary 3.3]{BReal}).  

\bigskip

A second question one may ask about complex structure is whether we can classify all of the scalar multiplications $\Cdb \times V \to V$ on a real operator system that make $V$ a  complex operator system?   Any complex operator system $V$ has a second complex operator system structure $\bar{V}$. 
If $V$ and $W$ are complex operator systems then $V \oplus^\infty \bar{W}$ is a  complex operator system.
This is just $V \oplus^\infty W$ with a `twisted'
complex multiplication given by $i \cdot (x,y) = (ix, -iy)$.  As in the operator space case \cite[Theorem 3.4]{BReal} the latter  {\em is the only} example of a complex structure on an
complex operator system besides the given one. 

\begin{theorem}  \label{coosstr}  A  complex operator system $V$  has an essentially unique complex operator system  structure up to `conjugates' of complete $M$-summands. 
Indeed for any other complex operator system  structure  on $V$ (with the same underlying real operator system  structure) there exists 
two complex (with respect to the first  complex structure) selfadjoint subspaces 
$X_+$ and $X_{-}$ of $V$ which are each abstract operator systems, with $V = X_+ \oplus^\infty X_{-}$ completely order isomorphically, 
such that the second complex multiplication of $\lambda \in \Cdb$ with 
$x_+ + x_{-}$ is $(\lambda \, x_+ ) + (\bar{\lambda} \, x_{-})$, for $x_+ \in X_+,  x_{-} \in X_{-}$.  
 \end{theorem} 

\begin{proof} Any other  complex operator system  structure  on $V$ is a  complex operator space  structure  on $V$. 
Thus by   \cite[Theorem 3.4]{BReal}  there exists 
two complex (with respect to the first  complex structure) subspaces 
$X_+$ and $X_{-}$ of $X$  with $X = X_+ \oplus^\infty X_{-}$ 
such that the second complex multiplication is as stated. 
Since $X_+$ and $X_{-}$  are complete $M$-summands, Corollary \ref{asyzc} informs us that they are selfadjoint subspaces which are abstract real operator systems.
  Indeed in the language of the proof of  Corollary \ref{asyzc} and the theorem preceding  that, we have  $i(V) = p \cdot i(V) + (1-p) \cdot i(V)$, which is
completely order isomorphic to the operator system sum  $p \cdot i(V) \oplus^\infty (1-p) i(V)$.  This shows that  $X = X_+ \oplus^\infty X_{-}$ completely order isomorphically.
\end{proof}

{\bf Remark.} In \cite[Theorem 3.4]{BReal} 
it is stated that every complex operator space $X$ has a unique complex operator space structure up to complete isometry.  This is somewhat misleading, since `real complete isometry' was meant there, which is tautological unless one looks at the specific form given for these complete isometries. The result is probably best stated similarly to the system case above.  

In a probabilistic sense, most operator spaces have no nontrivial complete $M$-summands.  They would have 
unique complex operator space structure up to complete isometry:

 \begin{corollary} \label{nntm} A complex operator space (resp.\ system) with no nontrivial complete complex $M$-summands 
 has unique complex operator space  (resp.\ system) structure up to complex linear complete isometry (resp.\ unital complete 
 order isomorphism). 
 \end{corollary}  

{\bf Remark.} We recall that in a complex operator space $X$ the complete complex $M$-summands  are the same as the 
complete real $M$-summands of $X_r$.  This may be deduced from the fact that $Z_{\bC}(X) = Z_{\bR}(X_r)$ in 
\cite[Theorem 4.5]{BReal}. 

\section{C*-covers} \label{csc}

As mentioned earlier, the basic theory of the injective and C$^*$-envelopes of real operator systems is already in the literature (e.g.\ \cite{Sharma, BCK,BReal}).  
The maximal or universal C$^*$-cover $C^*_{\rm u}(\cS)$ is designed to be a C$^*$-algebra generated by a copy of $\cS$, with the universal property that any ucp map
$T : \cS \to A$ into a C$^*$-algebra extends to a complex $*$-homomorphism $\pi : C^*_{\rm u}(\cS) \to A$ with
$\pi = T$ on the copy of $\cS$.  Such an object is clearly unique.  
Indeed one may define the maximal (or universal) C$^*$-cover $C^*_{\rm u}(\cS)$ of a real operator system $\cS$ to be the real C$^*$-algebra $B$ generated by the copy of $\cS$ inside its complex  universal C$^*$-cover $C^*_{\rm u}(\cS_c)$.   Any ucp map
$T : \cS \to A$ extends to a complex $*$-homomorphism $C^*_{\rm u}(\cS_c) \to A_c$ making the obvious diagram commute, which in turn restricts to a real  $*$-homomorphism $\rho : B \to A_c$ such that $\rho$ `agrees with' $T$ on $\cS$.
By continuity and density this actually maps into $A$ since $T$ does.  Thus $B$ has the desired universal property above.

\begin{lemma} \label{ucom} For a real operator system $\cS$ we have
   $C^*_{\rm u}(\cS_c) = (C^*_{\rm u}(\cS))_c$.
\end{lemma}

\begin{proof} The canonical period 2 unital completely isometric conjugate linear map 
$\theta_{\cS} : \cS_c \to \cS_c$ extends to a period 2 conjugate linear $*$-automorphism $\theta$ of $C^*_{\rm u}(\cS_c)$ that fixes $B$.  
To see this, we will use the operator space structure 
$\bar{X}$ associated with a complex operator space $X$.  This is the set of formal symbols $\bar{x}$ for $x \in X$, with the operator space structure making $X \to \bar{X} : x \mapsto \bar{x}$ a  conjugate linear complete isometry.  In our case $X$ is a unital complex operator system, namely 
$\cS_c$ for a real operator system $\cS$, 
in which case the latter map is unital. 
Apply the universal property of $C^*_{\rm u}(\cS_c)$ to the composition of the latter map with $\theta_{\cS}$.  This composition is a complex linear complete order 
isomorphism $\cS_c \to \overline{\cS_c}$.  We obtain a $*$-isomorphism $C^*_{\rm u}(\cS_c) \to C^*_{\rm u}(\overline{\cS_c})$.
However it is easy to see  by a diagram chase that $C^*_{\rm u}(\overline{\cS_c})$ may be identified with the `conjugate C$^*$-algebra' 
$\overline{C^*_{\rm u}(\cS_c)}$.   The last $*$-isomorphism then yields a 
conjugate linear $*$-automorphism $\theta$ of $C^*_{\rm u}(\cS_c)$.
It is easy to check that this fixes $\cS$, hence also fixes $B$, and is period 2. 

We will be done if $B \cap iB = (0)$ and $B + iB = 
C^*_{\rm u}(\cS_c)$.  This will be obvious already to some readers, but for completeness we give one version of the argument.
Let $D$ be the fixed points of $\theta$, and $E$ the points with $\theta(x) = -x$.  These are closed with 
 $D \cap E = (0)$ and $D + E = C^*_{\rm u}(\cS_c)$. 
 Also $B \subseteq D$ and $iB \subseteq E$.  If $b_n, c_n \in B$ with 
 $b_n + ic_n \to x$ then by applying $\theta$ we see that 
 $b_n - ic_n \to \theta(x)$.  It follows that $x + \theta(x) \in B$ and $x - \theta(x) \in iB$, so that $x \in B + iB$.  Thus $B + iB$ is closed, hence is a complex C$^*$-subalgebra of
 $C^*_{\rm u}(\cS_c)$ containing 
 $\cS_c$.    Thus $B \oplus iB = 
C^*_{\rm u}(\cS_c)$.
\end{proof}

The next result is probably mostly folklore (partially used in \cite{BK} for example).  

\begin{proposition} \label{gpalg} 
    Let $G$ be a locally compact group.  The complexification of the 
    real full (resp.\ reduced) group C$^*$-algebra of $G$ 
    is $*$-isomorphic to the 
    complex full (resp.\ reduced) group C$^*$-algebra of $G$.   Also $G$ is amenable if and only if the two real algebras here coincide: 
    $C^*_{\Rdb}(G) = (C^*_{\lambda})^{\Rdb}(G)$.
\end{proposition}

\begin{proof}  See \cite{B} for  details.  
For example a modification of some ideas in the proof of 
Lemma \ref{ucom} shows that $C^*_{\Rdb}(G)_c = C^*_{\Cdb}(G)$.  It is well known that that $G$ is amenable if and only if 
 $C^*_{\Cdb}(G) = (C^*_{\lambda})^{\Cdb}(G)$.  By the relation  $C^*_{\Rdb}(G)_c = C^*_{\Cdb}(G)$ and the analogous (easier, since we can use  the action on $L^2(G,\bC) = L^2(G,\bR) \oplus_2 L^2(G,\bR)$) fact for the reduced group C$^*$-algebra, the identity in the last sentence holds 
  if and only   $C^*_{\Rdb}(G) = (C^*_{\lambda})^{\Rdb}(G)$.
\end{proof}

In a later work we plan to discuss boundary representations and their relation with real matrix convex sets and the real Choquet boundary, for example.

    \begin{theorem} \label{ucomla} The (canonical equivalence classes of) real (resp.\ complex) C$^*$-covers of a real  (resp.\ complex) operator system $\cS$ form a lattice, indeed a topology.  
    In the real case it may be identified with a complete sublattice of the complete lattice or topology of (equivalence classes of) complex C$^*$-covers of $\cS_c$.  
\end{theorem}

\begin{proof}
    Just as in the complex operator algebra case 
    discussed in the Notes on Chapter 2.4 in \cite{BLM}, the partially ordered set  of
(equivalence classes of) C$^*$-covers of an operator system is a lattice.  (This was later rediscovered in the operator algebra case by others.) Indeed the  C$^*$-covers of $\cS$
correspond, as in the Notes just referred to, exactly to the lattice $L$ of closed ideals
in the kernel  $J$ of the canonical 
$*$-homomorphism $C^*_{\rm u}(\cS) \to C^*_{\rm e}(\cS)$.  Similarly, the complex C$^*$-covers of $\cS_c$
correspond to the lattice $\hat{L}$ of closed complex ideals
in the kernel $\hat{J}$ of the
$*$-homomorphism $C^*_{\rm u}(\cS_c) \to C^*_{\rm e}(\cS_c)$.
Since $C^*_{\rm u}(\cS_c) = C^*_{\rm u}(\cS)_c$ by Lemma \ref{ucom}, 
and $C^*_{\rm e}(\cS_c) = C^*_{\rm e}(\cS)_c$, it follows that 
$\hat{J} = J_c = J + iJ$.  Moreover $K \mapsto K_c$
injects the lattice $L$ into the lattice $\hat{L}$ preserving infs and sups. 

 It is  known (see the Notes on Chapter 2.4 in \cite{BLM}) that the closed complex ideals of a $C^*$-algebra,
 and in particular $\hat{L}$, is lattice isomorphic to the set of open sets
in a certain (generally non Hausdorff) topological space.  Namely the open sets in the {\em spectrum}
of the C$^*$-algebra,   in particular of $\hat{J}$ above. We recall that the spectrum 
is the set of the equivalence classes of irreducible
representations of $J$.  This is probably true in the real case     too, following from some of the points below, and may be in the literature, but  we have not checked this.  Instead we give a novel proof that is of independent interest (both in the real and complex case) 
 that the closed ideals form  a topology.
 
 It seems to be a strange omission in the books on the subject that it is not said explicitly that a complete lattice is a topology if and only if 
 the prime elements in the lattice order-generate the lattice.   Here we recall that $p$ is prime if $p \neq 1$ and $a \wedge  b \leq p$ implies 
 $a$ or $b$ is $\leq p$; and $P$ order-generates means that every $x$ is the infimum of $\{ p \in P : x \leq p \}$. 
 We are indebted to a 2018 Stackexchange post of Eric Wofsey which pointed this out together with a proof. The fact is sometimes attributed to B\"uchi; it is referred to for example with some references on p.\ 392 in \cite{XY}.   
 
 The set of closed ideals in a C$^*$-algebra is certainly a complete lattice.  
 In the complex case it is well known that the closed prime ideals order-generate all the ideals in the sense above--see  II.6.5.3 and II.6.5.2 in \cite{Bla}.  That is, every closed ideal is an intersection of primitive ideals (i.e.\ kernels of irreducible representations).
 The real case of this latter fact may be known but we offer a proof.   Or rather we checked that 
 the standard proofs in \cite{Bla} leading to  II.6.5.3 and II.6.5.2 there, work in the real case except for the proof of 6.4.8 there.
 The latter instead needs \cite[Proposition 5.3.7 (2)]{Li}.  Indeed    we will 
need all parts of  \cite[Proposition 5.3.7]{Li}. We also need \cite[5.1.6 (3)]{Li} (norm equals spectral radius on normal elements, which follows from the definitions and the complex case)  and \cite[5.3.5 (2) and (3)]{Li} (in particular that any selfadjoint element achieves its norm at a pure state).  Note that we only need selfadjoint elements since an element $x$ is annihilated by every irreducible representation if and only if the same is true for $x^* x$.
Finally, II.6.1.11 in \cite{Bla} (in the irreducible representation case) needs \cite[Proposition 5.3.7 (1)]{Li}.  \end{proof}

{\bf Remarks.} 1)\ Some open questions about the relations between pure states on $A_c$ and pure states on $A$, for a real C$^*$-algebra $A$,
are stated on p.\ 91 in \cite{Li}.  A negative solution with $A = M_2(\bR)$ was given in \cite{BMcII}.

\smallskip

2)\ Note that closed complex ideals in $J_c$ need not all be of form $K_c$ for a closed ideal  $K$ in $J$.   For example consider $J$ the real C$^*$-algebra $$\left\{f\in C(\bT,\bC): \tau (f) (t) = \overline{f(\overline{t})} \ (t\in \bT) \right\},$$  on the unit circle $\bT$.   
Hence there does not seem to be a canonical bijection between the complex C$^*$-covers of $\cS_c$ and the real C$^*$-covers of $\cS$. 

\smallskip

3)\ The results here would apply immediately to unital operator spaces $X$, via Arveson's $\cS = X + X^*$ trick. 

\section{Archimedeanization} \label{archs}

In this section and Section \ref{nmin}, we consider various ways to modify a given real matrix ordering $\{C_n\}$ on a real $*$-vector space $V$, as well as how these modifications interact with the complexification.

We make note of two well-known facts for complex operator systems which which may fail to hold in the real case.

\begin{proposition} \label{2enf}
    Suppose that $V$ is a (real or complex) $*$-vector space, $\{C_n\}$ is a matrix ordering on $V$, $e \in V_{\rm sa}$, and $I_2 \otimes e$ is an order unit for $M_2(V)_{\rm sa}$. Then $e$ is a matrix order unit.
\end{proposition}

\begin{proof}
    Let $[a_{ij}] \in M_n(V)_{\rm sa}$. Fix $1 \leq i<j \leq n$. Consider the $2 \times 2$ matrix \[ A(i,j) := \begin{bmatrix} a_{ii} & a_{ij} \\ a_{ji} & a_{jj} \end{bmatrix} \in M_2(V)_{\rm sa}. \]
    Choose $t_{ij} > 0$ such that $A(i,j) + t_{ij} I_2 \otimes e \geq 0$. Then $A(i,j) \oplus 0_{n-2} + t_{ij} I_n \otimes e \geq 0$.  Conjugating by an appropriate permutation matrix, we get that $M(i,j) := \sum_{k,l \in \{i,j\}} E_{kl} \otimes a_{kl} + t_{ij} I_n \otimes e \geq 0$ (where $E_{ij}$ are the canonical matrix units). Adding these matrices for all $1 \leq i < j \leq n$, we get
    \[ \sum_{i=1}^n \sum_{j > i} M_{ij} = [a_{ij}] + {\rm diag}(r_1 a_{11}, \dots, r_n a_{nn}) + \sum t_{ij} I_n \otimes e \geq 0 , \] where $r_1, \dots, r_n \geq 0$ account for the repeated occurrences of $a_{11}, \dots, a_{nn}$ on the diagonal in the sum. Since $a_{ii} \in V_{\rm sa}$ for each $i$, there exists $t > 0$ such that $tI_n \otimes e - {\rm diag}(r_1 a_1, \dots, r_n a_n) \geq 0$. Adding this positive term to the sum above, we get $[a_{ij}]+ (t+ \sum t_{ij}) I_n \otimes e \geq 0$.
\end{proof}

In the complex case, it suffices for $e$ to be an order unit.  This was shown in \cite{CEinj}, but for the readers convenience we give a quick alternative proof using the previous result.

\begin{proposition}
    Suppose that $V$ is a complex $*$-vector space, $\{C_n\}$ is a matrix ordering, and $e \in V_{\rm sa}$ is an order unit for $(V_{\rm sa}, C_1)$. Then $e$ is a matrix order unit.
\end{proposition}

\begin{proof}
    It suffices to show that $I_2 \otimes e$ is an order unit for $(M_2(V)_{\rm sa}, C_2)$. Suppose that \[ \begin{bmatrix} a & b \\ b^* & c \end{bmatrix} \in M_2(V)_{\rm sa}. \] Choose $t > 0$ such that $a + te, c + te \geq 0$. Then ${\rm diag}(a,c) + t I_2 \otimes e \geq 0$. Also, notice that we have \[ \begin{bmatrix} 0 & b \\ b^* & 0 \end{bmatrix} = \begin{bmatrix} 1 & 1 \\ 1 & 1 \end{bmatrix} \otimes \Re(b) + \begin{bmatrix} 1 & i \\ -i & 1 \end{bmatrix} \otimes \Im(b) + \begin{bmatrix} 1 & 0 \\ 0 & 1 \end{bmatrix} \otimes (-\Re(b)) + \begin{bmatrix} 1 & 0 \\ 0 & 1 \end{bmatrix} \otimes (-\Im(b)). \] Now observe that if $A \in M_2(\mathbb{R})^+$ and $x \in V_{\rm sa}$, $pI_2 \geq A$, and $x + qe \geq 0$, then $$A \otimes x + pq I_2 \otimes e \geq A \otimes x + A \otimes (qe) = A \otimes (x + qe) \geq 0.$$ Thus, in the expression above we can find $t_1, t_2, t_3, t_4 \geq 0$ such that \[ \begin{bmatrix} 0 & b \\ b^* & 0 \end{bmatrix} + (t_1 + t_2 + t_3 + t_4) I_2 \otimes e \geq 0. \] This implies \[ \begin{bmatrix} a & b \\ b^* & c \end{bmatrix} + (t+t_1+t_2+t_3+t_4)I_2 \otimes e \geq 0. \] So $I_2 \otimes e$ is an order unit and therefore $e$ is a matrix order unit.
\end{proof}

We will see an example soon where $e$ is an order unit for a real matrix ordered $*$-vector space but fails to be a matrix order unit. We now turn our attention to the matrix ordering.

\begin{proposition} \label{prop: C2 proper}
    Let $V$ be a real or complex matrix ordered $*$-vector space and suppose that $C_2$ is proper. Then $\{C_n\}$ is proper for all $n$.
\end{proposition}

\begin{proof}
    If $\pm x \in C_1$, then $\pm x \oplus 0  \in C_2$ and  thus $x=0$. So $C_1$ is proper.

    Next we check that $C_n$ is proper for any $n > 2$. Suppose that $n > 2$ and $\pm \begin{bmatrix} x_{ij} \end{bmatrix} \in C_n$. Conjugating by an appropriate real matrix, we see that \[ \pm \begin{bmatrix} x_{ii} & x_{ij} \\ x_{ji} & x_{jj} \end{bmatrix} \in C_2. \] It follows that $x_{ii}=x_{ij}=x_{ji}=x_{jj}=0$. So $\begin{bmatrix} x_{ij} \end{bmatrix} = 0$. Hence $C_n$ is proper.
\end{proof}

In the complex case, it suffices for $C_1$ to be proper. 

\begin{proposition}
    Let $V$ be a complex $*$-vector space and $\{C_n\}$ a matrix ordering. Suppose that $C_k$ is proper for some $k$. Then $C_n$ is proper for all $n$.
\end{proposition}

\begin{proof}
    We will show that $C_k$ proper implies $C_1$ is proper, and that $C_1$ proper implies that $C_n$ is proper for all $n$. 
    
    If $\pm x \in C_1$ then $\pm x \oplus 0_{k-1} \in C_k$. So $x = 0$, since $C_k$ is proper.

    Now assume $C_1$ is proper. Suppose that $\pm [a_{ij}] \in C_n$. By compressing to the diagonal entries, we see that $a_{ii} = 0$ for each $i$. Compressing to a $2 \times 2$ submatrix, we see that \[ \pm \begin{bmatrix} 0 & a_{ij} \\ a_{ij}^* & 0 \end{bmatrix} \geq 0. \] Conjugating this matrix by $2^{-1/2} \begin{bmatrix} 1 & 1 \end{bmatrix}$ and by $2^{-1/2} \begin{bmatrix} 1 & i \end{bmatrix}$, we see that $\pm \Re(a_{ij}) \geq 0$ and $\pm \Im(a_{ij}) \geq 0$ and therefore $a_{ij} = 0$. So $C_n$ is proper.
\end{proof}

Suppose that $\{D_n\}$ is a (real or complex) matrix ordering on a (real or complex) $*$-vector space $\cS$, with matrix order unit 
 $e \in \cS$. The \textit{archimedean closure} of $\{D_n\}$ is defined to be the sequence of cones $C_n := \{x \in M_n(\cS)_{\rm sa} : x + \epsilon e_n \in D_n \; \forall \epsilon > 0 \}$. In the complex case, $\{C_n\}$ is known to be the smallest matrix ordering such that $D_n \subseteq C_n$ and such that the unit is 
 archimedean (see e.g.\ \cite{PTT}).  The same holds in the real case.

\begin{proposition} \label{Prop: Arch closure of matrix ordering}
    Let $\{D_n\}$ be a real matrix ordering on a $*$-vector space $V$ and let $e$ be an 
        order unit (resp.\ matrix order unit) for $V$. Then $C_n := \{x \in M_n(\cS)_{\rm sa} : x + \epsilon \, e_n \in D_n  \; \forall \epsilon > 0 \}$ is the smallest matrix ordering for which $e$ is a matricially archimedean unit 
             (resp.\ 
             matricially archimedean  
             matrix order unit) such that $D_n \subseteq C_n$. 
\end{proposition}

\begin{proof}
    The proof that $\{C_n\}$ is a matrix ordering with matrix  archimedean  unit $e$ is the same as in the complex case:
        Suppose that $x \in C_n$ and $y \in C_m$. Let $t > 0$. Then $x + t e_n \in D_n$ and $y + t e_m \in D_m$. Since $\{D_k\}$ is  a matrix ordering, $x \oplus y + t e_{n+m} = (x + t e_n) \oplus (y + t e_m) \in D_{n+m}$. This holds for all $t > 0$ and hence $x \oplus y \in D_{n+m}$.
Now let $\alpha \in M_{n,k}$. WLOG, assume $\alpha \neq 0$, and let $\beta = \alpha^* \alpha$. Let $t > 0$. Then $x + (t/\|\alpha\|^2) e_n \in D_n$. Since $\{D_m\}$ is a matrix ordering, $\alpha^*(x + t/\|\alpha\|^2) e_n) \alpha = \alpha^* x \alpha + t (\beta / \|\alpha\|^2) \otimes e_k \in D_k$. Since $I_k - \beta / \|\beta\| \geq 0$ and $e_k \geq 0$, $e_k - (\beta / \|\alpha\|^2) \otimes e_k \in D_k$ and hence $\alpha^* x \alpha + t e_k \in D_k$. Therefore $\alpha^* x \alpha \in C_k$.  We conclude that $\{C_n\}$ is a matrix ordering.

That   $e$ is a matricially archimedean unit for 
 $\{C_n\}$  is obvious. If $e$ happens to be a matrix order unit for $\{D_n\}$, then it is necessarily a matrix order unit for $\{C_n\}$. That is because if $x=x^*$, there exists $R > 0$ such that $x + R e_n \in D_n \subseteq C_n$.
    Suppose $\{C_n'\}$ is another matrix ordering for which $e$ is a matricially archimedean unit 
    and such that $D_n \subseteq C_n'$ for every $n$. Let $x \in C_n$. Then for every $\epsilon > 0$, $x + \epsilon e_n \in D_n \subseteq C_n'$. It follows that $x \in C_n'$ since  $e$ is a matricially archimedean unit for $C_n$. Thus $C_n \subseteq C_n'
    $ for all $n$.
\end{proof}

{\bf Remarks.}\ 1)\  If the $D_n$ are proper cones and $D_k$ is archimedean for some $k \geq 2$, then $C_n$ is proper for all $n$.
Indeed in this case $D_2$ is archimedean, and so $D_2 = C_2,$ so that $C_n$ is proper by Proposition \ref{prop: C2 proper}.  

2)\  If $e$ is not an order unit but only a positive non-zero vector, the archimedean closure  with respect to $e$ in Proposition \ref{Prop: Arch closure of matrix ordering} is still a matrix ordering and is the smallest matrix ordering  for which $e$ is a matricially archimedean unit, by the same proof.

For future use we  define an element of $V$ to be {\em matricially archimedean}  if $x \in M_n(W)_{\rm sa}$ and $x + t Q_n \geq 0$ for all $t >0$, then $x \in M_n(W)^+$.

\bigskip

Suppose that $\{D_n\}$ is a real matrix ordering and let $\{C_n\}$ denote its archimedean closure. Let \[ N = \{x \in V : \begin{bmatrix} 0 & x \\ x^* & 0 \end{bmatrix} \in C_2 \} \] (see Section 17 of \cite{Oz} for motivation of this definition). If $C_2$ is proper, then $N = \{0\}$ and, by Proposition \ref{prop: C2 proper}, $\{C_n\}$ is a proper matrix ordering making $V$ into an operator system 
(if it or $\{D_n\}$ has a matrix order unit). Otherwise, we will show that $V/N$ can be made into an operator system, called the \textit{archimedeanization} of $V$, which the same universal properties as in the complex case.

\begin{lemma} \label{lemma: N=C cap -C}
    The subspace $N \subseteq V$ defined above is a self-adjoint real subspace. Moreover, $M_2(N)_{\rm sa} = C_2 \cap -C_2$ and consequently $M_n(N)_{\rm sa} = C_n \cap -C_n$ for every $n \in \mathbb{N}$.
\end{lemma}

\begin{proof}
    Since $C_2$ is a cone, it is easy to see that if $x,y \in N$ and $t > 0$, then $tx \in N$ and $x+y \in N$. Also, because
    \[ \begin{bmatrix} 0 & x \\ x^* & 0 \end{bmatrix} \in C_2 \quad \text{implies} \begin{bmatrix} 1 & 0 \\ 0 & -1 \end{bmatrix} \begin{bmatrix} 0 & x \\ x^* & 0 \end{bmatrix} \begin{bmatrix} 1 & 0 \\ 0 & -1 \end{bmatrix} = \begin{bmatrix} 0 & -x \\ -x^* & 0 \end{bmatrix} \in C_2 \] we see that $x \in N$ implies $-x \in N$. It now follows that $N$ is a real subspace of $V$. To see it is self-adjoint, notice \[ \begin{bmatrix} 0 & x \\ x^* & 0 \end{bmatrix} \in C_2 \quad \text{implies} \begin{bmatrix} 0 & 1 \\ 1 & 0 \end{bmatrix} \begin{bmatrix} 0 & x \\ x^* & 0 \end{bmatrix} \begin{bmatrix} 0 & 1 \\ 1 & 0 \end{bmatrix} = \begin{bmatrix} 0 & x^* \\ x & 0 \end{bmatrix} \in C_2. \] So $x \in N$ implies $x^* \in N$ and $N$ is self-adjoint.

    For the next claim, assume $[x_{ij}] \in M_2(N)_{\rm sa}$. Then since $\pm x_{ii} \in N$, we see that $\pm x_{ii} \in C_1$, since \[ \pm x_{ii} = \frac{1}{2} \begin{bmatrix} 1 & 1 \end{bmatrix} \begin{bmatrix} 0 & \pm x_{ii} \\ \pm x_{ii} & 0 \end{bmatrix} \begin{bmatrix} 1 \\ 1 \end{bmatrix} \in C_1. \] Together with the fact that $\pm x_{ij} \in N$, we see that \[ \pm \begin{bmatrix} x_{11} & 0 \\ 0 & x_{22} \end{bmatrix} \pm \begin{bmatrix} 0 & x_{ij} \\ x_{ij}^* & 0 \end{bmatrix} = \pm [x_{ij}] \in C_2. \] On the other hand, suppose $\pm [x_{ij}] \in C_2$. Repeating tricks from the first paragraph, we see \[ \begin{bmatrix} x_{11} & x_{12} \\ x_{21} & x_{22} \end{bmatrix} + \begin{bmatrix} -x_{11} & x_{12} \\ x_{21} & -x_{22} \end{bmatrix} = \begin{bmatrix} 0 & x_{12} \\ x_{21} & 0 \end{bmatrix} \in C_2. \] Therefore $x_{12} \in N$ and $x_{21}=x_{12}^* \in N$. Also, compressing $[x_{ij}]$ gives $\pm x_{ii} \in C_1$. Hence
    \[ \frac{1}{2} \begin{bmatrix} 1 & 1 \\ 1 & 1 \end{bmatrix} \otimes x_{ii} + \frac{1}{2} \begin{bmatrix} 1 & -1 \\ -1 & 1 \end{bmatrix} \otimes (-x_{ii}) = \begin{bmatrix} 0 & x_{ii} \\ x_{ii} & 0 \end{bmatrix} \in C_2. \] Thus $x_{ii} \in N$, so $[x_{ij}] \in M_2(N)_{\rm sa}$. Therefore $M_2(N)_{\rm sa} = C_2 \cap -C_2$.

    For the final claim, let $[x_{ij}] \in M_n(N)_{\rm sa}$. Then each $2 \times 2$ principal submatrix is positive and negative, i.e. for $i \neq j$, \[ X(i,j) := \begin{bmatrix} x_{ii} & x_{ij} \\ x_{ji} & x_{jj} \end{bmatrix} \in C_2 \cap -C_2. \] Thus for each fixed $i<j$, $\sum_{k,l \in \{i,j\}} E_{kl} \otimes x_{kl} \in C_n \cap -C_n$. Summing over all such matrices, we see that $[x_{ij}] + \text{diag}(r_1 x_{11}, \dots, r_n x_{nn}) \in C_n \cap -C_n$ where $r_1, \dots, r_n \geq 0$ account for repeated occurrences of $x_{ii}$ in the sum. But $x_{ii} \in C_1 \cap -C_1$ for each $i$ and thus $\text{diag}(-r_1 x_{11}, \dots, -r_n x_{nn}) \in C_n \cap -C_n$. Adding this to the sum above shows that $[x_{ij}] \in C_n \cap -C_n$. 
    We leave the converse inclusion as an exercise, looking at $2 \times 2$ principal submatrices. 
\end{proof}

As in  the complex case, we define the archimedeanization of a real operator system to be the $*$-vector space $V/N$ with order unit $e + N$ and matrix ordering $\{C_n'\}$ to be the image of $\{ C_n \}$.
Note that the preimage of $C_n'$ in $M_n(V)_{\rm sa}$ is
$C_n + M_n(N)_{\rm sa} = \{ p + x : p \in C_n, x \in M_n(N)_{\rm sa}\}$.

\begin{theorem} \label{archr}
    Let $V$ be a real $*$-vector space with matrix ordering $\{D_n\}$ and matrix order unit $e$, and let $\{C_n\}$ denote the archimedean closure  of $\{ D_n \}$. With $N$ and $\{C_n'\}$ defined as above. Then $(V/N, \{C_n'\}, e+N)$ is an operator system. Moreover, if $W$ is an operator system and $\varphi: V \to W$ satisfies $\varphi^{(n)}(D_n) \subseteq M_n(W)^+$ for every $n \in \mathbb{N}$, then there exists a unique ucp map $\phi': V/N \to W$ such that $\varphi'(x + N) = \varphi(x)$ for every $x \in V$.
\end{theorem}

The universal property of the archimedeanization described in the theorem can be summarized in the following diagram, where $q: V \to V/N$ is the map $q(x) = x+N$:

\[
\begin{tikzcd}
V/N \arrow[dashrightarrow, rd, "\varphi'" black] \\
 V \arrow[u, "q"] \arrow[r, "\varphi" black] & W
\end{tikzcd}
\]

\begin{proof}
    Proving that $(V/N, \{C_n'\}, e+N)$ is an operator system  is as in the complex case except for showing that the ordering $\{C_n'\}$ is proper (since $N$ is defined differently in the real case).
    For instance, to show that the matrix ordering is archimedean, suppose that $x \in M_n(V)_{\rm sa}$ and $x + \epsilon I_n + M_n(N) \in C_n'$ for every $\epsilon > 0$. So for every $\epsilon > 0$ there exists $y \in M_n(N)_{\rm sa}$ such that $x + \epsilon I_n + y \in C_n$. Since $M_n(N)_{\rm sa} = C_n \cap -C_n$ by Lemma \ref{lemma: N=C cap -C}, it follows that $-y \in C_n$ and hence $x + \epsilon I_n \in C_n$. Since $e$ is a matricially archimedean unit 
for $C_n$  by Proposition \ref{Prop: Arch closure of matrix ordering}, $x \in C_n$ and thus $x+M_n(N) \in C_n'$.  
    
    To see that $C_n'$ is proper, suppose that $\pm X + M_2(N) \in C_2'$. Then there exist $P, Q \in C_2$ and $Y, Z \in M_2(N)_{\rm sa}$ such that $X = P + Y$ and $-X = Q + Z$. Adding these, we see $0=P+Q + Y + Z$, so $P+Q \in M_2(N)_{\rm sa}$. By Lemma \ref{lemma: N=C cap -C}, $P+Q \in C_2 \cap -C_2$. Thus $-P-Q \geq 0$, so $-P = (-P-Q) + Q \geq 0$ and $-Q = (-P-Q) + P \geq 0$. Since $\pm P, \pm Q \in C_2$, by Lemma \ref{lemma: N=C cap -C} again $P,Q \in M_2(N)_{\rm sa}$. So $X \in M_2(N)$ and therefore $C_n'$ is proper. The universal properties follow as in the complex case, after checking that if $x \in N$, then $x \in \ker(\varphi)$ for every ucp map $\varphi$ on $V$.
\end{proof}

We conclude by relating the real archimedeanization to the complex one described in \cite{PTT} and \cite{PT}. The construction of the archimedeanization of a complex matrix-ordered $*$-vector space is essentially the same as the real case described above, except that in the complex case $N = \text{span} \; C_1 \cap -C_1$. (Alternatively $N$ is the intersection of the kernels of all states on $V$. The equivalence of these notions is shown in Proposition 2.34 of \cite{PT}.) This is equivalent to our definition. To see this, first let $x \in C_1 \cap -C_1$. Then \[ \frac{1}{2} \begin{bmatrix} 1 & 1 \\ 1 & 1 \end{bmatrix} \otimes x + \frac{1}{2}\begin{bmatrix} 1 & -1 \\ -1 & 1 \end{bmatrix} \otimes (-x) = \begin{bmatrix} 0 & x \\ x & 0 \end{bmatrix} \in C_2 \cap -C_2 = M_2(N)_{\rm sa} \] since the proof of Lemma \ref{lemma: N=C cap -C} is valid in the complex case, 
with the additional observation that $x \in N$ implies $ix \in N$, since 
\[ \begin{bmatrix} 0 & x \\ x^* & 0 \end{bmatrix} \in C_2 \quad \text{implies} \begin{bmatrix} 1 & 0 \\ 0 & -i \end{bmatrix} \begin{bmatrix} 0 & x \\ x^* & 0 \end{bmatrix} \begin{bmatrix} 1 & 0 \\ 0 & i \end{bmatrix} = \begin{bmatrix} 0 & ix \\ -ix^* & 0 \end{bmatrix} \in C_2 \] 
so that $N$ is a complex subspace. Thus $i x \in N$. On the other hand, if $x \in N$, then \[ \Re(x) = \frac{1}{2} \begin{bmatrix} 1 & 1 \end{bmatrix} \begin{bmatrix} 0 & x \\ x^* & 0 \end{bmatrix} \begin{bmatrix} 1 \\ 1 \end{bmatrix}, \quad \Im(x) = \frac{1}{2} \begin{bmatrix} 1 & -i \end{bmatrix} \begin{bmatrix} 0 & x \\ x^* & 0 \end{bmatrix} \begin{bmatrix} 1 \\ i \end{bmatrix} \in C_1. \] Conjugating by $1 \oplus -1$ and repeating the above, we also have $-\Re(x), -\Im(x) \in C_1$. So $x \in {\rm span} \; C_1 \cap -C_1$.  

The complex archimedeanization of a complex operator system $V$ is then constructed by equipping the vector space $V/N$ with the matrix ordering $C_n + M_n(N)$, where $C_n$ is again the archimedean closure of the given matrix ordering.

\begin{proposition}
Let $V$ be a real $*$-vector space with matrix ordering $\{D_n\}$ and matrix order unit $e$, and let $V_c$ denote its complexification. Let $N_c$ denote the complex span of $N$ in $V_c$. Then $V_c / N_c$ is the complex archimedeanization of $V_c$, and $V_c / N_c \cong (V/N)_c$.
\end{proposition}

\begin{proof}
    If one begins with a real matrix ordered space $(V,\{D_n\}, e)$, the complexification $(V_c, \{\widetilde{D}_n\}, e)$ is a complex matrix ordered space. The archimedean closure of the cones is given by $\widetilde{C}_n = \{ x : x + \epsilon e \otimes I_n \in \widetilde{D}_n \text{ for all } \epsilon > 0\}$. It is clear then that the archimedean closure of $C_n$ is precisely $M_n(V) \cap \widetilde{C}_n$. Hence, if we let $\hat{N} = \widetilde{C}_1 \cap - \widetilde{C}_1$, then \[ \hat{N} \cap V = \{ x \in V : \begin{bmatrix} 0 & x \\ x^* & 0 \end{bmatrix} \in \widetilde{C}_2 \cap -\widetilde{C}_2  \} = N. \] So $\hat{N} =  N + iN = N_c$ inside $V_c$. So $V_c / N_c$ equipped with the matrix ordering $\widetilde{C}_n + N_c$ and unit $e + N_c$ is the archimedeanization of the complexification $V_c$. Therefore $V_c / N_c$ is the complex archimedeanization described in \cite{PTT}. Let $q: V \to V/N$ denote the map $x \mapsto x + N$. Then $q_c: V_c \to (V/N)_c$ is ucp. Since $V/N$ is an operator system, so is $(V/N)_c$. Let $q'_c$ be the induced map $V_c/N_c \to (V/N)_c$.  By the universal property of the complex archimedianization, the diagram 
    \[
    \begin{tikzcd}
    V_c/N_c \arrow[dashrightarrow, rd, "q_c'" black] \\
    V \arrow[u, "q"] \arrow[r, "q_c" black] & (V/N)_c
    \end{tikzcd}
    \]
    commutes. 
    Thus the map from $V_c/N_c$ to $(V/N)_c$ is completely positive.

    Conversely, let $W = V_c/N_c$, a complex operator system, and let $\varphi: V_c \to W$ be the canonical ucp quotient map. Let $\varphi_r: V \to W$ denote the restriction to $V$. Then viewing $W$ as a real operator system, $\varphi_r$ is ucp. It follows that $\varphi_r': x+N \mapsto \varphi(x)$ is ucp.  This extends uniquely by Lemma \ref{ncomma} (or rather, by the proof of that result, which works even if the domain is just a matrix ordered space) to a ucp complex linear map $(V/N)_c \to W$. 
    Thus the identity map from $(V/N)_c$ to $V_c/N_c$ is ucp. 
\end{proof}

\section{Quotients} \label{quot}

Let $J \subseteq V$ be the kernel of a ucp map on an real operator system $V$. Then a matrix ordering $\{Q_n\}$ on $V/J$ makes $(V/J, \{Q_n\}, e+J)$ into a \textit{quotient operator system} if whenever $\varphi:V \to W$ is ucp with $J \subseteq \ker(\varphi)$ then there exists a unique ucp map $\pi_{\varphi}: V/J \to W$ such that $\varphi = \pi_{\varphi} \circ q$, where $q: V \to V/J$ is the quotient map $q(x)=x+J$. Kernels are of course order ideals. Kavruk 
et al.\ show that whenever $J$ is the kernel of a ucp map, then the quotient operator system exists and is unique. The same argument holds for real operator spaces.

\begin{proposition} \label{kuq} 
    Let $J \subseteq V$ be the kernel of a ucp map. Then there exists a unique matrix ordering $\{Q_n\}$ on $V/J$ making $(V/J, \{Q_n\}, e+J)$ a quotient operator system.
\end{proposition}

\begin{proof}
    Define $Q_n$ to be the set of all cosets $x + M_n(J)$ such that whenever $\psi: V \to W$ is ucp and $J \subseteq \ker(\psi)$, then $\psi^{(n)}(x) \geq 0$. Then $\{Q_n\}$ is a matrix ordering since $x + M_n(J) \in Q_n$, $y + M_k(J) \in Q_k$, and $\alpha \in M_{n,k}(\mathbb{R})$ implies $\psi^{(n+m)}(x \oplus y) \geq 0$ and $\psi^{(k)}(\alpha^* x \alpha) = \alpha^* \psi^{(n)}(x) \alpha \geq 0$. Choose ucp $\varphi: V \to W$ so that $\ker(\varphi)=J$. Then if $\pm x \in Q_1$, it follows that $\pm\varphi(x) \geq 0$ and hence $x \in J$. By such considerations we see that
     $Q_n$ is proper. To verify that $e + J$ is an archimedean matrix order unit, assume $x=x^* \in M_n(V)$. Then there exists $t > 0$ such that $x + tI_n \otimes e \geq 0$ and hence $\psi^{(n)}(x + tI_n) \geq 0$ for any ucp $\psi$. If $\psi^{(n)}(x + \epsilon I_n) = \psi^{(n)}(x) + \epsilon I_n \geq 0$ for every ucp $\psi$, then $\psi^{(n)}(x) \geq 0$ since each $\psi$ ranges in an operator system. Hence $x + M_n(J) \in Q_n$. So $e + J$ is archimedean. The uniqueness of $\{Q_n\}$ follows from the universal property. 
\end{proof}

Kavruk et al.\ show that when $V$ is a complex operator space, then (selfadjoint) $x + M_n(J) \in Q_n$ if and only if for every $\epsilon > 0$ there exists $y \in M_n(J)_{\rm sa}$ such that $x+y + \epsilon I_n \geq 0$. The same proof applies in the real case, so the matrix ordering is defined the same way for real operator systems.

\begin{lemma}
    Let $V$ be a real operator system and $J \subseteq V$ the kernel of a ucp map. Suppose that $\{Q_n\}$ makes $(V/J, \{Q_n\}, e+J)$ into a quotient operator system. Then $x + M_n(J) \in Q_n$ if and only if for every $\epsilon > 0$ there exists $y \in M_n(J)$ such that $x+y + \epsilon I_n \geq 0$.
\end{lemma}

\begin{proof}
     Let $\{Q_n'\}$ be defined as the set of cosets $x + M_n(J)$ such that for every $\epsilon > 0$ there exists $y \in M_n(J)$ such that $x+y + \epsilon I_n \geq 0$. Then $\{Q_n'\}$ is easily seen to be a matrix ordering on $V/J$ with archimedean matrix order unit $e + J$. Now suppose that $\pm x \in Q_n'$. Then if $J = \ker(\varphi)$, we have that for every $\epsilon > 0$ there exists $y,z \in M_n(J)$ such that $x + y + \epsilon I_n \otimes e \geq 0$ and $-x + z + \epsilon I_n \otimes e \geq 0$. Hence $\varphi^{(n)}(x) + \epsilon I_n, -\varphi^{(n)}(x) + \epsilon I_n \geq 0$. It follows that 
      $x \in M_n(J)$, so $\{Q_n'\}$ is proper and hence $(V/J, \{Q_n'\}, e+J)$ is an operator system. By the universal property of the quotient, we see that $Q_n \subseteq Q_n'$ since the identity map from $(V/J, \{Q_n\}, e+J)$ to $(V/J, \{Q'_n\}, e+J)$ is completely positive.

    Now let $x \in Q_n'$. Let $\psi: V \to B(H)$ be ucp with $J \subseteq \ker(\psi)$ and let $\epsilon > 0$. Then there exists $y \in M_n(J)$ such that $x + y + \epsilon I_n \otimes e \geq 0$. Since $\psi$ is ucp, $\psi^{(n)}(x + y + I_n \otimes e) = \psi^{(n)}(x) + \epsilon I_n \geq 0$. It follows that $\psi^{(n)}(x) \geq 0$. Since this holds for every $\psi$ with $J \subseteq \ker(\psi)$, $x \in Q_n$. Hence $Q_n' = Q_n$.
\end{proof}

{\bf Remark.} Kavruk 
et al.\ prove that every kernel is an intersection of hyperplanes, i.e. $J = \cap_i \ker(\varphi_i)$ where $\varphi_i: V \to \mathbb{C}$ are states with $J \subseteq \ker(\varphi_i)$. This does not hold in the real case. Often there are not enough states into $\mathbb{R}$ to distinguish kernels.

Indeed let $\Phi : \bH \to \bH$ be the identity map.  This is ucp.  Then Ker $\Phi =0$. 
But any state on $\bH$ contain the span of $i,j,k$ in their kernel, since any positive selfadjoint functional on $\bH$ annihilates this span. So the trivial kernel $\{0\}$ cannot be expressed as an intersection of kernels of states $\varphi: \bH \to \bR$.  (This example may be easily modified by taking a direct sum if one prefers a nontrivial kernel.)

\bigskip

It is non-trivial to decide if a given subspace $J \subseteq V$ is the kernel of a ucp map. Using the characterization of $Q_n$ above, we can give the following intrinsic characterization for kernels in operator systems.

\begin{proposition}
    Let $V$ be a complex operator system and $J \subseteq V$ be a subspace. Then $J$ is the kernel of a ucp map if and only if whenever  there exist sequences $\{y_n\}, \{z_n\} \subseteq J$ and $x=x^* \in V$ such that $x + y_n + \frac{1}{n} I \geq 0$ and $-x + z_n + \frac{1}{n} I \geq 0$
    for each $n \in \bN$, then $x \in J$. If $V$ is instead a real operator system, then $J$ is the kernel of a ucp map if and only if whenever  there exist sequences $\{y_n\}, \{z_n\} \subseteq M_2(J)$ and $x=x^* \in M_2(V)$ such that $x + y_n + \frac{1}{n} I_2 \otimes e \geq 0$ and $-x + z_n + \frac{1}{n} I_n \otimes e \geq 0$, then $x \in M_2(J)$.
\end{proposition}

\begin{proof}
    It is easily checked that $Q_n'$ above is a matrix ordering with an archimedean matrix order unit. It remains to check that $Q_n'$ is proper. In the real case, it is sufficient to check that $Q_2'$ is proper, while in the complex case it is sufficient to check that $Q_1'$ is proper. In the complex case, suppose $\pm x \in Q_1'$. Then by the definition of $Q_1'$, for every $n \in \mathbb{N}$  there exist $y_n, z_n \in J$ such that $x + y_n + \frac{1}{n} I, -x + z_n + \frac{1}{n}I \geq 0$. Assuming the condition in the statement above holds, then $x \in J$ and therefore $x + J = J$, so $Q_1'$ is proper and hence $\{Q_n'\}$ is proper. It follows that the map $\pi: V \to (V/J, \{Q_n'\}, e + J)$ given by $x \mapsto x + J$ is ucp with kernel $J$. The case for real operator systems is similar.

    On the other hand, suppose $J$ is the kernel of a map $\varphi$. Then if there exist sequences $y_n, z_n \in J$ (or $M_2(J)$) such that $x + y_n + \frac{1}{n} I, -x + z_n + \frac{1}{n}I \geq 0$, then applying $\varphi$ (or $\varphi^{(2)}$ in the real case) we see that $\varphi(x)=0$ (or $\varphi^{(2)}(x)=0$ in the real case). It follows that $x \in J$ (respectively, $x \in M_2(J)$), so $\{Q_n\}$ is proper.
\end{proof}

\begin{proposition}
    Let $V$ be a real operator system. Then $J \subseteq V$ is a kernel in $V$ if and only if $J_c$ is a kernel in $V_c$.
\end{proposition}

\begin{proof}
    First assume $J$ is the kernel of some ucp map $\varphi: V \to W$. Then $J_c$ is the kernel of $\varphi_c: V_c \to W_c$. So $J_c$ is a kernel.

    Now assume $J_c$ is the kernel of a ucp map on $V_c$. Suppose there exist sequences $\{y_n, z_n\} \subseteq M_2(J)$ and $x=x^* \in M_2(V)$ such that $x + y_n + \frac{1}{n} I \geq 0$ and $-x + z_n + \frac{1}{n} I \geq 0$. Then, since $x=x^* \in M_2(V_c)$ and $\{y_n\}, \{z_n\} \subseteq M_2(J_c)$, we conclude that $x \in M_2(J_c)$ (as in the proof of the previous proposition, this is checked by applying $\varphi^{(2)}$ where $J_c = \ker(\varphi)$). But since $x \in M_2(V)$ and $V \cap J_c = J$, we have $x \in M_2(J)$. So $J$ is a kernel by the previous proposition. 
\end{proof}

\begin{proposition} \label{coker}
    If $V$ is a real operator system with kernel $J$, then $(V/J)_c \cong V_c / J_c$. 
\end{proposition}

\begin{proof}
    Assume $V$ is a real operator system with kernel $J$. Then $J_c$ is a kernel in $V_c$. Suppose that $(x+iy) + M_n(J_c) \in M_n((V/J)_c)^+$, where $x,y \in M_n(V)$. Without loss of generality, we may assume $x=x^*$ and $y=-y^*$ by fixing an appropriate representative for $x+iy + M_n(J_c)$. By the definition of complexification, \[ \begin{bmatrix} x & -y \\ y & x \end{bmatrix} + M_{2n}(J) \in M_{2n}(V/J)^+. \] By the Lemma, for every $\epsilon > 0$ there exists $\begin{bmatrix} z_{ij} \end{bmatrix} \in M_{2n}(J)$ such that \[ \begin{bmatrix} x & -y \\ y & x \end{bmatrix} + \begin{bmatrix} z_{ij} \end{bmatrix} + \epsilon I_{2n} \in M_{2n}(V)^+ \subseteq M_{2n}(V_c)^+. \] Conjugating by the matrix $\begin{bmatrix} 1 & i \end{bmatrix}$ and applying $\varphi^{(n)}$, where $\varphi: V_c \to B(H)$ is ucp with $J_c \subseteq \ker(\varphi)$, we see that $\varphi^{(n)}(x + iy) \geq 0$. It follows that $x + iy \in M_{n}(V_c/J_c)^+$ by the definition of the operator system quotient. So the identity map from $(V/J)_c$ to $V_c/J_c$ is ucp.

    For the other direction, let $\pi: V \to V/J$ denote the quotient map. Then $\pi_c: V_c \to (V/J)_c$ is ucp with kernel $J_c$. By the universal property of the quotient, the identity map from $V_c/J_c$ to $(V/J)_c$ is ucp. We conclude that $(V/J)_c \cong V_c/J_c$.
\end{proof}

Theorem 3.7 and Corollaries 3.8 and 3.9 of \cite{KPTT2} are valid in the real case with the same proofs.  In particular $\cS/J \subseteq C^*_u(\cS)/ \langle J \rangle$.  Also,
Proposition 3.6 (the universal property of the quotient):

\begin{proposition} \label{qup}  Let $V$ and $W$ be real operator systems, and let $J$ be a  kernel in $V$.  Every ucp map $u : V \to B(H)$, for a real Hilbert space $H$, with $J \subseteq {\rm Ker}(u)$, 
induces  a ucp map $\tilde{u} : V /J \to B(H)$ with $\tilde{u} \circ q_J = u$.  Moreover suppose that $\cS$ is a 
real operator system and $\eta : V \to \cS$ is a ucp map with the property: for every ucp
$u : V \to \cT$ into an operator system $\cT$ with $J \subseteq {\rm Ker}(u)$, there exists a unique  ucp  map $\hat{u} : \cS \to \cT$ such that $\hat{u} \circ \eta = u$, then there is a unital complete order isomorphism $\gamma : \cS \to V/J$ with $\gamma \circ \eta = q_J$. 
\end{proposition} 

\begin{proof} The proof is the same as in \cite[Proposition 3.6]{KPTT2}.   \end{proof} 

Lemma 4.3 in  \cite{KPTT2} is valid in the real case, so has the same applications. 
If we complexify then there is a ucp map on $\cS_c$ with kernel $\bC y$. The restriction $u$ to
$\cS$ contains $\bR y$ in its kernel.  However if $u(x) = 0$ then $x \in \bC y \cap \cS = \bR y$.

Recall that for a (real or complex) operator space $V$ with subspace $J \subseteq V$, the \textit{operator space quotient} is the operator space structure on the algebraic quotient $V/J$ defined by the matrix norms \[ \|x + M_n(J)\|_n = \inf \{ \|x + y \| : y \in M_n(J) \}. \] For operator spaces, is is known that $(V/J)_c = V_c / J_c$ \cite{Sharma}.  A subspace $J \subseteq V$ is called {\em completely proximinal} if for every $x \in M_n(V)$ there exists $y \in M_n(J)$ such that $\|x + M_n(J)\| = \|x + y\|$.

Suppose that $V$ is an operator system and $J \subseteq V$ is a kernel. Let $(V/J)_{osp}$ and $(V/J)_{osy}$ denote the operator space and operator system quotients, respectively. In general, the identity map from $(V/J)_{osp}$ to $(V/J)_{osy}$ is completely contractive. This follows from the universal property of the operator space quotient: Since the operator system quotient map $\pi: V \to (V/J)_{osy}$ is completely contractive with kernel $J$, it factors completely contractively through the operator space quotient, and the induced complete contraction from $(V/J)_{osp}$ to $(V/J)_{osy}$ is the identity map on the vector space $V/J$. The identity is not completely isometric in general, but there are special cases where it is known to be completely isometric in the complex case.

\begin{definition} \label{cobip} 
    A kernel $J \subseteq V$ of a (real or complex) operator system is called {\em completely biproximinal} if
    \begin{enumerate}
        \item The identity map from $(V/J)_{osp}$ to $(V/J)_{osy}$ is completely isometric,
        \item $J$ is completely proximinal in the operator space $V$,
        \item $J$ is {\em completely order proximinal}, i.e. for every $x + M_n(J) \geq 0$ there exists $p \in M_n(V)^+$ such that $x+M_n(J) = p+M_n(J)$, and
        \item Whenever $x + M_n(J) \geq 0$, there exists $p \in M_n(V)^+$ such that $x+M_n(J) = p+M_n(J)$ and $\|x+M_n(J)\| = \|p\|$.
    \end{enumerate}
\end{definition}

The following is proven in the complex case in \cite{KPTT2} (Proposition 4.8). The proof in the real case is identical.

\begin{proposition} \label{neolp}
    Let $J$ be a completely proximinal kernel in a real operator system $V$ and assume that $(V/J)_{osp}$ and $(V/J)_{osy}$ are completely isometrically isomorphic. Then $J$ is completely biproximinal.
\end{proposition}

\begin{corollary} \label{proxiff}
    Let $J$ be a kernel in a real operator system $V$. Then $J$ is completely biproximinal if and only if $J_c$ is completely biproximinal in $V_c$.
\end{corollary}

\begin{proof}
    We have seen that $J$ is a kernel in $V$ if and only if $J_c$ is a kernel in $V_c$. If $J$ is completely proximinal in $V$, then for any $x + iy \in M_n(V_c)$ we can choose $\begin{bmatrix} z_{ij} \end{bmatrix} \in M_{2n}(J)$ such that
    \begin{eqnarray} 
        \| x + iy + M_n(J) \| & = & \| \begin{bmatrix} x & -y \\ y & x \end{bmatrix} + M_{2n}(J) \| \nonumber \\
        & = & \| \begin{bmatrix} x & -y \\ y & x \end{bmatrix} + \begin{bmatrix} z_{ij} \end{bmatrix} \| \nonumber \\
        & \geq & \| \frac{1}{2} \begin{bmatrix} I_n & iI_n \end{bmatrix} \left( \begin{bmatrix} x & -y \\ y & x \end{bmatrix} + \begin{bmatrix} z_{ij} \end{bmatrix} \right) \begin{bmatrix} I_n \\ -iI_n \end{bmatrix} \| \nonumber \\
        & = & \| x+iy + z \|, \text{ where } z = \frac{1}{2} \begin{bmatrix} I_n & iI_n \end{bmatrix} \begin{bmatrix} z_{ij} \end{bmatrix} \begin{bmatrix} I_n \\ -iI_n \end{bmatrix}. \nonumber
    \end{eqnarray}
    Therefore $J_c$ is completely proximinal in $V_c$.
    
    Conversely, suppose that $J_c$ is completely proximinal in $V_c$. Let $x \in M_n(V)$. Choose $y, z \in M_n(J)$ such that $\|x + M_n(J_c)\| = \|x + y + iz\|$.  The latter is $\geq \| x + y \|$ by a fact from the introduction.
    It follows that $\|x+ M_n(J_c)\| = \|x + y\|$. Since $(V/J)_{osp} \subseteq (V_c/J_c)_{osp}$ completely isometrically, we conclude that $J$ is completely proximinal in $V$.

    Since $((V/J)_{osp})_c = (V_c/J_c)_{osp}$ and $((V/J)_{osy})_c = (V_c/J_c)_{osy}$, we see that the operator space and operator system quotients $V/J$ are completely isometric if and only if the same is true of $V_c/J_c$. By the previous Proposition, we conclude that $J$ is completely biproximinal in $V$ if and only if $J_c$ is completely biproximinal in $V_c$.
\end{proof}

\section{Duality} \label{Duality}

   A certain care is needed when considering the matrix order structure on $\cS^d$.   Indeed in modern works the duality pairing $\sum_{i,j} \, \varphi_{ij}(x_{ij})$ between  $M_n(\cS)$ and $M_n(\cS^d)$ is
    used, which is different from the `trace pairing' $\sum_{i,j} \, \varphi_{ij}(x_{ji})$ considered in older works such as the seminal papers  \cite{CEcplp, CEinj} of Choi and Effros.

    Choi and Effros prove in \cite[Corollary 4.5]{CEinj} that if
    $\varphi$ is a faithful state on a finite dimensional operator system $\cS$ then $(\cS^d,\varphi)$ is an operator system.   This holds in the real case too.  
    That $\varphi$ is an order unit follows from the fact that $\varphi$ achieves its  minimum $\delta > 0$ on the set of norm 1 positive elements  in $\cS$.  
    If $x$ is such an element then for any state $\psi$ we have 
    $$\psi(x) \leq 1 \leq \frac{1}{\delta} \varphi(x).$$   
    Any selfadjoint functional $\psi$ is dominated by a positive multiple of a 
    state. Indeed as we said in Section \ref{rcast}, $\psi$ is the difference of two (completely) positive functionals, so is dominated by one (or the sum) of these. In particular we see that $\varphi$ is an order unit.  
    That it is a  matrix order unit is similar.  Indeed as we said in Section \ref{rcast}, a selfadjoint
    $\Psi : \cS \to B(H)$ is the difference of two completely positive maps, 
    so is dominated by one (or the sum) of these.  There exists a positive constant $C$ such that for $x$ a norm 1 positive element  in $M_n(\cS)$ we have $\varphi(\sum_k \, x_{kk}) \geq C$.  For another  positive constant $K$ we have 
$$\sum_{i,j} \, \psi_{ij}(x_{ij}) \leq K  \leq \frac{K}{C} \, \sum_k \, \varphi(x_{kk}) =  \frac{K}{C} \, \varphi(\sum_k \, x_{kk}).$$ 
Thus $\Psi + c \varphi_n \geq  0$ for a positive constant $c$, where $\varphi_n = {\rm diag} (\varphi , \cdots ,\varphi)$.   Thus $\cS^d$ is a real operator system by our earlier  real 
    Choi and Effros  characterization theorem \ref{ce}. 
    
    Conversely, 
    any  order unit for $\cS^d$ is 
    faithful (it dominates (a positive multiple of) any faithful state).
    
    It is even easier to see that a faithful state  $\varphi$ is (matricially) archimedean (hint: consider $t \varphi^{(n)}(x) + \psi(x) \geq 0$ for fixed positive $x \in C_n$).

    \begin{theorem} \label{dcom} Let  $\cS$ be a finite dimensional real operator system with faithful state $\varphi$.  Then $(\cS^d,\varphi)$ is a real operator system, $((\cS_c)^d,\varphi_c)$ is a complex operator system, and $$(\cS_c)^d \cong (\cS^d)_c$$ unitally complete order isomorphically.
\end{theorem}

\begin{proof} Choose any faithful state $\psi$ on $\cS_c$.   The real part $g$ of its restriction to $\cS$ is faithful: if $x \in \cS^+$ satisfies 
    $${\rm Re} \, \psi(x) =\psi(x) =0 $$
    then $x=0$.   Note that $g_c$ is a faithful state.   Indeed if $x+iy \geq 0$ for $x,y \in \cS$, and $g_c(x+iy) =0$, then $x=0$ since $g$  is faithful.   However since $x \pm iy \geq 0$ this implies $y=0$.   So $((\cS_c)^d,g_c)$ is an operator system.  
    
    Next, $(\cS^d, g)$ is embedded as an operator subsystem of $((\cS_c)^d,g_c)$.  The map $\varphi \mapsto \varphi_c$ maps $(\cS^d, g)$ into $((\cS_c)^d,g_c)$.   This is a complete order embedding.   
 Indeed $\varphi : \cS \to M_n(\bR)$ is completely positive 
 if and only if  $\varphi_c : \cS_c \to M_n(\bR)_c = M_n(\bC)$ is completely positive. So (one may also appeal to the facts above the theorem) 
 $\cS^d$ is a real operator system. 

The period 2 conjugate linear complete order automorphism  $\theta_{\cS}$ dualizes to a 
period 2 conjugate linear complete order automorphism 
of the operator system $(\cS_c)^d$. 
We have to adjust this map slightly. 
Define $\rho(\psi) = \theta_{\bR} \circ \psi \circ \theta_{\cS}$.
 We have $$\rho(\varphi_c) (x+iy)
 = \theta_{\bR} (\varphi_c (x-iy))
 = \varphi_c (x+iy).$$ So 
 $\rho$ 
 fixes 
 the canonical copy of $\cS^d$.   
 Conversely if $\rho(\psi) = \psi$
 then $$\psi(x+iy) = \theta_{\bR} (\psi (x-iy)), \qquad x,y \in \cS.$$  Setting $y = 0$ shows that $\psi(\cS) \subseteq \bR$.
 If $\varphi = \psi_{| \cS}$ then $$\psi(i y) 
 = - \theta_{\bR} (i \varphi(y))=  i \varphi(y).$$
 So $\psi = \varphi_c$. It follows from Proposition \ref{cco} that $(\cS_c)^d \cong (\cS^d)_c$ complete order isomorphically. 
 \end{proof}

{\bf Remark.} Regarding the duality of subspaces and quotients: By the same proof as in the complex case (see \cite{FP} or \cite{Kavrukthesis}) if $u : \cS \to \cT$ is an operator system quotient map (that is, it induces a unital complete order isomorphism on $\cS / {\rm Ker} \, u$), then $u^d$ is a complete order embedding.
    Thus if $J$ is a kernel in $\cS$ then $(\cS/ J)^d$ is  completely order isomorphic to the subspace $J^\perp$ of $\cS^d$.
    If $\cS$ is finite dimensional then  $(\cS/  J)^d$ is an operator system.  However to show that $(\cS/  J)^d$ is an operator subsystem of  $\cS^d$ requires more, namely that $J$ is a {\em null subspace}, that is it contains no positive elements 
    other than 0.  See the proof of Proposition 2.7 in \cite{Kavruk}, which works in the real case. 
    
    For the converse, if $\cS$ is finite dimensional and $\cS_0 \subseteq \cS$ is a subsystem then $(\cS_0)^d \cong \cS^d/ \cS_0^\perp$ as operator systems. See  \cite{FP} or \cite[Theorem 2.8]{Kavruk} for a proof which works in the real case.

\bigskip

{\bf Example.} The quaternions $\bH$ are certainly not selfdual as an operator space.  (Indeed $M_2^* =(H_c)^* = (H^*)_c$. This would be $M_2$ completely isometrically if $\bH$ were selfdual as an operator space.  This is a contradiction.)  But the cone (the matrix ordering) is selfdual.   Thus the new  matrix order norm on the dual gives the quaternions back again. 

To see this let $\varphi$ be the unique state on $\bH$, namely the functional picking out the coefficient of $1$.   Now $\bH_c = M_2(\bC)$, and it is easy to 
 see that $\varphi_c$ becomes the (normalized) trace on $M_2(\bC)$.

\begin{proposition} \label{hh}  On the  operator system dual $(\bH^d, \varphi)$ of the quaternions, the order unit norm on $\bH^d$ induced by the order unit $\varphi$ 
coincides with the Banach space dual norm (that is, 
the usual norm of $\bH^* = B(\bH,\bR)$). 
Also $(\bH^d, \varphi)$
is unitally isometrically completely order isomorphic to  $(\bH , 1)$.
    \end{proposition}

\begin{proof} We use the standard embedding $\bH \subseteq M_2(\bC)$, and we will deduce the result from the analogous result for 
    $M_2(\Cdb)$ from Theorem 6.2 of \cite{PTT}.  (There ought to be a more direct argument that explicitly computes the completely positive maps from $\bH$ to $M_n$.) 
    
    Note that $$(\bH^d)_c \cong (\bH_c)^d \cong (M_2(\Cdb))^d,$$ and with respect to this correspondence we claim that the
    complete order isomorphism $\gamma : M_2(\bC) \to (M_2(\Cdb))^d$ from Theorem 6.2 of \cite{PTT} maps the real subspace $\bH$ onto
    the copy $E$ of $\bH^d$ in $(M_2(\Cdb))^d$.    Note that $\gamma(e_{ij})$
    is the `evaluation at the $i$-$j$-entry' functional, thus $\gamma$ is simply implementing the  duality pairing  mentioned at the start of Section \ref{Duality} on $M_2$.  This duality pairing on $\bH$ (the latter viewed as matrices) gives a map $\bH \to \bH^*$.  It is an exercise to check that the  complexification of this map is identifiable with $\gamma$.

    To see that $\bH$ is isometrically order isomorphic to its dual, notice that the map $\bH \to \bH^*$ in the proof above, divided by 2, is unital, and is given by the usual pairing $(a,b) \mapsto \frac{1}{2} \, \sum_{i,j} a_{i,j} b_{i,j}$.
However this is essentially the real Hilbert space inner product on $\bH$ (with some minus signs). Of course real Hilbert space is selfdual, and the  minus signs do not effect the conclusion.  
\end{proof}

The isometric phenomena here does not happen in the complex theory.   This has been observed by several authors (e.g. see \cite[Proposition 3.1]{PaulsenNg}), but for the readers convenience we supply a quick and simple  proof. 

\begin{theorem}
    Let $V$ be a complex finite-dimensional operator system with $\dim(V) > 1$. Then there does not exist an isometric order embedding from $V$ to $V^d$, the latter with its Banach dual norm.
\end{theorem}

\begin{proof}
    Suppose $\pi: V \to V^d$ is isometric and an order embedding. Now choose $x \in V_{\rm sa}$ neither positive nor `negative' with $\|x\|=1$. This is possible if $\dim(V)>1$. Indeed, if we represent $V$ as operators on $B(H)$ and if we choose $a \in V$ such that $a = a^*$ but $a$ is not a multiple of the identity, then the spectrum of $a$ contains at least two real numbers, say $r_1, r_2 \in \sigma(a)$. Then $a - \frac{1}{2}(r_1 + r_2)I$ has spectrum including both negative and positive values. Therefore $x=a - \frac{1}{2}(r_1 + r_2)I$ is neither positive nor negative. Rescaling, we may assume $\|x\|=1$.
    
    Let $\pi(x) = \psi \in V^d$. Then $\psi$ is self-adjoint but not positive nor `negative', since $\pi$ is a complete order embedding. As we said in the Introduction, there exist positive selfadjoint linear functionals $\psi_+, \psi_-$ such that $\psi = \psi_+ - \psi_-$ and $\|\psi\| = \|\psi_+\| + \|\psi_-\|$. If $\pi$ is isometric, then $\|\psi\| = 1$ and thus $\|\psi_+\|, \|\psi_-\| < 1$. Now $x = x_+ - x_-$ where $x_+ = \pi^{-1}(\psi_+)$ and $x_- = \pi^{-1}(\psi_-)$. So \[ -\|x_-\|I \leq -x_- \leq x \leq x_+ \leq \|x_+\|I, \] and thus $\|x\| \leq \max \{ \|x_+\|, \|x_-\| \} < 1$. This contradicts $\|x\|=1$.
\end{proof}

(See also e.g.\ \cite[Proposition 1.1]{Bknw} for a simpler related result.)

The bidual of a real operator system is always a real operator system.  Indeed if $i : \cS \to A$ is
the inclusion into a real unital C$^*$-algebra, then $i^{**} : \cS^{**} \to A^{**}$ is a 
unital complete isometry, hence a complete order embedding, into a real $W^*$-algebra.

\section{Minimal and maximal operator system structures} \label{nmin}

A \textit{complex AOU space} is a triple $(V,V^+, e)$ where $V$ is a $*$-vector space, $V^+$ is a proper cone in $V_{\rm sa}$, and $e$ is an archimedean order unit. In \cite{PTT}, it was shown that for any complex AOU space $(V,V^+,e)$, there exists a matrix ordering $\{C_n\}$ making $(V,\{C_n\},e)$ a complex operator system with $C_1 = V^+$. Moreover, there are \textit{maximal} and \textit{minimal} examples. The maximal operator system structure, denoted $\text{OMAX}(V)$, has the smallest matrix ordering (and hence the largest norm) making $V$ into an operator system, while the minimal operator system structure, denoted $\text{OMIN}(V)$, has the largest matrix ordering (and hence the smallest norm) making $V$ into an operator system. Similarly, if one begins with an $*$-vector space and a cone $M_k(V)^+$ in $M_k(V)_{\rm sa}$, one can define minimal and maximal operator system structures $\text{OMIN}_k(V)$ and $\text{OMAX}_k(V)$ whose matrix ordering agree at the $k$-th cone (c.f. \cite{ART23}).

In this section, we consider the real version of these constructions. While the constructions are very similar to the complex constructions, we will see that the results are quite different and do not follow from the obvious complexifications. We will avoid developing an explicit theory of AOU spaces (and, more generally, $k$-AOU spaces as in \cite{ART23}). Instead, we will start with an operator system $V$ and define new matrix-ordered vector spaces $\text{OMIN}_k(V)$ and $\text{OMAX}_k(V)$, keeping in mind that only the positive cone at the $k$-th level is needed for these constructions.

\subsection{Maximal operator system structure} 
\label{maosys}

Let $(V,\{C_n\},e)$ be a real operator system. For $k \in \mathbb{N}$, we define $\text{OMAX}_k(V)$ to be the triple $(V, \{C_n^{k-max}\}, e)$ where $\{C_n^{k-max}\}$ is the archimedean closure of the matrix ordering $\{D_n^{k-max}\}$ defined by \[ D_n^{k-max} := \{ \alpha^T \text{diag}(P_1, P_2, \dots, P_m) \alpha : \alpha \in M_{mk,n}(\mathbb{R}); P_1, \dots, P_m \in C_k \}. \] 

\medskip

{\bf Remark.}\ If $V$ is finite dimensional then it can be shown (see Theorem \ref{fdom})
that the archimedean closure is unnecessary, and indeed that $D_n^{k-max}$ is actually closed.

\begin{proposition} \label{Prop91} 
    Suppose that $(V,\{C_n\},e)$ is a real operator system and let $k \in \mathbb{N}$. Then $\{C_n^{k-max}\}$ is an matrix ordering 
    on $V$  for which $e$ is a matricially archimedean unit.  \end{proposition}

\begin{proof}
    To see that $\{C_n^{k-max}\}$ is a matrix ordering, it suffices to check that $\{D_n^{k-max}\}$ is a matrix ordering, since $e$ is 
    a  matricially archimedean unit for its archimedean closure. Suppose $x \in D_n^{k-max}$ and $y \in D_m^{k-max}$. Then $x = \alpha^T \text{diag}(P_1, \dots, P_n) \alpha$ and $y = \beta^T \text{diag}(Q_1, \dots, Q_m) \beta$ for scalar matrices $\alpha$ and $\beta$ and $P_1, \dots P_n, Q_1, \dots, Q_m \in C_k$. Then $$x \oplus y = (\alpha \oplus \beta)^T \text{diag}(P_1, \dots, P_n, Q_1, \dots, Q_m) (\alpha \oplus \beta).$$ So $x \oplus y \in D_{n+m}^{k-max}$. Now suppose $\gamma \in M_{n,j}(\bR)$. Then $$\gamma^T x \gamma = (\alpha \gamma)^T \text{diag}(P_1, \dots, P_n) (\alpha \gamma) \in D_j^{k-max}.$$ Hence $\{D_n^{k-max}\}$ is a matrix ordering.
\end{proof}

The resulting matrix-ordered $*$-vector space $\text{OMAX}_k(V)$ enjoys the following universal properties.  In this section all $k$-positive maps are assumed to be selfadjoint. 

\begin{proposition} \label{Prop: OMAX properties}
    Suppose that $(V,\{C_n\},e)$ is a real operator system and let $k \in \mathbb{N}$.
    \begin{enumerate}
        \item Suppose that $\varphi: V \to W$ is a $k$-positive
        map into a matrix-ordered  $*$-vector space $W$.  Suppose that either $W$ has an archimedean matrix order unit $f$, or that
        $Q = \varphi(e)$ is `matricially archimedean'  (see Remark 2 after Proposition {\rm \ref{Prop: Arch closure of matrix ordering}}).
        Then $\varphi: \text{OMAX}_k(V) \to W$ is completely positive. 
        \item Suppose $\{C_n'\}$ is a matrix ordering on $V$ for which $e$ is a matricially archimedean unit, 
        such that $C_k' = C_k$.  Then $C_n^{k-max} \subseteq C_n'$ for all $n \in \mathbb{N}$ (i.e.\ $\text{OMAX}_k$ is the smallest matrix ordering 
        on $V$ with matricially archimedean unit $e$  which agrees with $\{C_n\}$ at level $k$). 
    \end{enumerate}
    Consequently, $\{C_n^{k-max}\}$ is a proper matrix ordering, since $\{C_n\}$ is a proper matrix ordering.
\end{proposition}

\begin{proof} (1)\ 
    First suppose that $\varphi: V \to W$ is $k$-positive. Suppose $x \in M_n(\text{OMAX}_k(V))^+$ and let $\epsilon > 0$. Let $Q = \varphi(e) \in W^+$.  If $W$ is as in the hypothesis fix $t > 0$ such that $t f \otimes I_n - Q \in W^+$, where $f =1_W$. If $x + \epsilon t^{-1} e \otimes I_n = \alpha \text{diag}(P_1, \dots, P_n) \alpha$ for $\alpha$ a scalar matrix and $P_1, \dots, P_n \in M_k(V)^+$, then $$\varphi^{(n)}(x) + \epsilon t^{-1} Q \otimes I_n = \varphi^{(n)}(x + \epsilon t^{-1} e \otimes I_n) = \alpha^* \text{diag}(\varphi^{(k)}(P_1), \dots, \varphi^{(k)}(P_n)) \alpha \geq 0.$$ Thus $$\varphi^{(n)}(x) + \epsilon f \otimes I_n = \varphi^{(n)}(x) + \epsilon t^{-1} Q \otimes I_n + \epsilon f \otimes I_n - \epsilon t^{-1} Q \geq 0$$ for any $\epsilon > 0.$ It follows that $\varphi^{(n)}(x) \geq 0$, so $\varphi$ is completely positive.  The other case is similar: as before  
    $$\varphi^{(n)}(x) + \epsilon \, Q \otimes I_n = \varphi^{(n)}(x + \epsilon  e \otimes I_n)  \geq 0, \qquad \epsilon > 0,$$ and the hypothesis gives 
    $\varphi^{(n)}(x) \geq 0$. 
    This proves (1).

    (2)\ 
    Since $M_k(\text{OMAX}_k(V))^+ = M_k(V)^+$, 
    then because the identity map on $V$ is $k$-positive we will have by (1) that $M_n(\text{OMAX}_k(V))^+ \subseteq C_n'$ for any matrix ordering $\{C_n'\}$ with matricially archimedean unit $e$ such that $C_k' = M_k(V)^+$. 
    
    The final statement follows from $C_n^{k-max} \cap -C_n^{k-max} \subseteq C_n \cap -C_n = \{0\}$.
\end{proof}

\begin{corollary} If $V$ is a real operator system and $k \geq 2$ then $\text{OMAX}_k(V)$ is a real operator system.   \end{corollary}

\begin{proof} Indeed we saw above that it has a proper matrix ordering and is (matricially) archimedean.    We already saw that 
$C_k = C_k^{k-max}$, but  it is easy to see that similarly 
 $C_2 = C_2^{k-max}$.  From Proposition \ref{2enf} we deduce that $e$ is a matrix order unit
for  $\text{OMAX}_k(V)$.   So $\text{OMAX}_k(V)$ is a real operator system.  \end{proof}

\begin{corollary} \label{ifn}
    Let $(V, \{C_n\}, e)$ be a real operator system and let $k \in \mathbb{N}, k \geq 2$. Then $V 
    = \text{OMAX}_k(V)$ if and only if for every operator system $W$, every $k$-positive map $\varphi: V \to W$ is completely positive.  
    A similar statement holds  for $k = 1$, but now $W$ should be a  matrix ordered $*$-vector space with a specified matricially archimedean unit $f$, and we also assume that $\varphi(e) = f$. 
\end{corollary} 

\begin{proof} 
    If $V \cong \text{OMAX}_k(V)$, then every $k$-positive map $\varphi: V \to W$ is 
    completely positive by Proposition \ref{Prop: OMAX properties} part (1).  
    If $k=1$ we also take $\varphi$ to be unital. 
        For this we only need $W$ to be a  matrix ordered $*$-vector space with matricially archimedean unit $f$. 
    On the other hand, if every $k$-positive map $\varphi: V \to W$ is completely positive for any $W$, then this holds for the identity map from $V$ to $\text{OMAX}_k(V)$. Since $M_k(V)^+ = M_k(\text{OMAX}_k(V))^+$, this implies the matrix ordering on $V$ is contained in the matrix ordering $\{C_n^{k-max}\}$. By Proposition \ref{Prop: OMAX properties} part (2), $V \cong \text{OMAX}_k(V)$.  
\end{proof}

{\bf Remark.}  In the last corollary if $k \geq 2$, or if $k = 1$ and $\text{OMAX}(V)$ is
an operator system, then one may take $W = M_m$ for all $m \in \bN$.   To see this recall that any 
 $u : V \to B(H)$ is $k$-positive or completely positive if and only if $P_F u(\cdot)_{|F}$ is $k$-positive or completely positive for every finite dimensional $F \subseteq H$, and that $B(F) \cong M_m$ where $m = {\rm dim} \, F$.  See also \cite[Theorem 6.6]{Pnbook}. 
 
\begin{corollary} \label{isc}  Let $V$ be a real operator system, and suppose that either  $k  \geq 2$, or that $k = 1$ and $\text{OMAX}(V)$ is an 
operator system.  If  $x \in M_n(V)_{\rm sa}$ then $x \in M_n(\text{OMAX}_k(V))^+$ if and only if $\Psi^{(n)}(x) \geq 0$ for every 
$k$-positive map $\Psi : V \to W$ into an operator system.   This is also true if we replace $W$ by $M_r$, insisting on  all $r \in \bN$.
These are both also true if we further require that $\Psi$ is unital. 
  \end{corollary}

\begin{proof}
  The `forward direction' of these  follows from Proposition \ref{Prop: OMAX properties} part (1) again.  
  For the converse, take $W = \text{OMAX}_k(V)$ and $\Psi = I_V : V \to W$, which is  $k$-positive and unital, to obtain $x \in M_n(\text{OMAX}_k(V))^+$.
 The other cases follow by replacing $W$ first by $B(H)$, and then by $M_r$ as in the Remark after Corollary \ref{ifn}. \end{proof}

 Let $C_n'$ be the cone implicit in the Corollary, namely the $x \in M_n(V)_{\rm sa}$ with $\Psi^{(n)}(x) \geq 0$ for all  $k$-positive $\Psi \to B(H)$.
 
In the above, we have refrained from stating that $\text{OMAX}_k(V)$ is a real operator system when $k=1$. In particular, we have not claimed that it has a matrix order unit. The following example shows that when $k=1$, a matrix order unit is not guaranteed.

\begin{example} \label{ex. OMAX(bH)}
\emph{ 
    Consider  $\bH = \text{span}\{I, X, Y, Z \}$,  the quaternions with its natural matrix ordering inherited from $M_2(\bC)$. We claim that 
    \[ M_n(\text{OMAX}(\bH))^+ = \{ P \otimes I : P \in M_n(\mathbb{R})^+ \}. \]
    Indeed, if $T = A \otimes I + B \otimes X + C \otimes Y + D \otimes Z \in M_n(\text{OMAX}(\bH))^+$, then for every $\epsilon > 0$, \[ T + \epsilon I_n \otimes I = \alpha^T \text{diag}( r_1 I, \dots, r_n I) \alpha \] for some positive $r_i$,
    and thus $T + \epsilon I_n \otimes I$ is a sum of matrices of the form $P_i \otimes I$ with $P_i \in M_n(\mathbb{R})^+$. It follows that $T$ itself is a positive real matrix tensored with the identity. }
    
  \emph{  This matrix ordering does not make $\bH$ into an operator system. While it is archimedean, the unit fails to be an order unit: for example, there is no $t > 0$ such that \[ \begin{bmatrix} tI & -X \\ X & tI \end{bmatrix} \in  M_2(\text{OMAX}(\bH))^+ \] where $X$ is the canonical skew-hermitian basis element.}

   \emph{ Furthermore, we claim that for $\bH$ there is no smallest matrix ordering $\{C_n\}$ making $\bH$ into an operator system and such that $C_1 = \bH^+ = \mathbb{R}^+ \otimes I$. To see this notice that the mapping $\varphi_n: \bH \to \bH$ given by $\varphi_n(aI + bX + cY + dZ) = aI + n(aX + cY + dZ)$ is a unital order embedding. The matrix ordering $C_k(n) = \phi_n^{-1}(M_n(\bH)^+)$ is a matrix ordering on  $\bH$ with $C_1(n) = \mathbb{R}$ and making $\bH$ into an operator system. But it is easy to argue by looking at the principal $2 \times 2$ submatrices that $\cap_n C_k(n) = M_k(\text{OMAX}(\bH))^+$ for all $k$. }
\end{example}

{\bf Remark:} The only important property used in the example above is that $\bH$ has skew-adjoint elements. Repeating these arguments for any real operator system with non-zero skew adjoint elements, $\text{OMAX}(V)$ will lack an order unit and hence $V$ has no maximal order structure agreeing with $V^+$. In particular, $M_n(\bR)$ for $n > 1$ lacks a maximal operator system structure that agrees with the ``base'' cone $M_n(\bR)^+$.

To see this, suppose that  $X \in V$ is a skew element but \[ a(t)  = \begin{bmatrix} tI & -X \\ X & tI \end{bmatrix} \in  M_2(\text{OMAX}(V))^+ ,\] 
for some $t > 0$.   Define $u : \text{OMAX}(V) \to \text{OMAX}(V)$ to be the identity on $V_{\rm sa}$ but $C$ times the identity on $V_{\rm as}$
for $C > 0$.  This is positive, so ucp.   Thus \[ \begin{bmatrix} tI & -C X \\ C X & tI \end{bmatrix} \in  M_2(\text{OMAX}(V))^+ \] for all $C > 0$.
Dividing by $C$ and using the archimedean property we see that $a(0) \geq 0$.  But then $-a(0) \geq 0$, so that $X = 0$.

One consequence of these examples is that $\text{OMAX}_k(V)_c$ may not be the same as $\text{OMAX}_k(V_c)$ in general.  (Here and in the rest of this section $W_c$ is the complexification considered in 
Section 
2, which works for matrix ordered spaces).  The relationship between these spaces is summarized in the following.

\begin{proposition} \label{whdl}
    Let $V$ be a real operator system and $k \in \mathbb{N}$ an integer. Then the identity map $id: \text{OMAX}_k(V)_c \to \text{OMAX}_k(V_c)$ is ucp.
\end{proposition}

\begin{proof}
    Let $x+iy \in 
        M_n(\text{OMAX}_k(V)_c)^+$. Then $c(x,y) \in M_{2n}(\text{OMAX}_k(V))^+$. So for each $\epsilon > 0$, there exists a real scalar matrix $C$ and $P_1, \dots, P_n \in M_k(V)^+$ such that $c(x,y) + \epsilon I_{2n} = C^T \text{diag}(P_1, \dots, P_n) C$. Write $C = \begin{bmatrix} A \\ B \end{bmatrix}$ where $A,B$ are real scalar matrices, each with $n$ rows, and write $P = \text{diag}(P_1, \dots, P_n)$. Since $c(x,y) + \epsilon I_{2n} = C^T \text{diag}(P_1, \dots, P_n) C$, We get $x + \epsilon I_n = A^T P A = B^T P B$ and $y = B^T P A = -A^T P B$. Thus
    \[ c(x,y) + \epsilon I_n = \begin{bmatrix} A & -B \\ B & A \end{bmatrix}^T \begin{bmatrix} P & 0 \\ 0 & P \end{bmatrix} \begin{bmatrix} A & -B \\ B & A \end{bmatrix}. \] Therefore $x + iy + \epsilon I_n = \alpha^* \text{diag}(P_1, \dots, P_n) \alpha$ for $\alpha = A + i B$. It follows that $x+iy \in M_n(\text{OMAX}_k(V_c))^+$. So the identity map from $\text{OMAX}_k(V)_c$ to $\text{OMAX}_k(V_c)$ is ucp.
\end{proof}

The obvious example of a case where $\text{OMAX}_k(V)_c \neq \text{OMAX}_k(V_c)$ is already discussed above. When $k = 1$ and $V = \bH$, we know that $\text{OMAX}(V)$ lacks a matrix order unit and hence $\text{OMAX}(V)_c$ (defined as matrix-ordered $*$-vector space) is not an operator system. On the other hand, $\bH_c = M_2(\bC)$ and so $\text{OMAX}(V_c) = \text{OMAX}(M_2(\bC))$ is an operator system by \cite{PTT}.

We next (and also in the discussion after Proposition
\ref{mr})  relate $\text{OMAX}_k(V)_c \neq \text{OMAX}_k(V_c)$ to the notion of entanglement in quantum physics.
  Example 7.1 in \cite{CDPR2} provides explicit examples of matrices $A \in M_{4}(\bR) = M_{2}(M_{2}(\bR))$ which are not separable as real matrices (that is, are entangled) but are separable as complex matrices.   We claim that any such matrix in $M_{n^2} (\bR) =  M_{n}(\bR) \otimes M_{n}(\bR)$ gives an example 
 demonstrating that $\text{OMAX}(V_c) \neq \text{OMAX}(V)_c$  where $V = M_{n}(\bR)$.  (This follows from the above as well, just as in the 
 Example \ref{ex. OMAX(bH)} and by the reasoning there and in the last paragraph, since this algebra  has nontrivial skew-adjoint elements).  See \cite{CDPR2} and \cite{JKPP} for consequences of this in quantum physics.  
Indeed, separability of a matrix $A$ in their sense 
is the same as  Schmidt number SN$(A) = 1$ \cite{CDPR2,  JKPP}.  By the calculations in \cite[Theorem 5b]{JKPP},
 if  $V = M_n(\bF)$ then the archimedean closure 
in the definition of OMAX$(V)$ is unnecessary, and every separable state of $M_n \otimes M_m$ is in 
$M_n({\rm OMAX}(M_m))^+$.
Indeed  for $\xi \in \bF^n, \eta \in \bF^m$ we
have $| \xi \rangle \langle \xi | \otimes | \eta \rangle \langle \eta | = \alpha^* P  \alpha \in M_n({\rm OMAX}(M_m))^+$, 
where $\alpha = \xi \in M_{n,1}$ and $P = | \eta \rangle \langle \eta | \in M_m^+$. 
Thus applying Theorem 5 of \cite{JKPP} (which is stated in the complex case, but whose proof is also valid in the real case), we see that $A \in M_{n}(\text{OMAX}(M_{n}(\bC)))^+ = M_{n}(\text{OMAX}(M_{n}(\bR)_c))^+$ but $A \notin M_{n}(\text{OMAX}(M_{n}(\bR)))^+$ and hence is not an element of $M_{n}(\text{OMAX}(M_{n}(\bR))_c)^+$.

\subsection{Minimal operator system structure}

Let $V$ be a real operator system and let $k \in \mathbb{N}$. We define
\[ M_n(\text{OMIN}_k(V))^+ := \{ x \in M_n(V)_{\rm sa} : \alpha^T x \alpha \geq 0 \text{ for every } \alpha \in M_{n,k}(\bR) \}. \]
We let $\text{OMIN}_k(V)$ denote the resulting matrix ordered $*$-vector space. This definition is based on the one in \cite{ART23}. We will relate this definition to the one from \cite{BXhabli12} below.

\begin{proposition}
    The sequence $\{M_n(\text{OMIN}_k(V))\}$ is a matrix ordering on $V$ and $e$ is an archimedean matrix order unit. If this matrix ordering is proper, then $\text{OMIN}_k(V)$ is an operator system.  In particular $\text{OMIN}_k(V)$ is an operator system if $k \geq 2$.
\end{proposition}

\begin{proof}
    Suppose $x \in M_n(\text{OMIN}_k(V))^+$ and $y \in M_m(\text{OMIN}_k(V))^+$. Let $\alpha \in M_{n+m,k}(\bR)$. We may write $\alpha = \begin{bmatrix} \alpha_1 & \alpha_2 \end{bmatrix}$ with $\alpha_1$ $n \times k$ and $\alpha_2$ $m \times k$. Then
    \[ \alpha^* (x \oplus y) \alpha = \alpha_1^* x \alpha_1 + \alpha_2^* y \alpha_2 \in M_k(V)^+. \] Therefore $x \oplus y \in M_{n+m}(\text{OMIN}_k(V))^+$. Now let $\beta \in M_{n,m}(\mathbb{R})$ and consider $\beta^* x \beta \in M_m(V)$. For every $\alpha \in M_{m,k}(\bR)$ we have $\alpha^*(\beta^* x \beta) \alpha = (\beta \alpha)^* x (\beta \alpha) \geq 0$. Therefore $\beta^* x \beta \in M_m(\text{OMIN}_k(V))^+$. We conclude that $(M_n(\text{OMIN}_k(V))^+$ is a matrix ordering.

    To see that $e$ is a matrix order unit, let $x \in M_n(V)_{\rm sa}$. Then there exists $t > 0$ such that $x + te \geq 0$. Hence for every $\alpha \in M_{n,k}(\bR)$, $ \alpha(x + te) \alpha \geq 0$. So $x + te \in M_n(\text{OMIN}_k(V))^+$ and hence $e$ is a matrix order unit. If $x + \epsilon t \in M_n(\text{OMIN}_k(V))^+$ for every $\epsilon > 0$, then for every $\alpha \in M_{n,k}(\bR)$ with $\|\alpha\| \leq 1$ we have $\alpha^* x \alpha + \epsilon \alpha^* \alpha \otimes e \geq 0$. Since $\alpha^* \alpha \leq I_k$, this implies that $\alpha^* x \alpha + \epsilon I_k \otimes e \geq 0$. Hence $\alpha^* x \alpha \geq 0$ for every $\alpha \in M_{k,n}(\bR)$ with $\|\alpha\| \leq 1$. By rescaling, we see that $x \in M_n(\text{OMIN}_k(V))^+$. Therefore $e$ is archimedean.
    
    As in the OMAX case if $k \geq 2$ then $M_2(V)^+ = M_2(\text{OMIN}_k(V))^+$ so that the latter cone is proper, hence all cones are proper.  Now appeal to the second last statement of the Proposition. 
\end{proof}

We will see a case where the matrix ordering on $\text{OMIN}_k(V)$ is not proper. The following relates our definition of $\text{OMIN}_k(V)$ to the one used by Xhabli \cite{BXhabli12} (in the complex case).

\begin{proposition} \label{prop: k-min universal property}
    Let $V$ be a real operator system and $W$ a matrix-ordered $*$-vector space. Suppose that $\varphi: W \to V$ is $k$-positive. Then $\varphi: W \to \text{OMIN}_k(V)$ is completely positive.
\end{proposition} 

\begin{proof}
    Let $x \in M_n(W)^+$. Then for every $\alpha \in M_{n,k}(\bR)$, $\alpha^* \varphi^{(n)}(x) \alpha = \varphi^{(k)}(\alpha^* x \alpha) \geq 0$  since $\varphi$ is $k$-positive and $\alpha^* x \alpha \in M_k(V)^+$. Thus $\varphi^{(n)}(x) \in M_n(\text{OMIN}_k(V))^+$.
\end{proof}

Note that the above result does not assume that $\text{OMIN}_k(V)$ is an operator system, only a matrix ordered $*$-vector space. 

\begin{corollary}
    The matrix ordering $\{M_n(\text{OMIN}_k(V))^+\}$ contains any other matrix ordering $\{D_n\}$ making
    $V$ an operator system and satisfying $D_k = M_k(V)^+$.
\end{corollary} 

\begin{proof}
    If $\{D_n\}$ is a matrix ordering on $V$ such that $D_k=M_k(V)^+$, then the identity map from $(V,\{D_n\})$ to $(V,\{M_n(V)^+\})$ is $k$-positive. Hence By Proposition \ref{prop: k-min universal property} the identity map from $(V,\{D_n\})$ to $\text{OMIN}_k(V)$ is ucp. Therefore $D_n \subseteq M_n(\text{OMIN}_k(V))^+$ for every $n$.
\end{proof} 

The following real analogue of a well-known fact for complex maps
may be proved in the same way (e.g. Proposition 2.2 of \cite{BXhabli12}).  
(Indeed this proof works even if $V$ is merely matrix ordered.) 

\begin{proposition} \label{prop: k-pos to M_k implies ucp}
    A linear map $\varphi: V \to M_k(\bR)$ on a real operator system 
    is completely positive if and only if it is $k$-positive. 
\end{proposition}

We can now relate our definition of $\text{OMIN}_k(V)$ to the one from \cite{BXhabli12}.

\begin{proposition} \label{prop: OMIN-k matrix state}
    Let $V$ be a real operator system. Then $x \in M_n(\text{OMIN}_k(V))^+$ if and only if $\varphi^{(n)}(x) \geq 0$ for every ucp $\varphi: V \to M_k(\bR)$.
\end{proposition}

\begin{proof}
    First assume $x \in M_n(\text{OMIN}_k(V))^+$ and $\varphi: V \to M_k(\bR)$ is ucp. Clearly $M_k(\text{OMIN}_k(V))^+ = M_k(V)^+$. Hence $\varphi: \text{OMIN}_k(V) \to M_k$ is ucp by the last result, and thus $\varphi^{(n)}(x) \geq 0$.

    Next suppose that $\varphi^{(n)}(x) \geq 0$ for every ucp $\varphi: V \to M_k(\bR)$. Then for every $\alpha \in M_{n,k}(\bR)$, $(\alpha^* \otimes I_n) \varphi^{(n)}(x) (\alpha \otimes I_n) = \varphi^{(k)}(\alpha^* x \alpha) \geq 0$. Since this holds for all ucp $\varphi: V \to M_k$, by Corollary \ref{plem} we have  $\alpha^* x \alpha \geq 0$ for all such $\alpha$. So $x \in M_n(\text{OMIN}_k(V))^+$. 
\end{proof}

We say that a real operator system is $k$-minimal if $V \cong \text{OMIN}_k(V)$. The next statement summarizes the properties of $k$-minimal operator systems.

\begin{corollary} \label{cor: k-minimal characterization}
    Let $V$ be a real operator system. 
    The following statements are equivalent:
    \begin{enumerate}
        \item $V$ is $k$-minimal.
        \item For every operator system  $W$, every $k$-positive map $\varphi: W \to V$ is completely positive. 
         \item For every $m \in \bN$, every $k$-positive map $\varphi: M_m \to V$ is completely positive. 
          \item For every matrix-ordered $*$-vector space  $W$, every $k$-positive map $\varphi: W \to V$ is completely positive. 
        \item $V$ admits a complete order embedding into $C(X, M_k(\bR))$ for some compact Hausdoff space $X$.
    \end{enumerate} \end{corollary}

\begin{proof}  That (1) implies  (4) 
follows from Proposition \ref{prop: k-min universal property}.   That (4)  implies (2) implies (3) is trivial. 
That  (3)  implies  (1) follows because the $k$-positive maps $M_m \to V$ are the completely positive maps $M_m \to \text{OMIN}_k(V)$, 
 and correspond to $M_m(\text{OMIN}_k(V))^+$;  similarly the  completely positive maps $M_m \to V$ 
 correspond to $M_m(V)^+$. 

    To show (1) implies (5), assume $V$ is $k$-minimal. Let $X$ denote the set of ucp maps from $V$ to $M_k(\mathbb{R})$ equipped with the 
    weak* topology. Define $\pi: V \to C(X,M_k(\bR))$ by $\pi(x)(\varphi) = \varphi(x)$. By Proposition \ref{prop: OMIN-k matrix state}, $\pi$ is a unital complete order embedding.
    
    Conversely, suppose that $V \subseteq C(X,M_k(\bR))$. Suppose $W$ is a real 
    operator system or 
     matrix-ordered $*$-vector space and $\varphi: W \to V$ is 
     $k$-positive.   
     Then $\varphi$ is completely positive if and only if for every $x \in X$ and $w \in M_n(W)^+$, $\varphi^{(n)}(w)(x) \geq 0$. For each $x \in X$, define $\varphi_x: W \to M_k(\bR)$ by $\varphi_x(w) = \varphi(w)(x)$. Since $\varphi$ is $k$-positive, $\varphi_x$ is $k$-positive for every $x \in X$. By Proposition \ref{prop: k-pos to M_k implies ucp}, each $\varphi_x$ is completely positive. It follows that $\varphi$ is completely positive. So (5) implies (2). \end{proof}

{\bf Remark.}   In the last proof we used the fact that for a matrix ordered space $W$, CP$(M_m,W)$ is in bijective correspondence 
with $M_m(W)^+$.   One direction of this follows by applying the map to 
$[e_{ji}]$.  We leave the other direction as an exercise (see e.g.\ \cite{CEinj}). 

\medskip

We now give an example of a real operator system with no minimal matrix ordering.

\begin{example} \label{ex: OMIN(bH)}
    \emph{Consider the real operator system $\bH$. Recall that $\bH$ possess exactly one real state, namely $\varphi: tI + aX + bY + cZ \mapsto t$. Consider $\text{OMIN}(V) := \text{OMIN}_1(\bH)$. By Proposition \ref{prop: OMIN-k matrix state}, $\begin{bmatrix} x_{ij} \end{bmatrix} \in M_n(\text{OMIN}(V))^+$ if and only if $\begin{bmatrix} \varphi(x_{ij}) \end{bmatrix} \in M_n(\bR)^+$. However, this implies that $C_2$ is not proper. For example, \[ \pm \begin{bmatrix} 0 & -X \\ X & 0 \end{bmatrix} \in C_2 \] since its image under $\varphi^{(2)}$ is the zero matrix.}

    \emph{To see that $\bH$ has no minimal real matrix ordering, consider the mapping $\varphi_n: \bH \to \bH$ given by $\varphi_n(tI + aX + bY + xZ) = tI + n^{-1}(aX + bY + xZ)$. This map is positive and one-to-one, so $C_k(n) := \varphi_n^{-1}(M_k(\bH)^+)$ defines a matrix ordering on $\bH$. However
    \[ \pm \begin{bmatrix} I & -nX \\ nX & I \end{bmatrix}, \pm \begin{bmatrix} I & -nY \\ nY & I \end{bmatrix}, \pm \begin{bmatrix} I & -nZ \\ nZ & I \end{bmatrix} \in C_2(n) \] for every $n$. From this, we see that the archimedean closure of $\cup C_2(n)$ is non-proper, since for instance \[ \pm \begin{bmatrix} \frac{1}{n} I & -X \\ X & \frac{1}{n} I \end{bmatrix} \in C_2(n) \] for all $n \in \mathbb{N}$. 
    Thus $\pm \begin{bmatrix} 0 & -X \\ X & 0 \end{bmatrix}$ is in this archimedean closure.  So there is no proper largest matrix ordering on $\bH$ agreeing with $\bH^+$ and making $\bH$ into a real operator system.}
\end{example}

{\bf Remark:} As was the case for $\text{OMAX}(V)$, it is easy to see that the argument above showing that $\text{OMIN}(\bH)$ is a non-proper matrix ordering applies for any operator system $V$ with a non-trivial skew-adjoint operator.
To see this, notice that the 
argument in the example after Corollary \ref{isc}, applied to a skew element $X \in V$ shows  that $0 \neq \pm a(0) \geq 0$
(since the positive map  $u : \text{OMIN}(V) \to \text{OMIN}(V)$ there is again ucp).   So $M_2(\text{OMIN}(V))^+$ is non-proper.

\bigskip

We now consider the relationship between the complexification of $\text{OMIN}_k(V)$ and $\text{OMIN}_k$ of the complexification of $V$.

\begin{lemma}
    Let $V$ be a real operator system such that $\text{OMIN}_{k}(V)$ is also 
    an operator system (e.g.\ if $k \geq 2$). Then $\text{OMIN}_{k}(V)_c$ is a $k$-minimal complex operator system.
\end{lemma}

\begin{proof}
    This is because $\text{OMIN}_{k}(V)$ admits a complete order embedding into $C(X,M_k(\mathbb{R}))$ and thus
     $\text{OMIN}_{k}(V)_c$ admits a complete order embedding into $C(X,M_k(\mathbb{C}))$.
\end{proof}

\begin{proposition} \label{mr} 
 Under the conditions of the last lemma, the identity map from $\text{OMIN}_k(V_c)$ to $(\text{OMIN}_k(V)_c$ is ucp.
\end{proposition}

\begin{proof}
    Since $(\text{OMIN}_k(V))_c$ is $k$-minimal, it suffices to show that the identity map is $k$-positive. 
    Let $x + iy \in M_k(\text{OMIN}_k(V_c))^+ = M_k(V_c)+$. Then $c(x,y) \in M_{2k}(V)^+$. Hence for every scalar matrix $A \in M_{2k,k}(\bR)$, $A^T c(x,y) A \in M_k(V)^+$. It follows that $c(x,y) \in M_{2k}(\text{OMIN}_k(V))^+$. Therefore $x+iy \in M_k(\text{OMIN}_k(V)_c)^+$. So the identity map is $k$-positive.
\end{proof}

As in the case for $\text{OMAX}_k$, there are examples of real operator systems $V$ for which $\text{OMIN}_k(V)_c \neq \text{OMIN}_k(V_c)$. Once again, the quaternions $\bH$ provide such an example, since $\text{OMIN}(\bH)$ lacks a proper cone, as does its formal complexification, whereas $\text{OMIN}(\bH_c) = \text{OMIN}(M_2(\bC))$ is an operator system by \cite{PTT}. 

We next  relate the equality of $\text{OMIN}_k(V)_c$ and $\text{OMIN}_k(V_c)$, and the analogous equality for $\text{OMAX}_k$,
 to the notion of entanglement breaking in quantum physics: 
 
\subsection{Applications and complements} 
 
 {\bf Example (Entanglement-breaking channels)} 
Example 8.3 of \cite{CDPR2} provides, for any integer $p \in \mathbb{N}$, examples of maps on $\bR^{2p} \otimes \bR^{2p}$ which are $p$-entanglement breaking as complex maps 
(that is, their complexification is $p$-entanglement breaking) but are not $p$-entanglement breaking as real maps. 
We claim that  Corollary 6 of \cite{JKPP}  (originally
proved in \cite[Section 6]{BXhabli12}; see also \cite[Section 6]{PTT}), and Theorem 8.1 of \cite{CDPR2} in the real case, 
 then implies that  $\text{OMIN}_{p}(M_{2p}(\bR_c)) \neq \text{OMIN}_p(M_{2p}(\bR))_c$, and 
that  $\text{OMAX}_{p}(M_{2p}(\bR_c)) \neq \text{OMAX}_p(M_{2p}(\bR))_c$, the latter if $p > 1$ (the $p=1$ case was discussed earlier).

To see this we first argue that parts (6) and (5) respectively of  \cite[Theorem 8.1]{CDPR2} 
imply that a map $\Phi : M_{2p}(\bF) \to M_{2p}(\bF)$ is $p$-entanglement breaking if and only if $\Phi$ is completely positive on $\text{OMIN}_{p}$, and  if and only if $\Phi$
 is completely positive  when regarded as a map into $\text{OMAX}_p$ if $p \geq 2$.
  In the case of part (6) 
    this follows from the proof that  (3)  implies  (1) in Corollary \ref{cor: k-minimal characterization}. 
      For a matrix $x \in M_r(\text{OMIN}_{p}(M_n))^+$ corresponds to a $p$-positive map 
  $\Psi : M_r \to M_n$, and similarly $\Phi^{(r)}(x) \geq 0$ corresponds to $\Phi \circ \Psi$ being completely positive.
  In the case of part  (5)  this follows from Corollary \ref{isc}.
  
Recall from the Remark after Proposition \ref{inco} that  any matrix ordered space is completely order embedded in its complexification.
If $\varphi: M_{2p}(\bR) \to M_{2p}(\bR)$ is not completely positive on $\text{OMIN}(M_{2p}(\bR))^+$ but $\varphi_c: M_{2p}(\bC) \to M_{2p}(\bC)$ is completely positive on $\text{OMIN}_k(M_{2p}(\bC))^+$, then we must conclude that $\text{OMIN}_{p}(M_{2p}(\bR_c)) \neq \text{OMIN}_p(M_{2p}(\bR))_c$.  Similarly, if $\varphi: M_{2p}(\bR) \to M_{2p}(\bR)$ is not completely positive as a map into $\text{OMAX}_p(M_{2p}(\bR))^+$,
 but $\varphi_c$ is completely positive  as a map into $\text{OMAX}_p(M_{2p}(\bC))^+$, then $\text{OMAX}_{p}(M_{2p}(\bR_c)) \neq \text{OMAX}_p(M_{2p}(\bR))_c$. 
\bigskip

Some if not all of the next result seems to be known in the complex case, where the proof is similar, but we are not aware of a proof in the literature.

\begin{theorem} \label{fdom} Let $V$ be a finite-dimensional real operator system.  
 Then the cones $(D_n^{k-max})$ are closed, and agree with the cones
 $(C_n^{k-max})$ and  $(C_n')$  (defined after Corollary {\rm \ref{isc}}).  
Indeed if $k = 1$ then these cones are just $M_n^+ \otimes V^+$, or equivalently,  the cones and 
$M_n({\rm OMAX}_k(V))^+$ are exactly 
the nonentangled (i.e.\ separable) elements of $M_n(V)^+$.  \end{theorem} 

\begin{proof}   Write $D_n = D_n^{k-max}$.  Note first that $D_n = C_n$ for all $n \leq k$, so that we only need to consider $n > k$. Suppose $x + \epsilon I_n \in D_n$ for every $\epsilon > 0$. Then for every $m \in \mathbb{N}$, there exist matrices $A_{1,n}, \dots, A_{r,m} \in M_{k,n}$ and positive $P_{1,m}, \dots, P_{r,m} \in D_k$ such that $x + \frac{1}{m} I_n = \sum A_{i,m}^* P_{i,m} A_{i,m}$. Since $M_n(V)$ is finite-dimensional, by the well known Caratheodory's theorem
for cones and the conical hull (a short proof of which 
may even be found on wikipedia!), we can assume the number of terms $r$ in the sum is no more than $\dim(M_n(V)) =: d$. In fact, we can take the number of terms to be exactly $d$ by adding terms with $P_{i,m} = 0$. Hence, we have
    \[ x + \frac{1}{m} I_n = \sum_i A_{i,m}^* P_{i,m} A_{i,m} = A_m^* P_m A_m \] where $A_m^* = \begin{bmatrix} A_{1,m}^* & \dots & A_{d,m}^* \end{bmatrix}$ and $P_m = \text{diag}(P_{1,m}, \dots, P_{d,m})$. Write each $A_{i,m}$ as $A_{i,m} = Q_{i,m} U_{i,m}$ where $Q_{i,m} \in M_k^+$ and $U_{i,m} \in M_{k,n}$ is a partial isometry, and such that $Q_i U_{i,m} U_{i,m}^* = Q_i$ (using polar decomposition). Notice \[ \| Q_{i,m} P_i Q_{i,m} \| = \| U_{i,m} A_{i,m}^* P_i A_{i,m} U_{i,m}^* \| \leq \| A_{i,m}^* P_i A_{i,m} \| \leq \| Q_{i,m} P_i Q_{i,m}\|. \] Then replacing $P_{i,m}$ with $Q_{i,m}^* P_{i,m} Q_{i,m} \in D_k$ and replacing $A_{i,m}$ with $U_{i,m}$, we may assume $\|A_{i,m}\|=1$ and $\|P_{i,m}\| = \|A_{i,m}^* P_{i,m} A_{i,m}\|$. Because
    \[ 0 \leq A_{i,m}^* P_{i,m} A_{i,m} \leq \sum_j A_{j,m}^* P_{j,m} A_{j,m} = x + \frac{1}{m} I_n \] 
    we have $\|P_{i,m}\| \leq \|x\| + \frac{1}{m}$ for each $i=1,\dots,d$. Thus we may assume $\|P_m\| \leq \|x\| + \frac{1}{m}$ and $\|A_m\| = \| \sum A_{i,m}^* A_{i,m} \|^{1/2} \leq d^{1/2}$ for each $m \in \mathbb{N}$.
    
    Since the ball of radius $d^{1/2}$ in $M_{k,n}$ is compact, there exists a subsequence of $\{A_m\}$ that converges to a matrix $A$ with $\|A\|\leq d^{1/2}$. Passing to a convergent subsequence, we observe that for any norm on $M_n(V)$ coming from an operator system structure on $V$ with matrix cones $(D_k)$, we have 
    \begin{eqnarray} \| A_m^* P_m A_m - A^* P_m A \| & = & \| A_m^* P_m A_m - A_m^* P_m A + A_m^* P_m A - A^* P_m A\| \nonumber \\ & \leq &  \| A_m^* P_m \| \|A_m - A\| + \|A_m^* - A^*  \| \|P_m A\| \nonumber \\
    & \leq & 2 d^{1/2} (\|x\| + \frac{1}{m}) \, \|A_m - A\| \to 0. \nonumber \end{eqnarray}
    Also, notice each $P_m$ is block diagonal with diagonal blocks $P_{i,m}$ satisfying $\|P_{i,m}\| \leq \|x\| + 1$ and $P_{i,m} \in D_k$. Because $D_k$ is closed in the norm induced by the archimedean order unit $I_k$ on $M_k(V)$, each bounded sequence $\{P_{i,m}\}$ has a convergent subsequence converging in $D_k$ (since $V$ is finite-dimensional). Hence, by passing to a subsequence, $P_m$ has a limit $P$ which is a block diagonal element with blocks in $D_k$, and such that $\|P\| \leq \|x\|$.

    We claim that $x = A^* P A$. With respect to any norm on $M_n(V)$ coming from an operator system structure on $V$ with $M_k(V)^+=D_k$ we have
    \begin{eqnarray}
        \|x - A^* P A\| & = & \|x + \frac{1}{m}I_n - \frac{1}{m} I_n - A_m^* P_m A_m + A_m^* P_m A_m - A^* P_m A + A^* P_m A - A^* P A \| \nonumber \\
        & \leq & \|x + \frac{1}{m} I_n - A_m^* P_m A_m\| + \|\frac{1}{m} I_n \| + \|A_m^* P_m A_m - A^* P_m A\| + \|A^* (P_m - P) A\| \nonumber \\
        & \leq & \frac{1}{m} + \|A_m^* P_m A_m - A^* P_m A\| + \|P_m - P\| \to 0 \nonumber
    \end{eqnarray}
    since $x + \frac{1}{m} I_n = A_m^* P_m A_m$. It follows that $x = A^*PA$.
        Thus $(D_k) = (C_n^{k-max})$, since they  are archimidean.   
    
    If $k = 1$ it  is easy to see that $D_n \subseteq M_n^+ \otimes V^+$, and the converse is an exercise.   
  One can prove $M_n^+ \otimes V^+$ is closed using a similar but simpler `Caratheodory's theorem
for cones' argument.  Given a (bounded) converging sequence $(z_n)$ from $M_n^+ \otimes V^+$, use that theorem 
to write $z_n = \sum_{k=1}^d \, t_k^n A_k^n \otimes v_k^n$ for $t_k^n \geq 0$.  Fix a faithful state $\varphi$ on $V$.  We may assume that
$\varphi(v_k^n) = {\rm tr}(A_k^n) = 1$.  Since $(({\rm tr} \otimes \varphi)(z_n))$ is bounded, so are the $t_k^n$.  Then extract a converging subsequence.
We leave the details to the reader.
     (We are not saying that $D_n$ is closed in the   order unit norm associated with $D_n$, indeed  $D_n$ may have no order unit.)  

That $D_n = C_n'$ follows by a variant of the proof of Corollary \ref{plem}.  If $v \in C_n' \setminus D_n$ and 
$D_n$ is  closed  there exists a linear functional $\psi : M_n(V) \to \bR$ with $\psi(K) \geq 0 > \psi(v)$.  
As in the cited proof we obtain cp  $\Psi : {\rm OMAX}_k(V) \to M_n$ with $\Psi^{(n)}(v) \geq 0$ by definition of $C_n'$ and this implies 
$0 \leq \psi(v),$  a contradiction. 
\end{proof}

{\bf Remark.}  1)\  By the last result, if $V$ is a real finite-dimensional operator system 
then Corollary \ref{isc} holds without the assumption that $V = {\rm OMAX}(V)$. 
(For the unital assertion
note that the real variant of \cite[Exercise 6.2 (ii)]{Pnbook} works even if $\cS$ is not an operator system.) 
We are also able to prove that $D_n$ is  the intersection of all the $n$th level
cones for operator system structures
on $V$ with the
 same $V^+$.
Of course $V = {\rm OMAX}(V)$ if and only if every unital 
positive map $\Psi : V \to M_m$ is completely positive for all $m \in \bN$, and this implies that
$V$ has trivial (i.e.\ identity)  involution.  Thus if the involution on a  real finite-dimensional operator system
is nontrivial then there exists a selfadjoint positive but not completely positive map 
on $V$.  Indeed the latter map may be chosen to be unital.
Also, $\Phi : W \to V$ is completely positive as a map into  ${\rm OMAX}(V)$ if and only if 
$\Psi \circ \Phi $ is completely positive  for  every positive map $\Psi : V \to M_m$ and  for all $m \in \bN$.
Indeed if every positive map $\Psi : V \to M_m$ is completely positive then $C_n' = M_n(V)^+$.
And  if 
$\Psi \circ \Phi $ is completely positive  for  every positive map $\Psi : V \to M_m$ and $w \in 
M_n(W)^+$ then $(\Psi \circ \Phi)^{(n)}(w) \geq 0$ and so $\Phi^{(n)}(w) \in C_n' = M_n({\rm OMAX}(V))^+$. 

\medskip

2)\ If $V$ is infinite dimensional and $k = 1$ then $D_n = M_n^+ \otimes V^+$ as before, but 
$D_n \neq C_n$ necessarily, and $D_n$ is not necessarily closed in $M_n(V)$.   See \cite[Remark 3.12]{PTT}, where 
the example of $C([0,1])$ is considered. 
Indeed if $V = {\rm OMAX}(V)$  then $D_n$ is not necessarily closed in $M_n(V)$ because it is not even 
archimedean.

\medskip

3)\ If $\cS$ is a finite dimensional 
real or complex operator system the fact above that $M_n^+ \otimes \cS$ is closed allows one to improve \cite[Corollary 6.7(iii)]{Pnbk} to: $M_n(\cS)^+ = M_n^+ \otimes \cS$.  We do not recall seeing this fact in the literature in this generality. 

\medskip

Next, we check that when $V$ is finite-dimensional, $\text{OMIN}_k$ and $\text{OMAX}_k$ are dual 
structures. Note that if $V$ is a real operator system then either both   OMAX$(V)$ and  OMIN$(V)$ are operator systems, or both are not (since we saw above and in Proposition \ref{isos} that these are operator systems  if and only if  the involution is trivial, i.e.\ the identity). 

\begin{theorem} \label{OMINMAXduality}
    Suppose that $V$ is a finite-dimensional operator system. Then $\text{OMAX}_k(V)^d \cong \text{OMIN}_k(V^d)$ and $\text{OMIN}_k(V)^d \cong \text{OMAX}_k(V^d)$ for all $k \in \bN$.  
\end{theorem}

\begin{proof}   Let $(W,f)$ be an operator system or an archimedean matrix ordered  $*$-vector  space 
and suppose that $\varphi: \text{OMIN}_k(V)^d \to W$ is 
unital and $k$-positive.  Suppose that $k \geq 2$.
Without loss of generality, assume that $W$ is finite dimensional (for example, by restricting the range to $\varphi(V^d)$). Then $\varphi^d: W^d \to \text{OMIN}_k(V)$ is $k$-positive. By the universal property of $\text{OMIN}_k$, $\varphi^d$ is completely positive. It follows that $\varphi$ is ucp. Since $W$ was arbitrary, $\text{OMIN}_k(V)^d \cong \text{OMAX}_k(V^d)$. The other identity is similar.

The OMIN cones are closed if $V$ is a finite dimensional operator system. We use the original  norm topology' on $M_n(V)$, but of course all TVS topologies coincide.  If selfadjoint $x_n \to x$ in this topology, and $\alpha^T x_n \alpha \in V_+$, then since $V_+$ is closed we see that $\alpha^T x \alpha \in V_+$. So the OMIN cones are closed, and hence OMIN$(V) \cong {\rm OMIN}(V)^{dd}$ completely order isomorphically by \cite[Lemma 4.3]{CEinj}.  
Next note that the argument in the line before Theorem \ref{dcom}  shows that ${\rm OMIN}(V)^{d}$ is matricially archimedean.  Hence its cones contain those of ${\rm OMAX}(V^d)$ by Proposition 9.2 (2). On the other hand the converse containment follows from the argument
above, and the fact that the OMAX cones are $C_n'$.   So $\text{OMIN}(V)^d \cong \text{OMAX}(V^d)$.  Replacing $V$ by $V^d$ 
we have OMIN$(V^d) \cong \text{OMIN}(V^d)^{dd} \cong \text{OMAX}(V)^d$. 
\end{proof}

We conclude this section with some examples where real operator systems do behave as their complex counterparts.  We remarked after Example \ref{ex. OMAX(bH)} and Example \ref{ex: OMIN(bH)} that whenever a real operator system $V$ has non-trivial skew-adjoint elements, then $\text{OMAX}(V)$ and $\text{OMIN}(V)$ are not operator systems. Moreover, it is generally the case that $\text{OMAX}_k(V)_c \neq \text{OMAX}_k(V_c)$ and $\text{OMIN}_k(V)_c \neq \text{OMIN}_k(V_c)$. In the special case where $V_{\rm as} = \{0\}$, 
or equivalently that $V$ is an operator system in the class discussed at the end of Section 
2, it turns out that this bad behavior goes away:

\begin{proposition} \label{isos}
    Let $V$ be a real operator system, and suppose that $V_{\rm as} = \{0\}$. Then $\text{OMAX}(V)_c = \text{OMAX}(V_c)$ and $\text{OMIN}(V)_c = \text{OMIN}(V_c)$. Consequently $\text{OMAX}(V)$ and $\text{OMIN}(V)$ are operator systems.
\end{proposition}

\begin{proof}   
    We will just prove the statement for $\text{OMAX}(V)$, the case of $\text{OMIN}(V)$ being nearly identical and left to the reader. 
    
    Note that since $V_{\rm as} = \{0\}$ it follows that
$(V_c)^+ = V^+$. 
Since the identity map $\text{OMAX}(V)_c \to \text{OMAX}(V_c)$ is ucp (see Proposition \ref{whdl}), it suffices to show that $\text{OMAX}(V)_c^+ = \text{OMAX}(V_c)^+ = V^+$ (i.e. the induced positive cones on $V_c$ agree). Once this is shown, it follows that the identity map is a complete order isomorphism since the matrix ordering on $\text{OMAX}(V_c)$ is the smallest 
archimedean  matrix ordering whose first positive cone agrees with $V^+$.

Let $z = x + iy \in \text{OMAX}(V)_c$ be self-adjoint, where $x, y \in V$. Since $V_{\rm as} = \{0\}$ we have  $y=0$, so $z \in \text{OMAX}(V)$.  However $\text{OMAX}(V)$  is completely order embedded in 
$\text{OMAX}(V)_c$ (see Proposition \ref{inco} and the Remark after it). 
    We conclude that $\text{OMAX}(V)_c^+ = V^+$. The final statement, that $\text{OMAX}(V)$ is an operator system, follows since $\text{OMAX}(V)$ is a real subsystem of the operator system $\text{OMAX}(V_c)$.
\end{proof}

{\bf Remark.} When $V_{\rm as} = \{0\}$ and $k > 1$, we do not know if $\text{OMAX}_k(V)_c = \text{OMAX}_k(V_c)$ or $\text{OMIN}_k(V)_c = \text{OMIN}_k(V_c)$.

\bigskip

Suppose that $V$ is a real $*$-vector space with positive cone $V^+$ and archimedean order unit $e$. If $V_{\rm as} = \{0\}$, then $\text{OMIN}(V)$ (or $\text{OMAX}(V)$) is an operator system, and thus $V$ admits a unital order embedding into an operator system. When $V_{\rm as} \neq \{0\}$, $\text{OMAX}(V)$ and $\text{OMIN}(V)$ are not operator systems. However, we have the following characterization.

\begin{theorem} \label{hasn}
    Let $V$ be a $*$-vector space with proper positive cone $V^+ \subseteq V_{\rm sa}$ and archimedean order unit. Then there exists a real operator system $\mathcal{S}$ and a unital order embedding $V \to \mathcal{S}$.   Conversely, 
    any real operator system is
    a $*$-vector space with positive cone $V^+ \subseteq V_{\rm sa}$ and archimedean order unit.
\end{theorem}

\begin{proof}
    Notice that if there exists a real operator system $\mathcal{S}$ and a unital order embedding $V \to \mathcal{S}$, then $V$ has an archimedean order unit, and 
    the norm on $\mathcal{S}$ restricted to $V_{\rm as}$ makes the latter a normed space. 
    
    Conversely, choose any norm $\| \cdot \|$ on $V_{\rm as}$ (we recall that any vector space has a norm via choosing a Hamel basis and using this to define a norm).  Define $D_n \subseteq M_n(V)_{\rm sa}$ by
    \[ D_n = \{ x = \alpha^* \text{diag}(p_1, \dots, p_n, q_1, \dots, q_m) \alpha \} \] where $p_1, \dots, p_n \in V^+$, $\alpha$ is a scalar matrix of appropriate size, and the $q_i$'s have the form \[ q_i = \begin{bmatrix} e & -x \\ x & e \end{bmatrix} \] where $x \in V_{\rm as}$ and $\|x\| \leq 1$. That $\{D_n\}$ is a matrix ordering can be seen by mimicking the proof of Proposition \ref{Prop91}.  To see that the order unit $e$ is a matrix order unit, it suffices to check that it is an order unit for $D_2$. If $[x_{ij}] \in M_2(V)_{\rm sa}$, then we may write 
    \[ [x_{ij}] = \begin{bmatrix} x_{11} & y - z \\ y + z & x_{22} \end{bmatrix} \] where $y \in V_{\rm sa}$ and $z \in V_{\rm as}$. Choose $t_1, t_2, t_3 > 0$ such that $t_1 e + x_{ii} + y, t_2 e - y \in V^+$ and $t_3 > \|z\|$. Then for $t = t_1 + 2t_2 + t_3$,
    \[ t I_2 \otimes e + [x_{ij}] = \begin{bmatrix} t_1 e + x_{11} + y & 0 \\ 0 & t_1 e + x_{22}+y \end{bmatrix} + \begin{bmatrix} 1 & -1 \\ -1 & 1 \end{bmatrix} \otimes (t_2 e - y) + t_2 \begin{bmatrix} 1 & 1 \\ 1 & 1 \end{bmatrix} \otimes e + \begin{bmatrix} t_3 e & -z \\ z & t_e e \end{bmatrix} \] is an element of $D_2$. So $e$ is a matrix order unit for $\{D_n\}$.

    Let $\{C_n\}$ denote the archimedean closure of $\{D_n\}$. First, we claim that $C_1 = V^+$. To see this, first notice that $V^+ \subseteq C_1$ trivially. Let $x \in C_1$. Then for each $\epsilon > 0$, $x + \epsilon e = \alpha^* \text{diag}(p_1, \dots, p_n, q_1, \dots, q_m) \alpha$ where $p_1, \dots, p_n \in V^+$, $\alpha$ is a scalar column matrix of appropriate size, and the $q_i$'s have the form \[ q_i = \begin{bmatrix} e & -x_i \\ x_i & e \end{bmatrix} \] with $x_i \in V_{\rm as}$ and $\|x_i\| \leq 1$, and $\alpha \in M_{N,1}(\bR)$ for $N = n+2m$. Since the compression of any $q_i$ by a real vector is a scalar multiple of $e$, we may assume without loss of generality that $x_i = 0$ for each $i$. It is now clear that $x + \epsilon e \in V^+$. Since $e$ is archimedean, $x \in V^+$.

    Finally, we claim that $\{C_n\}$ is proper, 
    so that $\mathcal{S} = (V,\{C_n\},e)$ is an abstract operator system.
     It suffices to show that $C_2$ is proper. Suppose that $\pm [x_{ij}] \in C_2$. As before, we may assume \[ [x_{ij}] = \begin{bmatrix} x_{11} & y - z \\ y + z & x_{22} \end{bmatrix} \] where $y \in V_{\rm sa}$ and $z \in V_{\rm as}$. By compressing to the 1-1 and 2-2 corners, we see that $\pm x_{11}, \pm x_{22} \in C_1 = V^+$. Since $V^+$ is proper, $x_{11}=x_{22} = 0$. Now compressing by $\alpha = \begin{bmatrix} 1 & 1 \end{bmatrix}$, we see that $\pm y \in V^+$, so $y = 0$. Now by the definition of $C_2$, \[ [x_{ij}] + \epsilon I_2 \otimes e = \begin{bmatrix} \epsilon e & - z \\ z & \epsilon e \end{bmatrix} = \alpha^* \text{diag}(p_1, \dots, p_n, q_1, \dots, q_m) \alpha \] where the $p_i$'s and $q_i$'s are as above and $\alpha$ has 2 columns. Write the columns of $\alpha$ as $(a_i)$ and $(b_i)$. Then $$\epsilon e = \sum_{i=1}^n a_i^2 p_i + \sum_{i=n+1}^{n+2m} a_i^2 e = \sum_{i=1}^n b_i^2 p_i + \sum_{i=n+1}^{n+2m} b_i^2 e,$$ and 
    \[ z = \sum_{i=1}^{n} a_i b_i p_i + \sum_{i=n+1}^{n+2m} a_i b_i e + \sum_{i=1}^{m} (a_{2i-1+n}b_{2i+n} - a_{2i + n}b_{2i-1+n})x_i \] where $\|x_i\| \leq 1$. Since each $p_i$ is positive,
    \[ 0 \leq a_i^2 p_i \leq \sum_{i=1}^{n} a_i^2 p_i = (\epsilon - \sum_{i=n+1}^{n+2m} a_i^2) e \leq \epsilon e \quad \text{and} \quad 0 \leq b_i^2 p_i \leq \sum_{i=1}^{n} b_i^2 p_i = (\epsilon - \sum_{i=n+1}^{n+2m} b_i^2) e \leq \epsilon e. \]
    It follows that $\sum_{i=n+1}^{n+2m} a_i^2, \sum_{i=n+1}^{n+2m} b_i^2 \leq \epsilon$. Also, since $z \in V_{\rm as}$ while $p_i, e \in V_{\rm sa}$, we have $$\sum_{i=1}^{n} a_i b_i p_i + \sum_{i=n+1}^{n+2m} a_i b_i e = 0$$ and thus $z = \sum_{i=1}^{m} (a_{2i-1+n}b_{2i+n} - a_{2i + n}b_{2i-1+n})x_i$. This shows that    \begin{eqnarray}
        \|z\| & = & \| \sum_{i=1}^{m} (a_{2i-1+n}b_{2i+n} - a_{2i + n}b_{2i-1+n})x_i \| \nonumber \\
        & \leq & \sum_{i=1}^m |a_{2i-1+n}b_{2i+n}| + |a_{2i + n}b_{2i-1+n}| \nonumber \\
        & = & \langle (|a_{n+1}|,\dots,|a_{n+2m-1}|),(|b_{n+2}|,\dots,|b_{n+2m}|) \rangle + \nonumber \\ & & + \langle (|a_{n+2}|,\dots,|a_{n+2m}|),(|b_{n+1}|,\dots,|b_{n+2m-1}|) \rangle \nonumber \\
        & \leq & 2 (\sum_{i=n+1}^{n+2m} a_i^2)^{1/2} (\sum_{i=n+1}^{n+2m} b_i^2)^{1/2} \nonumber \\
        & \leq & 2 \epsilon. \nonumber
    \end{eqnarray}
    Since $\|\cdot\|$ is a norm on $V_{\rm as}$ we have $z=0$. So $C_2$ is proper.
\end{proof}

In the proof, the matrix ordering can be adjusted by rescaling the norm on $V_{\rm as}$. This can be used to give another proof that there is no minimal or maximal operator system structure when $V_{\rm as} \neq \{0\}$.  
 We also remark that with more work one can show (a)\ that the norm $\| z \|$ for  $z \in V_{\rm as}$  which we chose very early in the proof,
 actually equals the (order unit) norm of $z$ in the operator system $\mathcal{S}$ which we found.  Moreover, (b)\ the intersection
 $\cap \, M_n(\mathcal{S})^+$ over all real operator system structures  $\mathcal{S}$ on $V$ with $\mathcal{S}^+ = V^+$ as in the 
 last theorem, exactly equals the cone $M_n({\rm OMAX}(V))^+$; that is, equals the cone $C_n^{1-max}$ defined at the start of Subsection  \ref{maosys}. 

\begin{corollary}
  Let $V$ be a real vector space with trivial (i.e.\ identity)  involution and proper positive cone $V^+ \subseteq V_{\rm sa}$ and archimedean order unit. 
  Then $\text{OMAX}(V)$ and $\text{OMIN}(V)$ are real operator systems.
\end{corollary}

\begin{proof}  By Theorem \ref{hasn}  we see that $V$ is an operator system, and then we may appeal to Proposition \ref{isos}.
\end{proof}

\section{Tensor products} \label{tens} 

Let $V$ and $W$ be real operator systems with units $e$ and $f$, respectively. A \textit{tensor product structure} (of real operator systems) consists of a 
proper matrix ordering $\{D_n\}_{n=1}^\infty$ on $V \otimes W$ such that
\begin{enumerate}
    \item $e \otimes f$ is an archimedean matrix order unit
    \item if $P \in M_n(V)^+$ and $Q \in M_k(W)^+$, then $P \otimes Q \in D_{nk}$, and
    \item if $\varphi: V \to M_n(\mathbb{R})$ and $\psi: W \to M_k(\mathbb{R})$ are ucp, then $\varphi \otimes \psi: V \otimes W \to M_{nk}(\mathbb{R})$ is ucp.
\end{enumerate}
The definition of tensor product structure for complex operator systems is similar, except that for (3) the maps range in complex matrix algebras.

Given real vector spaces $V$ and $W$, we have the identification $(V \otimes_{\bR} W)_c = V_c \otimes_{\bC} W_c$ 
as vector spaces.  For example, 
\[ (a+ib) \otimes (c+id) = (a \otimes c - b \otimes d) + i(a \otimes d + b \otimes c) \in (V \otimes W)_c. \]
Thus we may identify $V \otimes W$ with the real part of $V_c \otimes W_c$ with respect the induced conjugation.

Suppose $(V \otimes W, \{D_n\}, e \otimes f)$ is an operator system tensor product for real operator systems $V$ and $W$. Then Definition 1.2 defines a matrix ordering $\{\widetilde{D}_n\}$ on the complexification $(V \otimes W)_c$. Since $(V \otimes W)_c = V_c \otimes W_c$ as vector spaces, one can ask if the induced matrix ordering $\{\widetilde{D}_c\}$ defines a complex operator system tensor product structure on $V_c \otimes W_c$. The next Proposition says that it does.

\begin{proposition} \label{prop: complexify tens prod structure}
    Let $V$ and $W$ be real operator systems. Then any real tensor product structure on $V \otimes W$ induces a complex tensor product structure on $V_c \otimes W_c$. 
\end{proposition}

\begin{proof}
    We first check (1) that $e \otimes f$ is an archimedian order unit for $(V \otimes W)_c$ whenever it is for $V \otimes W$. Suppose that $x + iy + \epsilon I_n \otimes e \otimes f \in (D_n)_c$ for every $\epsilon > 0$. Then
    \[ \begin{bmatrix} x & -y \\ y & x \end{bmatrix} + \epsilon \begin{bmatrix} I_n & 0 \\ 0 & I_n \end{bmatrix} \otimes e \otimes f \in D_{2n} \]
    for every $\epsilon > 0$ by the definition of $(D_n)_c$. Since $e \otimes f$ is archimedean, it follows that \[ \begin{bmatrix} x & -y \\ y & x \end{bmatrix} \in D_{2n} \] and therefore $x+iy \in (D_n)_c$. So $e \otimes f$ satisfies the archimedean property. Similarly, if $x + iy$ is self-adjoint then we can choose $\epsilon > 0$ such that \[ \begin{bmatrix} x & -y \\ y & x \end{bmatrix} + \epsilon \begin{bmatrix} I_n & 0 \\ 0 & I_n \end{bmatrix} \otimes e \otimes f \in D_{2n} \] and therefere $x + iy + \epsilon I_n \otimes e \otimes f \in (D_n)_c$. So $e \otimes f$ is a matrix order unit.

    Next, we check (2) that whenever $P \in M_n(V_c)^+$ and $Q \in M_k(W_c)^+$ we have $P \otimes Q \geq 0$. Suppose $P+iQ \in M_n(V_c)^+$ and $S+iT \in M_k(W_c)^+$. Then $P=P^*, Q=-Q^*, S=S^*, T=-T^*$, and we have
    \[ \begin{bmatrix} P & -Q \\ Q & P \end{bmatrix} \in M_{2n}(V)^+, \quad \begin{bmatrix} S & -T \\ T & S \end{bmatrix} \in M_{2k}(W)^+.  \]
    We wish to show that $(P+iQ) \otimes (S + iT) = (P \otimes S - Q \otimes T) + i(Q \otimes S + P \otimes T) \in M_{nk}((V \otimes W)_c)^+$. Now
    \[ \begin{bmatrix} P & -Q \\ Q & P \end{bmatrix} \otimes \begin{bmatrix} S & -T \\ T & S \end{bmatrix} = \begin{bmatrix} P \otimes S & - P \otimes T & -Q \otimes S & Q \otimes T \\ 
    P \otimes T & P \otimes S & -Q \otimes T & -Q \otimes S \\
    Q \otimes S & - Q \otimes T & P \otimes S & -P \otimes T \\
    Q \otimes T & Q \otimes S & P \otimes T & P \otimes S \end{bmatrix} \]
    and conjugating this $4 \times 4$ block matrix by the scalar block matrix
    \begin{equation} \label{24mat} \frac{1}{\sqrt{2}} \begin{bmatrix} I & 0 & 0 & -I \\ 0 & I & I & 0 \end{bmatrix}  \end{equation} 
    (where $I = I_{nk}$) yields the matrix
    \[ \begin{bmatrix} P \otimes S - Q \otimes T & -(Q \otimes S + P \otimes T) \\ Q \otimes S + P \otimes T & P \otimes S - Q \otimes T \end{bmatrix} \in M_{2nk}(V \otimes W)^+. \]
    Hence $(P+iQ) \otimes (S + iT) \geq 0$.

    Finally we check (3). Suppose that $\varphi: V_c \to M_n(\mathbb{C})$ and $\psi: W_c \to M_k(\mathbb{C})$ are ucp. Composing with $\rho: \mathbb{C} \to M_2(\mathbb{R})$ given by $x + iy \mapsto \begin{bmatrix} x & -y \\ y & x \end{bmatrix}$, we obtain real ucp maps $\varphi' = \rho^{(n)} \circ \varphi$ and $\psi' = \rho^{(n)} \circ \psi$. The restrictions of $\varphi'$ and $\psi'$ to $V$ are still ucp, and hence $\varphi' \otimes \psi': V \otimes W \to M_{4nk}(\mathbb{R})$ is ucp. It follows that the complexification $\pi = (\varphi' \otimes \psi')_c: (V \otimes W)_c \to M_{4nk}(\mathbb{C})$ is ucp. Note that for any $x = a+ib$ and $y = c+id$,
    \[ \pi(x \otimes y) = \varphi' \otimes \psi' (a \otimes c - b \otimes d) + i \varphi' \otimes \psi' (a \otimes d + b \otimes c) = \varphi'_c(x) \otimes \varphi'(y). \]
    Let $T = 2^{-1/2} \begin{bmatrix} I_n & i I_n \end{bmatrix} \in M_{n, 2n}(\mathbb{C})$. Then $\varphi(x) = T \varphi'(x) T^*$ and $\psi(y) = T \psi'(y) T$ for every $x \in V_c$ and $y \in W_c$. Thus $\varphi \otimes \psi(x \otimes y) = (T \otimes T) \pi (x \otimes y) (T \otimes T)^*$ for every $x \in V_c$ and $y \in W_c$. It follows that $\varphi \otimes \psi$ is ucp, since $\pi$ is ucp.
\end{proof}

\bigskip

 {\bf Remark.} The converse of the last result is true and straightforward: any  complex tensor product structure on $V_c \otimes W_c$ induces a  real tensor product structure on $V \otimes W$. Indeed, the real axioms (1) and (2) are obvious by applying the complex axioms, and (3) follows since $\varphi: V \to M_n(\bR)$ and $\psi: W \to M_k(\bR)$ ucp implies that $\varphi_c: V_c \to M_n(\bC)$ and $\psi_c: W_c \to M_k(\bC)$ are ucp. So the restriction of $\varphi_c \otimes \psi_c$ to $V \otimes W$ is equal to $\varphi \otimes \psi$ and is ucp. 

\bigskip

Let $\mathcal{R}$ denote the category of real operator systems. A {\em real operator system tensor product} is map $\tau: \mathcal{R} \times \mathcal{R} \to \mathcal{R}$ which associates to each pair $(S,T)$ of real operator systems a real operator system tensor product structure $S \otimes_{\tau} T$ (i.e.\ a 
proper matrix ordering satisfying points (1), (2), and (3) above). A tensor product $\tau$ is {\em functorial} if it is a functor, i.e. whenever $\varphi: S_1 \to S_2$ and $\psi: T_1 \to T_2$ are ucp, then $\varphi \otimes \psi: S_1 \otimes_{\tau} T_1 \to S_2 \otimes_{\tau} T_2$ is ucp. 

In the next few subsections we consider several particular tensor products structures.   For each of these, $\alpha$ say, we will usually write 
$V \otimes_\alpha W$
for the uncompleted tensor product.  However later we will also often write 
$V \otimes_\alpha W$
for the completion, the completed tensor product.
We leave it to the reader to determine what is meant.   There is no ultimate danger of confusion since we are almost always henceforth simply finding the real versions of complex results already in the literature, which we reference.  Any confusion in the readers mind can be immediately dispelled by consulting that literature. At some points though  we will write e.g.\ $E \bar{\otimes} F$ for the closure of $E \otimes F$ in $V \otimes_\alpha W$, if $E,F$ are subspaces of $V, W$ respectively. 

\subsection{Minimal tensor product}

We define the {\em minimal tensor product} as follows: for real operator systems $V$ and $W$, we say $x \in M_n(V \minten W)^+$ if and only if $(\varphi \otimes \psi)^{(n)}(x) \geq 0$ for every pair of integers $m,k$ and ucp maps $\varphi: V \to M_m$ and $\psi: W \to M_k$.

\begin{proposition} \label{prop: min real into complex}
    Let $V$ and $W$ be real operator systems. Then the inclusion $V \minten W \to V_c \minten W_c$ (where the latter space is endowed with the complex minimal tensor product structure) is a real complete order embedding. 
\end{proposition}

\begin{proof} 
    Let $x \in M_n(V \otimes W)$. First suppose that $x \in M_n(V_c \minten W_c)^+$. Let $\varphi$ and $\psi$ be real matrix states. Then $0 \leq (\varphi_c \otimes \psi_c )^{(n)}(x) = (\varphi \otimes \psi)_c^{(n)}(x) = (\varphi \otimes \psi)^{(n)}(x)$. Hence $x \in M_n(V \otimes_{\rm min} W)^+$.

    Now suppose that $x \in M_n(V \minten W)^+$. Let $\varphi: V_c \to M_n(\mathbb{C})$ and $\psi: W_c \to M_k(\mathbb{C})$ be complex matrix states. By extending the ranges, we obtain $\varphi': V_c \to M_{2n}(\mathbb{R})$ and $\psi' : W_c \to M_{2n}(\mathbb{R})$. Hence $(\varphi' \otimes \psi')^{(n)}(x) \geq 0$. Conjugating by a scalar matrix (as in the proof of (3) for Proposition \ref{prop: complexify tens prod structure}), we see that $(\varphi \otimes \psi)^{(n)}(x) \geq 0$. So $x \in M_n(V_c \minten W_c)^+$.
\end{proof}

\begin{corollary}
    The minimal tensor product is a real tensor product structure. Moreover, if $\tau$ is another real tensor product, then the identity map $i: V \otimes_{\tau} W \to V \minten W$ is ucp for every pair of real operator systems $V, W$.
\end{corollary}

\begin{proof}
    For every $V$ and $W$, $V \minten W$ is a tensor product structure since its complexification is, by Proposition \ref{prop: complexify tens prod structure} and Proposition \ref{prop: min real into complex}. If $\tau$ is another real tensor product, then by condition (3) of the operator system tensor product structure $x \in M_n(V \otimes_{\tau} W)^+$ implies $x \in M_n(V \minten W)^+$. 
\end{proof}

It follows from Proposition \ref{prop: min real into complex} and the complex case that the real minimal tensor product of operator systems is exactly what one would expect to be, and it coincides on the operator systems with their real minimal operator space tensor product.  It follows immediately that it has all the expected properties: it is  {\em functorial, symmetric, associative,} and {\em complete order injective}.  Moreover, 
if $V$ and $W$ are real operator subsystems of $B(H)$ and $B(K)$ respectively then $V \otimes_{\rm min} W$ is canonically completely order embedded in the bounded operators on the  Hilbert space tensor product $H \otimes K$.

\begin{corollary} \label{comin} 
    Let $V$ and $W$ be real operator systems. Then  $(V \otimes_{\rm min} W)_c \cong  V_c \otimes_{\rm min} W_c$ as complex operator systems ($V_c \otimes_{\rm min} W_c$ is an (the) operator system complexification of $V \otimes_{\rm  min} W$). 
\end{corollary} 

\begin{proof}  It is proved in \cite[Lemma 5.4]{BReal} that the map  $\theta = \theta_V \otimes \theta_W$ is a period 2 conjugate linear completely isometric automorphism of $V_c \otimes_{\rm min} W_c$ with fixed point space the copy of  $V \otimes_{\rm min} W$.   However $\theta$ is unital, hence is a complete order isomorphism.   Hence the result follows from Proposition \ref{cco}. \end{proof}

\begin{corollary} \label{cocpmin}
    Let $V$ and $W$ be real operator systems with $V$ finite dimensional. Then  $(V \minten W)^+ = CP(V^d,W)$ via the canonical isomorphism $L : V \otimes W \cong {\rm Lin}(V^* ,W)$.  Indeed $[u_{i,j}] \in 
    M_n(V \otimes_{\rm min} W)^+$  if and only if the map 
    $V^d \to M_n(W): f \mapsto [L(u_{i,j})(f)]$ is completely positive.
\end{corollary}

\begin{proof}  As in \cite[Lemma 8.4]{KPTT2}. \end{proof}

We will need later the real case of the  technical result \cite[Theorem 5.1]{KPTT2}:

\begin{theorem} \label{tenp51}  Let $A$ and $B$ be unital C$^*$-algebras,  $\cS$ be an operator subsystem of $A$, and let $I$ be an ideal in $B$.
Then $I \bar{\otimes} \cS$ is a kernel in  $B \minten \cS$, and the operator system quotient $(B \minten \cS)/(I \bar{\otimes} \cS)$ coincides with the 
same operator space quotient.   Also the induced map from $(B \minten \cS)/(I \bar{\otimes} \cS)$ into $(B \minten A)/(I \bar{\otimes} A)$ is a unital complete order isomorphism onto its range. 
\end{theorem} 

\begin{proof} This follows by complexification.   Alternatively, the proof of \cite[Theorem 5.1]{KPTT2} works in the real case, subject to checking the real case  of three of the results used in that proof.  
The first of these is the operator space result \cite[Lemma 2.4.8]{Pisbk} (which is 7.43 in \cite{P}).  This follows 
by the same proof, or by complexification: one needs to prove that $(u_c)_Q = (u_Q)_c$ in Pisier's notation 
of \cite[Theorem 5.1]{KPTT2}.  We leave this as a laborious exercise for the reader.  

Next, the assertion ker $\phi = I \bar{\otimes} \cS$ in the proof. Write $\phi$ as $\phi_{\cS}$.  
Let $q_A : B \minten A \to Q(A) = (B \minten A )/(I \minten A )$  be the quotient map, in Pisier's notation. 
Then $\phi = q_A \circ (I_B \otimes i_{\cS})$, where $i_{\cS} : {\cS} \to A$ is the inclusion map.  
By a part of the  laborious complexification in the last paragraph, we may identify $(q_A)_c$ with $q_{A_c} : B_c \minten A_c 
\to Q(A_c).$     Then $$(\phi_{\cS})_c =  (q_A)_c  \circ (I_{B_c} \otimes i_{\cS_c})= (\phi_c)_{\cS_c} . $$  By the complex case of our present task
we have $$({\rm ker} \; \phi)_c = {\rm ker} \; \phi_c = I_c \bar{\otimes} \cS_c = (I  \bar{\otimes} \cS)_c .$$
 It follows that ker $\phi = I \bar{\otimes} \cS$. 
 
Finally there is an appeal to Corollary 4.2 
of \cite{KPTT2}, but this is essentially the same, and was discussed above in the real case  before our Definition
\ref{cobip}. \end{proof}

\subsection{Maximal tensor product} \label{maxtens}

We define the {\em maximal operator system tensor product} as follows: for real operator systems $V$ and $W$, we say $x \in M_n(V \maxten W)^+$ if and only if for every $\epsilon > 0$ there exist $A_1,\dots,A_k \in M_{n_i}(V)^+$, $B_1,\dots,B_k \in M_{m_i}(B)^+$, and an appropriate sized real scalar matrix $\alpha$ such that $x + \epsilon I_n \otimes e \otimes f = \alpha^T {\rm diag}(A_1 \otimes B_1, \dots, A_k \otimes B_k) \alpha$. 
As in the complex case this is a matrix ordering which is proper (since it is contained in the cones for $V \minten W$).  
If $e$ and $f$ are the identities of $V$ and $W$ then $e \otimes f$ is  a matrix order unit.   To see this it is enough to show $e \otimes f$ is an order unit in the complex case, however we need a 2-order unit in the real case. Suppose $x = \sum_i \, a_i \otimes b_i$ is self-adjoint in $V \otimes W$.  Without loss of generality, we may rewrite this sum as $\sum_i \, 
 a_i \otimes b_i + \sum a_i^* \otimes b_i^*$ (after rescaling). Choose scalars $t_i > 0$ such that $\|a_i\|, \|b_i\| \leq t_i$. Then $A_i := \begin{bmatrix} t_i e & a_i \\ a_i^* & t_i e \end{bmatrix}$ and $B_i := \begin{bmatrix} t_i f & b_i \\ b_i^* & t_i f \end{bmatrix}$ are positive. Notice that the corners of the positive matrix $A_i \otimes B_i$ have the entries $t_i^2 e \otimes f, a_i \otimes b_i$, and $a_i^* \otimes b_i^*$. Compressing by an appropriate scalar matrix, we get $2t_i^2 e \otimes f + a_i \otimes b_i + a_i^* \otimes b_i^* \geq 0$. Thus $x + \sum_i \,  t_i^2 e \otimes f$ is positive.

We now check that $I_2 \otimes (e \otimes f) = (e \otimes f)_{2}$ is an order unit for $M_2(V \maxten W)$.  
If $z \in M_2(V \otimes W)_{\rm sa}$ then $z_{ii}$ and 
$z_{12} + z_{21}$ are self-adjoint.
Indeed $$z = z_{11} \otimes e_{11} + z_{22}  \otimes e_{11} + (z_{12} \otimes e_{12} + z_{12}^*  
\otimes e_{21})$$ is a sum of three self-adjoint matrices.
 Thus it is enough to show that each of these three matrices is dominated by a positive multiple of $(e \otimes f)_{2}$.
It is easy to see that the first two are, by the last paragraph.    For the third matrix, since $z_{12}$ is a sum of rank 1 tensors, we may treat each of the latter seperately.  Thus we may assume that $z_{12} = a \otimes b$ for $a \in V, b \in W$. To see that $a \otimes b \otimes e_{12} + a^* \otimes b^*  \otimes e_{21}$ is dominated by a positive multiple of $(e \otimes f)_{2}$, we make a slight 
modification of the trick in the last paragraph.  Let $c = a \otimes e_{12} + a^*  \otimes e_{21}, d = b \otimes e_{12} + b^*  \otimes e_{21},$ and choose $t > 0$ such that
$$A = \begin{bmatrix} t e_2 & c \\ c^* & t e \end{bmatrix} \geq 0 , \; \; \; \; B = \begin{bmatrix} t f & d \\ d^* & t f \end{bmatrix}  \geq 0 .$$ 
Compressing, i.e.\ deleting  the 2nd and 3rd rows and columns of the positive matrix $A \otimes B$, 
we get $t^2 (e \otimes f)_2 + a \otimes b  \otimes e_{12}+ a^* \otimes b^*  \otimes e_{21} \geq 0$ as desired.
It follows that $e \otimes f$ is a matrix order unit for $V \maxten W$. 

We say that a bilinear map $u : V \times W \to B(H)$ is jointly completely positive if $u^{(n)}(x,y) = [u(x_{ij},y_{kl})]_{(i,k),(j,l)} \geq 0$ whenever 
$x = [x_{ij}] \in M_n(V)^+, y = [y_{kl}] \in M_n(W)^+$, and $n \in \bN$.   

\begin{lemma} \label{cjcp} Let $V$ and $W$ be real operator systems. If $u : V \times W \to B(H)$ is jointly completely positive, then 
its complexification $u_c : V_c \times W_c \to B(H)_c$ is jointly completely positive. 
\end{lemma}

\begin{proof} By the correspondence in (\ref{ofr}), $u_c(x + iy, w + iz)$ may be viewed as 
 $$u_c \left( \begin{bmatrix} x & -y \\ y & x \end{bmatrix} , \begin{bmatrix} w & -z \\ z & w \end{bmatrix} \right) = \begin{bmatrix} u(x , w) -u( y,z) & -u(x , z )- u(y,  w ) \\ u(x, z) +  u(y,  w ) & u(x , w) -u( y,z)  \end{bmatrix}.$$
  Let $x,y \in M_n(V), w,z \in M_n(W)$, and
$x + i y \geq 0$ and $w + i z \geq 0$, or equivalently $$\begin{bmatrix} x & -y \\ y &  x \end{bmatrix} \geq 0, \; \; \; \; \begin{bmatrix} w & -z \\ z &  w \end{bmatrix} \geq 0.$$ It follows that the displayed $4 \times 4$ matrix a line above equation (\ref{24mat}) is positive when we replace the 16 $\otimes$ symbols by $u(\cdot,\cdot)$.  As in that argument,  conjugating this 4 × 4 block matrix by the scalar block matrix in (\ref{24mat})  is positive.  However this 
is $(u_c)_n(x + iy, w + iz)$ by the equation at the start of the proof.  Thus  $u_c$ is jointly completely positive.  \end{proof} 
  
{\bf Remark.} Unlike the situation for linear maps, jointly completely contractive bilinear maps with  $u(1,1) = 1$ need not be  jointly completely positive. 
For example the product on $B(H)$ fails to be jointly positive. 

\bigskip 

By Subsection \ref{archs}, we see that the matrix ordering on the maximal tensor product $V \maxten W$ is precisely the 
archimedeanization of the matrix ordering generated by the set of elementary tensors $\{P \otimes Q\}$ where $P \in M_n(V)^+$ and $Q \in 
M_m(W)^+$. Recall that we saw that the real
archimedeanization of a space V is the smallest set of cones turning V into an operator system as
in the complex case.  Indeed  Lemmas 2.5, 2.6 and 5.1
of \cite{KPTT1} hold in  the real case, and also Theorem 5.8 there (the universal property of 
$V \maxten W$).   The latter in our case is:

\begin{theorem} \label{mtenup}  Let $V$ and $W$ be real operator systems.  Every 
jointly completely positive map $u : V \times W \to B(H)$, for a real Hilbert space $H$,
linearizes 
to a completely positive map $V \maxten W \to B(H)$.  Moreover if $\tau$ is a 
real operator system structure on $V \otimes W$ such that every 
jointly completely positive map $u : V \times W \to B(H)$  linearizes 
to a completely positive map $V \otimes_{\tau} W \to B(H)$, then $V \otimes_{\tau} W = V \maxten W$.
\end{theorem} 

\begin{proof} The proof is the same as in \cite[Theorem 5.8]{KPTT1}.   This result appeals to Lemmas 2.5, 2.6 and 5.1
of \cite{KPTT1}, but these are valid in the real case too as we just said.  \end{proof} 

{\bf Remark.}  Because of this universal property, one could have defined $V \maxten W$ in terms of a supremum  over unital jointly completely positive maps. 

\smallskip

It is then easy to see  that $V \maxten W$ is completely order  embedded in 
$V_c \maxten W_c$.    
Indeed we now show that  $$E = {\rm Span} \, \{ x \otimes y : x \in V, y \in W \} \subseteq V_c \maxten W_c$$ 
has the desired universal property of 
$V \maxten W$.   If $u : V \times W \to B(H)$ is jointly completely positive, then by the lemma 
its complexification $u_c : V_c \times W_c \to B(H)_c$ is jointly completely positive.   The latter linearizes 
to a completely positive map $V_c \maxten W_c \to B(H)_c$.
This in turn restricts to a completely positive map 
$E \to B(H)$.  Thus by the real case of  \cite[Theorem 5.8]{KPTT1} mentioned above,  $E$  has the desired universal property of 
$V \maxten W$.   That is,  $V \maxten W$ is completely order  embedded in $V_c \maxten W_c$.  This also follows with a different proof from the next result. 
 
\begin{theorem} \label{mtenf}
Let $V$ and $W$ be real operator systems. Then  $(V \otimes_{\rm max} W)_c \cong  V_c \otimes_{\rm max} W_c$ as complex operator systems ($V_c \otimes_{\rm max} W_c$ is an (the) operator system complexification of $V \otimes_{\rm  max} W$). \end{theorem}

\begin{proof}
We apply the (complex) universal property of 
 $\maxten$, to the canonical 
 unital complex bilinear map $u : V_c \times W_c \to (V \maxten W)_c$. 
 Note  that $u$ is jointly completely positive.  Indeed 
 $u$ may be viewed as the map 
 $$m \left( \begin{bmatrix} x & -y \\ y & x \end{bmatrix} , \begin{bmatrix} w & -z \\ z & w \end{bmatrix} \right) = \begin{bmatrix} x \otimes w - y \otimes z & -x \otimes z - y \otimes w  \\ x \otimes z + y \otimes w & x  \otimes w - y \otimes z \end{bmatrix}.$$
Suppose that  $x,y \in M_n(V), w,z \in M_n(W)$, and
$x + i y \geq 0$ and $w + i z \geq 0$, or equivalently $$\begin{bmatrix} x & -y \\ y &  x \end{bmatrix} \geq 0, \; \; \; \; \begin{bmatrix} w & -z \\ z &  w \end{bmatrix} \geq 0.$$  In the proof of Proposition \ref{prop: complexify tens prod structure} we gave a trick that shows that
  $$\begin{bmatrix} x \otimes w - y \otimes z & -x \otimes z - y \otimes w  \\ x \otimes z + y \otimes w & x  \otimes w - y \otimes z \end{bmatrix}  \geq 0.$$    Thus $u$ is  jointly completely positive, so linearizes to a ucp map
  $\kappa : V_c \maxten W_c \to (V \maxten W)_c$. 
  
  Similarly, the canonical map  $V \times W \to V_c \maxten W_c$ is easily seen to be jointly completely positive, so linearizes to a ucp map $\iota : V \maxten W \to V_c \maxten W_c$.
  By Lemma \ref{ncomma} the map $\iota : V \maxten W \to V_c \maxten W_c$
 extends to a complex ucp map $\hat{\iota} : (V \maxten W)_c \to V_c \maxten W_c$. 
 It is easy to see that $\kappa$ is the inverse 
  of $\hat{\iota}$.   Thus the result is proved.
  \end{proof}

{\bf Remarks.}  1)\ It is obvious as in the complex case that there are canonical bijective correspondences between the  jointly completely positive functionals  $\psi : \cS \times \cT \to \bR$, and the completely positive real linear maps
$\cS \to \cT^d$, and the completely positive real linear maps
$\cT \to \cS^d$
(See Lemma 3.2 of \cite{Lan}, \cite[Lemma 5.7]{KPTT1}).

\smallskip

  2)\ If $A, B$ are real unital C$^*$-algebras then (the completion of) $A \maxten B$ with the usual product is a real C$^*$-algebra (whose canonical operator space structure is compatible with its 
  operator system structure) which has the following universal property: for all unital $*$-representations $\pi : A \to B(H)$ and 
  $\rho : B \to B(H)$ with commuting ranges, there exists a unique $*$-homomorphism 
  $\theta : A \maxten B \to B(H)$ such that $\theta(a \otimes b) = \pi(a) \rho(b)$.
  This follows fairly straightforwardly from the facts that  $A \maxten B \subseteq A_c \maxten B_c$ as operator systems, and that the latter is the (complex) C$^*$-algebraic maximal tensor product \cite[Theorem 5.12]{KPTT1}.  It follows that this embedding is a real faithful $*$-homomorphism.   We leave the remaining arguments, which are standard by now,  to the reader. 

  In the light of this, if $A, B$ are real unital C$^*$-algebras then  we call $A \maxten B$ the {\em maximal real C$^*$-algebra tensor product}.

\begin{lemma} \label{maxex} {\rm (Max-exactness \cite[Corollary 5.17]{KPTT2})}\  
For a real operator system $\cS$, and 
 $I$  an ideal in  a real  unital C$^*$-algebra $A$, we have $$(\cS \maxten A)/(\cS \bar{\otimes} I) \cong \cS \maxten (A/I)$$ unitally complete order isomorphically. 
\end{lemma}

\begin{proof} By complexification and Proposition \ref{coker} 
it follows that $\cS \bar{\otimes} I$ is a kernel in $\cS \minten A$. 
 By the functoriality of the maximal tensor product we get a ucp map 
 $\cS \maxten A \to \cS \maxten A/I$  with kernel containing $\cS \bar{\otimes} I$, and we get an induced  ucp map
$$(\cS \maxten A)/(\cS \bar{\otimes} I) \to \cS \maxten (A/I).$$
This is a complete order isomorphism since its complexification is easily argued to correspond to the 
complex complete order isomorphism  
$$(\cS_c \maxten A_c)/(\cS_c \bar{\otimes} I_c) \to \cS_c \maxten (A_c/I_c)$$
from \cite[Corollary 5.17]{KPTT2}. 
\end{proof}

 {\bf Remarks.} 1)\ As expected, in the real case $\maxten$ is still  {\em functorial}: whenever $\varphi: \cS_1 \to\cS_2$ and $\psi: \cT_1 \to \cT_2$ are ucp, then $\varphi \otimes \psi: \cS_1 \maxten \cT_1 \to \cS_2 \maxten \cT_2$ is ucp.   This can be seen by the usual proof or by complexifying. 

Similarly by complexifying and using Theorem \ref{mtenf} and the complex analogues, 
one sees that $\maxten$ is a real tensor product structure, and is  {\em symmetric, associative,} and {\em complete order  projective}. 
By the latter we mean that whenever $\varphi: \cS_1 \to\cS_2$ and $\psi: \cT_1 \to \cT_2$ are (ucp) complete order quotient maps (that is, they induce a unital complete order isomorphism after quotienting by the kernel), then $\varphi \otimes \psi: \cS_1 \maxten \cT_1 \to \cS_2 \maxten \cT_2$ is a complete order quotient map.   This can be seen e.g.\ by complexifying
\cite[Proposition 3.2]{Htp} (see also \cite[Proposition 1.12]{FP}).

\medskip

2)\ In \cite{Oz} Ozawa checked the real case of Farenick and Paulsen's duality relation $$(V \otimes_{\rm max} W)^d  \cong V^d \otimes_{\rm min} W^d$$ for finite dimensional real operator systems.   Indeed it is easy to argue this by complexifying both sides using Theorems \ref{mtenf},  \ref{dcom} and Corollary \ref{comin}.

\medskip

3)\ It is checked in \cite{KPTT1} that for complex operator spaces $X, Y$, the 
 operator space structure inherited on $X \otimes Y$ by its embedding into 
the  maximal operator system tensor product of the Paulsen systems $\cS(X)$ and $\cS(Y)$, agrees with their maximal operator space tensor product, sometimes called the  operator space projective tensor product \cite{BP}.  The same will be true in the real case by complexification.  Indeed since we have a complete order isomorphism $\cS_{\bR}(X)_c = \cS_{\bC}(X_c)$ (see 
\cite{Sharma,BReal}), and similarly for $Y$, we will have from Theorem \ref{mtenf}
and \cite[Lemma 5.4]{BReal} that 
$$(\cS_{\bR}(X) \maxten \cS_{\bR}(Y))_c \cong \cS_{\bR}(X)_c \maxten  \cS_{\bR}(Y)_c = \cS_{\bR}(X)_c \hat{\otimes}    \cS_{\bR}(Y)_c
= (\cS_{\bR}(X) \hat{\otimes} \cS_{\bR}(Y))_c.$$
It is not true of course that the  maximal operator system tensor product of 
operator systems agrees with their operator space projective tensor product.
For example we know that $M_n \maxten M_n \cong M_{n^2}$, which is very different to $M_n \hat{\otimes} M_n$.
 
\subsection{Commuting tensor product} 

Let $V$ and $W$ be real operator systems. If $\Phi :   V \to B(H)$ and 
  $\Psi : W \to B(H)$ are completely positive maps with  commuting ranges, then the bilinear map 
  $\Phi(x) \Psi(y)$ for $x \in V, y \in W$ is easily seen to be jointly completely positive.  This follows since the multiplication on $B(H)$ is  jointly completely positive on commuting operators.  Indeed we leave it as an exercise that if $A$ and $B$ are commuting $C^*$-subalgebras of $B(H)$, and if $a$ and $b$ are positive matrices with entries in $A$ and $B$ respectively, then $[a_{ij} b_{kl}] \geq 0$.

  The bilinear map in the last paragraph linearizes to 
  a map $V \otimes W \to B(H)$ which we write as $\Phi \cdot \Psi$.
  As in the complex case for each integer $n$ we define a cone in $M_n(V \otimes W)$ to be those matrices $z$ for which $(\Phi \cdot \Psi)^{(n)}(z) \geq 0$ for all such $H, \Phi, \Psi$.
  Following the arguments in \cite[Section 6]{KPTT1}, with these cones $V \otimes W$ becomes an operator system 
  $V \cten W$, {\em the (maximal) commuting tensor product}.  It is easy to see (similarly to the minimal tensor product in an earlier section) that it is a real tensor product structure. 
  
Theorem 6.7 of \cite{KPTT1} states that $A \cten \cS = A \otimes_{\rm max} \cS$
for any unital C$^*$-algebra $A$.  Since Theorem 15.12 in \cite{Pnbook} was   checked in Section \ref{rcast} above, the proof of \cite[Theorem 6.7]{KPTT1}   works to give $A \cten \cS = A \otimes_{\rm max} \cS$ in the real case too.

\begin{corollary} \label{cocin}
    Let $V$ and $W$ be real operator systems. Then 
  $V \cten W   \subseteq C^*_{\rm u}(V) \maxten C^*_{\rm u}(W)$ as real operator systems (that is, unitally complete order embedded). 
  Indeed $V \cten W   \subseteq C^*_{\rm u}(V) \maxten W$.
\end{corollary}

\begin{proof} The first follows as in the complex case \cite[Theorem 6.4]{KPTT1} from the fact (immediate from the lines above) that $C^*_{\rm u}(V) \maxten C^*_{\rm u}(W)
= C^*_{\rm u}(V) \cten C^*_{\rm u}(W)$.  
Similarly for the second, or use the simple argument in 
\cite[Proposition 3.4]{Kavruk}. \end{proof}

\begin{corollary} \label{coc}
    Let $V$ and $W$ be real operator systems. Then  $(V \cten W)_c \cong  V_c \cten W_c$ as complex operator systems ($V_c \cten W_c$ is an (the) operator system complexification of $V \cten W$).
\end{corollary}

\begin{proof} This follows from Corollary \ref{cocin} and Theorem \ref{mtenf},
since 
$$(C^*_{\rm u}(V) \maxten C^*_{\rm u}(W))_c \cong C^*_{\rm u}(V)_c \maxten C^*_{\rm u}(W)_c
\cong C^*_{\rm u}(V_c) \maxten C^*_{\rm u}(W_c).$$
Since $(V \cten W)_c$ and $V_c  \cten W_c$ are complete order embedded
in the first and third of these, they are equal.  
\end{proof}

It is easy to see e.g.\ by complexification that $\cten$ is {\em functorial}: whenever $\varphi: \cS_1 \to\cS_2$ and $\psi: \cT_1 \to \cT_2$ are ucp, then $\varphi \otimes \psi: \cS_1 \cten \cT_1 \to \cS_2 \cten \cT_2$ is ucp.  

\begin{corollary} \label{codu}
    Let $V$ and $W$ be real operator systems. Then 
  $V \cten W   \subseteq V^{**} \cten W$ as real operator systems. 
\end{corollary}

\begin{proof} This follows from  the complex case \cite[Corollary 6.6]{KPTT2}.
The latter yields a complete order embedding $$V_c \cten W_c   \subseteq (V_c)^{**} \cten W_c
\cong (V^{**})_c \cten W_c.$$  Then restrict to the copies of $V \cten W$
and $V^{**} \cten W$ in these spaces.
\end{proof}

 {\bf Remarks.} 1)\  By complexifying and using Corollary \ref{coc} and the complex analogue, 
one sees that $\cten$  is  a {\em symmetric} tensor product.  In \cite{Kavruk} on page 13, Kavruk states that it is unknown if the commuting tensor product is associative. We  do not know if there have been any updates on this.   

\medskip

2)\ It is checked in \cite{KPTT1} that for complex operator spaces $X, Y$, the 
 operator space structure inherited on $X \otimes Y$ by its embedding into 
the  maximal commuting  operator system tensor product of the Paulsen systems $\cS(X)$ and $\cS(Y)$, agrees with their $\mu^*$  operator space tensor product as defined above Proposition 6.9 in \cite{KPTT1}.  The same will be true in the real case by complexification.  Indeed as in Remark 3) at the end of the last subsection, we have from Corollary \ref{coc} and Proposition 6.9 in \cite{KPTT1} that  
$$(\cS_{\bR}(X) \cten \cS_{\bR}(Y))_c \cong \cS_{\bR}(X)_c \cten  \cS_{\bR}(Y)_c = \cS_{\bC}(X_c) \cten    \cS_{\bC}(Y_c),$$ which contains $X_c \otimes_{\mu^*} Y_c$.  
The latter of course has the universal property of the $\mu^*$  tensor product, and this may be used to show that the subspace $X \otimes Y$ has the matching universal property of the real version of the $\mu^*$  tensor product.  We leave the details to the reader.

\subsection{Injective tensor products}

We define $\elten$ and $\erten$ as in the complex case via the embeddings
$$V \otimes_{\rm el} W \subseteq I(V) \maxten W, \; \; \; V \otimes_{\rm er} W \subseteq V \maxten I(W).$$
That is we give $V \otimes W$ the operator system structure coming from these algebraic embeddings. 

\begin{corollary} \label{el}
    Let $V$ and $W$ be real operator systems. Then  $(V \otimes_{\rm el} W)_c \cong  V_c \otimes_{\rm el} W_c$ as complex operator systems ($V_c \otimes_{\rm el} W_c$ is the operator system complexification of $V \otimes_{\rm el} W$).
\end{corollary}

\begin{proof} This is similar to the proof of Corollary \ref{coc}, but using 
 $V \otimes_{\rm el} W   \subseteq I(V) \maxten W$ in place of 
Corollary \ref{cocin}.  We have 
$$(I(V) \maxten W)_c \cong I(V)_c \maxten W_c \cong I(V_c) \maxten W_c$$
   which contain both $(V \otimes_{\rm el} W)_c$ and $V_c  \otimes_{\rm el} W_c$.
\end{proof}

Similarly we have:

\begin{corollary} \label{er}
    Let $V$ and $W$ be real operator systems. Then  $(V \otimes_{\rm er} W)_c \cong  V_c \otimes_{\rm er} W_c$ as complex operator systems.
\end{corollary}

It is easy to see by complexification that $\elten$ and $\erten$ are left and right injective respectively, and are both functorial tensor structures as we earlier defined these, but neither is symmetric. The latter follows from the complex case by the complexifications above. 

As in the complex case $\elten$ and $\erten$ are the largest functorial tensor products which are left and right injective respectively.
It is also clear that $V \otimes_{\rm er} W = W  \otimes_{\rm el} V$. 

\medskip

{\bf Remark.}\ It follows from  the complex case, and from the  complexification results above, that ${\rm min}$ is dominated by 
${\rm el}$ and ${\rm er}$, and each of the latter are dominated by 
${\mathfrak c}$. Of course  
${\rm max}$ is larger than all of these.

 \section{Stability of nuclear related properties under complexification} \label{stabil}

 It is often obvious that if $\cS_c$ is `nuclear' (that is, has the Completely Positive Factorization Property (CPFP)) / is exact / has the Local 
 Lifting Property (SLLP)/ has the weak expectation property (WEP) / has the double-commutant expectation property (DCEP), then $\cS$ has these properties, or vice versa, by simply complexifying. 
  (We will define all these terms precisely later, but they will be the real versions of the well known corresponding complex notions.)
  For example, 
 it will follows that the quaternions have all these properties, since $\bH_c = M_2(\bC)$ does.  
   It is clear by complexifying that if $\cS$ has the real WEP (resp.\  DCEP, CPFP) then $\cS_c$ has complex WEP (resp.\  DCEP, CPFP).    (For the DCEP
  one needs to check that $(\cS_c)'' = (\cS'')_c$, but this is an exercise.) 

We will see that each of the five properties above, denoted by (P), have a tensor product characterization just as in the complex theory, as does C$^*$-nuclearity and QWEP.
(We say that a real $C^*$-algebra 
$A$ has real QWEP if it is $*$-isomorphic to a quotient of a real WEP C$^*$-algebra.) 
  Suppose that $\cS_c$ satisfies such a criterion in the complex case, e.g.\ if $\cS_c \otimes_\alpha \cT_c = \cS_c \otimes_\beta \cT_c$
  for two of the tensor structures $\alpha, \beta$ studied in Section \ref{tens} (beginning with $\minten$).  Since we proved that these contain 
  $\cS \otimes_\alpha \cT$ and $\cS \otimes_\beta \cT$ respectively we see that $\cS \otimes_\alpha \cT = \cS \otimes_\beta \cT$. 
  This usually immediately implies  that if $\cS_c$ satisfies complex (P) then $\cS$ satisfies real (P), by the tensor product characterizations.
  
  Conversely suppose that   $\cS \otimes_\alpha \cT = \cS \otimes_\beta \cT$ for two of the tensor structures $\alpha, \beta$.
  Then it follows that $\cS_c \otimes_\alpha \cT_c = \cS_c \otimes_\beta \cT_c$.  However an issue presents itself:
  to use the complex tensor product characterization
  of complex (P), we really need $\cS_c \otimes_\alpha W = \cS_c \otimes_\beta W$ for {\rm all} complex $W$ in the appropriate class,
  not just the $W$ which are complexifications.   Fortunately there is a trick for this.  Namely consider the sequence
  $$\cS_c \otimes_\alpha W \to \cS_c \otimes_\alpha W_c \to \cS_c \otimes_\alpha W$$
  induced by a  
  certain canonical sequence 
  $W \to W_c \to W$.    The latter maps are the complex linear maps $w \mapsto (w,0) \in W \oplus \bar{W} \cong W_c$
   (see the second paragraph of Section \ref{rcast}), and the cp projection $W_c \to W: (w , \bar{z}) \mapsto w$.   By functoriality of the tensor product $\otimes_\alpha$ we deduce that
 the canonical map $\cS_c \otimes_\alpha W \subseteq \cS_c \otimes_\alpha W_c$ is a 
 complete order embedding.   Similarly for $\beta$.  Thus if $\cS_c \otimes_\alpha W_c = 
  \cS_c \otimes_\beta W_c$ then $\cS_c \otimes_\alpha W = \cS_c \otimes_\beta W$.  
  This trick fixes the issue above, and so we usually immediately have that if $\cS$ satisfies real (P) then $\cS_c$ satisfies complex (P) by the tensor product characterizations.

 \subsection{The CPFP} We say that a real operator system $\cS$ has the CPFP 
 (Completely Positive Factorization Property)
 if there is a net of integers $(n_t)$ and ucp maps $u_t : \cS \to M_{n_t}(\bR)$ and $v_t  : M_{n_t}(\bR)
 \to \cS$ such that $v_t \circ u_t \to I_{\cS}$ strongly (so  $v_t (u_t(x)) \to x$ for all $x \in \cS$). 
 \[
\begin{tikzcd}
& M_{n_t}(\bR) \arrow[rd, "v_t" black] \\
\cS \arrow[rr, "I_{\cS}"] \arrow[ru, "u_t"]  & & \cS
\end{tikzcd}
\]

\bigskip

 A real unital C$^*$-algebra which has the CPFP will be called {\em nuclear}. 
 
 \begin{theorem} \label{ComplCPFP}  A real operator system $\cS$ has the (real) CPFP
 if and only if $\cS_c$ has the (complex)  CPFP.  This is also equivalent to: 
 $$\cS \minten \cT = \cS \maxten \cT$$ for all real operator systems $\cT$, or for all finite dimensional real operator systems $\cT$.
\end{theorem}

\begin{proof}  We already deemed it elementary,  by complexifying, that if $\cS$ has the real CPFP  then $\cS_c$ has complex CPFP.   Conversely, suppose that $\cS_c$ has the CPFP.  Compose the maps into $\cS_c$ (resp.\ from $\cS_c)$ in the diagram for the complex CPFP
 with the projection $\cS_c \to \cS$ (resp.\  the inclusion $\cS \to \cS_c$).  Also  
 recall that $M_n(\bC)$ may be viewed as a complemented real subsystem  of $M_{2n}(\bR)$.
 All of this then gives a factorization diagram for $i_{\cS} : \cS \to \cS$ through 
 $M_{2n_t}(\bR)$ showing that $\cS$ has the CPFP. 

 For the remaining assertions we can use the principles discussed in and after  the second paragraph of Section \ref{stabil}.  We have $\cS_c \minten \cT_c = (\cS \minten \cT)_c$. If $\cS_c$ has the CPFP
 then  $$\cS_c \minten \cT_c = \cS_c \maxten \cT_c$$ for real 
  operator systems $\cT$.  Hence `by restriction' we have $\cS \minten \cT = \cS \maxten \cT$.   
  
  Conversely, suppose that $\cS \minten \cT = \cS \maxten \cT$ for all finite dimensional real operator systems $\cT$.  We obtain  by complexifying as above that 
  $\cS_c \minten \cT_c = \cS_c \maxten \cT_c$.
  the principle just referred to implies $\cS_c$ has the CPFP.   An alternative proof: 
  in particular
  the last equality is true when $\cT$ is the dual of any finite dimensional subsystem $W$ of $\cS$.   We then have $\cT_c = (W^d)_c = (W_c)^d$.  The proof of \cite[Theorem 3.1]{HP} can then be slightly adjusted with the directed set of spaces $E$ there replaced by the directed set of the spaces $W_c$ above, to yield that $\cS_c$ has the CPFP. 
\end{proof}

A real operator system $\cS$ is called (real) $C^*$-nuclear if $\cS \minten B = \cS \maxten B$ for all real C$^*$-algebras $B$.

\begin{corollary} \label{csnd}  
A real operator system $\cS$ is real C$^*$-nuclear   if and only $\cS_c$ is complex C$^*$-nuclear.
This is also equivalent to  $\cS \minten \cT = \cS \cten \cT$ for all operator systems $\cT$.  A 
finite dimensional  real operator system $\cS$ is real C$^*$-nuclear   if and only 
if $\cS^d$ is real C$^*$-nuclear.
\end{corollary}

\begin{proof} It follows from the principles discussed in and after  the second paragraph of Section \ref{stabil}  that $\cS$ is real C$^*$-nuclear if and only if $\cS_c$ is complex C$^*$-nuclear.

 The second assertion is Proposition  4.11 in \cite{Kavruk} in the complex case, and  is also valid in the real case
 by complexifying just as in similar previous arguments above.

The last `if and only 
if'  follows from the matching complex result of Kavruk \cite{Kav3}.  So 
$\cS$ is real C$^*$-nuclear   if and only if $\cS_c$ is complex C$^*$-nuclear, and  if and only 
if $(\cS_c)^d \cong (\cS^d)_c$ is complex C$^*$-nuclear.  So  if and only if $\cS^d$ is real C$^*$-nuclear. 
\end{proof} 

The famous Choi-Effros characterization of nuclear $C^*$-algebras in the real case follows exactly as Kavruk did in the complex 
case immediately from Theorem \ref{ComplCPFP}, and from Corollary \ref{cocin} and the line before it.  
For a real $C^*$-algebra $A$ has CPFP if and only 
if $A \minten \cT = A \maxten \cT$ for all real operator systems $\cT$, and the latter follows if it is true
for real $C^*$-algebras by a simple argument from the aforementioned corollary and line.

 \subsection{The WEP}

We say that an operator system $\cS$ has the {\em weak expectation property} or $\text{WEP}$ if the inclusion $i: \cS \to \cS^{dd}$ extends to a ucp map $\hat{i}: I(\cS) \to \cS^{dd}$, where $\cS^{dd}$ denotes the operator system double-dual and $I(\cS)$ denotes the injective envelope of $\cS$. 
\[
\begin{tikzcd}
I(\cS) \arrow[dashrightarrow, rd, "\hat{i}" black] \\
\mathcal \cS \arrow[u, hook] \arrow[r, "i" black, hook] & \cS^{dd}
\end{tikzcd}
\]
 
\bigskip

\begin{theorem} \label{Complwep}  A real operator system $\cS$ has the WEP
 if and only if $\cS_c$ has the WEP.  This is also equivalent to: 
 $$\cS \elten \cT = \cS \maxten \cT$$ for all real operator systems $\cT$, or for all finite dimensional real operator systems $\cT$.
\end{theorem}

\begin{proof}  We already deemed it elementary,  by complexifying, that if $\cS$ has the real WEP  then $\cS_c$ has complex WEP.   The latter implies by the complex theory \cite{KPTT2} that 
$\cS_c \elten \cT_c = \cS_c \maxten \cT_c$, and hence as in the last proof that $\cS \elten \cT = \cS \maxten \cT$, for all real operator systems $\cT$. 

Henceforth suppose that  $\cS \elten \cT = \cS \maxten \cT$  for all finite dimensional real operator systems $\cT$.   That this  implies that $\cS_c$ has the WEP follows from the principles discussed in and after  the second paragraph of Section \ref{stabil}. 
Alternatively, it follows by the idea for the analogous implication in Theorem \ref{ComplCPFP}.   We obtain  by complexifying  that 
  $\cS_c \minten \cT_c = \cS_c \maxten \cT_c$.  We then choose  finite dimensional subsystems $W$ of $\cS$ so that $E_\lambda = W_c$ work in the  proof of \cite[Theorem 4.1]{Htp}  to yield that $\cS_c$ has the complex WEP. 
 
A similar proof (the real case 
of the proof of \cite[Theorem 4.1]{Htp})
shows that $\cS$ has the real WEP.
  That proof appeals to Lemma 5.7 of \cite{KPTT1}, which  holds in the real case (see Remark 1 after Theorem \ref{mtenf} in Section \ref{maxtens}). 
\end{proof}

{\bf Remark.}  Indeed if $\cS$ is a real C$^*$-algebra then WEP is also equivalent to $\cS \minten \cS_2 = \cS \maxten \cS_2$, where $\cS_n$ is the span of the identity and the first $n$ generators of $C^*(\bF_n)$ and their adjoints.  This follows by  complexifying and appealing to the complex case of this result  \cite{Kavruk, Kavrukthesis}.

\begin{corollary} \label{wepvn}  An injective real operator system has the real WEP.   A real operator system which is an (operator system) bidual, 
or even is a dual operator space
(or equivalently by {\rm Theorem \ref{dual}} is a dual operator system),  has the real WEP if and only if it is completely order isomorphic to a real injective real von Neumann algebra.  In particular a finite dimensional 
real operator system $\cS$ has the real WEP if and only if it is real injective, or if and only if it is completely order isomorphic to a real C$^*$-algebra.  
The latter is also equivalent (assuming  $\cS$ finite dimensional) to $\cS$ (or equivalently its dual) 
having real CPFP, or to $\cS \elten \cS^d = \cS \maxten \cS^d$.  
\end{corollary}

\begin{proof} 
The second statement follows from the proof, and an adaption of the proof, of  the complex case  \cite[Proposition 6.10 and Theorem 6.11]{KPTT2}.  Indeed if $\cS$ 
is a bidual real operator system then $\cS$ has the real WEP
if and only if $\cS_c$ has the complex WEP, and if and only if $\cS_c$
is complex injective.   If $\cS$ merely is a dual operator space, then one may adapt the proof in \cite[Proposition 6.10]{KPTT2} replacing $\cS^*$ in that proof by the predual operator space to get a 
 weak* continuous completely contractive projection $\cS^{**} \to \cS$.
This is necessarily unital so ucp.  The rest of the argument is the same. 
Consequently $\cS$ has the real WEP  if and only if $\cS$
is real injective \cite{BCK}.  
Ruan showed that
a real injective operator system is a C$^*$-algebra.  Since it has a predual it is a von Neumann algebra \cite{Li}.  

Similarly, the first  statement follows by complexifying and using the C$^*$-algebra case.
Note that every finite dimensional real C$^*$-algebra is injective \cite{Li}. 
Finally, the last assertions also follow from  the complex case by complexification.     
\end{proof}

 \subsection{The DCEP}

 An operator system $\cS$ has the 
 {\em double commutant expectation property} or {\em DCEP} if every unital complete order embedding $\varphi: \cS \to B(H)$ extends to a ucp map $\varphi': I(\cS) \to \varphi(\cS)''$, where $\varphi(\cS)''$ denotes the double commutant of $\varphi(\cS)$ inside $B(H)$.
 \[
\begin{tikzcd}
I(\cS) \arrow[dashrightarrow, rd, "\varphi' " black] & \\
\mathcal \cS \arrow[u, hook] \arrow[r, "\varphi" black] & \varphi(S)'' \arrow[r, hook] & B(H)
\end{tikzcd}
\]

\bigskip

\begin{theorem} \label{Compldcep}  A real operator system $\cS$ has the real DCEP
 if and only if $\cS_c$ has the complex DCEP.  This is also equivalent to: 
 $$\cS \elten \cT = \cS \cten \cT$$ for all real operator systems $\cT$, and is equivalent to: 
 $$\cS \elten C^*(\bF_\infty) = \cS \maxten C^*(\bF_\infty),$$
 and is equivalent to the unital complete order embedding  $\cS \maxten B \subseteq A  \maxten B$
 for all unital real C$^*$-algebras $A$ and $B$ with 
 $\cS$ a subsystem of $A$.
\end{theorem}

\begin{proof}  We already mentioned that,  by complexifying, if $\cS$ has the real DCEP  then $\cS_c$ has complex DCEP.   The latter implies by the complex theory \cite{KPTT2} that 
$\cS_c \elten \cT_c = \cS_c \cten \cT_c$, and hence as in the last proofs that $\cS \elten \cT = \cS \cten \cT$, for all real operator systems $\cT$.  As in the proof of Theorem \ref{ComOLLP} for the analogous result concerning $B(H)$, and using Proposition \ref{gpalg}, it is easy to argue that $\cS \elten C^*(\bF_\infty) = \cS \maxten C^*(\bF_\infty)$ if and only if $\cS_c \elten C^*_{\Cdb}(\bF_\infty) = \cS_c \maxten C^*_{\Cdb}(\bF_\infty).$ 
By the complex theory this is equivalent to $\cS_c$ having the  complex DCEP.  And the latter by a similar proof  implies the last statement of the theorem.  To see the converse, that the last statement of the theorem
implies the real DCEP, we checked the real case of that part of the proof of 
Theorem  7.1 of \cite{KPTT2}.  This relies on several previous results of those authors which need to be checked in the real case, such as the variant of Theorem  6.4 of \cite{KPTT1} which we presented in Corollary \ref{cocin}. 

We will be done if we can show that $\cS \elten \cT = \cS \cten \cT$  for all real operator systems $\cT$, implies $\cS$ has DCEP.   
However the proof that (5) implies (4) of Theorem  7.3 of \cite{KPTT2} works in the real case. 
Alternatively, we can use the principles discussed in and after  the second paragraph of Section \ref{stabil}. 
\end{proof}

{\bf Remarks.}  1)\ Again by complexifying, a real C$^*$-algebra $A$ has the 
real DCEP if and only if it has the real WEP.

2)\ Similarly, an operator system $\cS$ has the real DCEP if and only if 
$\cS \minten (\bR^5/I) = \cS \cten (\bR^5/I)$, using \cite[Theorem 5.9]{Kavrukw}.
Here $I$ is as in the discussion before Theorem \ref{realK}.

  \subsection{The SLLP} 

An operator system $\cS$ has the {\em operator system local lifting property} or {\em SLLP} if whenever $\varphi: \cS \to A/I$ is ucp, where $A$ is a C*-algebra and $I \subseteq A$ is an ideal, the restriction $\varphi_0$ of $\varphi$ to any finite-dimensional subsystem $\cS_0 \subseteq \cS$ admits a ucp lifting $\widetilde{\varphi}_0: \cS_0 \to A$ such that $\varphi_0 = q \circ \widetilde{\varphi}_0$, where $q: A \to A/I$ is the canonical quotient map.
\[
\begin{tikzcd}
& & A \arrow[d, "q" ] \\
\cS_0 \arrow[urr, "\widetilde{\varphi_0}"] \arrow[r, hook] & \cS \arrow[r, "\varphi" black] & A/I
\end{tikzcd}
\]
  (We write SLLP as opposed to the more common OSLLP to distinguish it from Ozawa's earlier operator space local lifting property.)

\bigskip
The operator system local lifting property (SLLP) 
 may be defined with or without the liftings being ucp (these definitions are equivalent
 as is explained in \cite[Remark 8.3]{KPTT2}, and the real case is the same). 
 For a real C$^*$-algebra  the real SLLP will be called the LLP.

 \begin{theorem} \label{ComOLLP}  A real operator system $\cS$ has the real SLLP if  and only if 
  $\cS_c$ has the complex SLLP.  These are also equivalent to  
  $$\cS \minten B(H) = \cS \maxten B(H)$$ for all real  Hilbert spaces $H$, and equivalent to
 $$\cS \minten \cT = \cS \erten \cT$$ for all real operator systems $\cT$.   In particular, every $C^*$-nuclear real operator system $\cS$ has the real SLLP. 
\end{theorem}

\begin{proof}  If  $\cS_c$ has the SLLP, then in the SLLP diagram we obtain a lifting $\cS_0 \subseteq (\cS_0)_c \to A_c$ making the SLLP diagram for $\cS_c$ commute.  If we then compose the latter lifting $\cS_0  \to A_c$
 with the projection $A_c \to A$ we obtain a desired lifting, since
 $$(q_I)_c \, (P( \widetilde{\varphi_0} (s_0))) = P((q_I)_c (\widetilde{\varphi_0} (s_0))) =
 P (\varphi_c(s_0)) =  \varphi_c(s_0)$$ for $s_0 \in \cS_0 \subseteq \cS$.  So 
  $\cS$ has the SLLP. 
  
  We also know $\cS_c$ has the SLLP if and only if 
  $$\cS_c \minten B_{\bC}(K) = \cS_c \maxten B_{\bC}(K)$$ for all complex  Hilbert spaces $K$.  We may write $K = H_c$ for a real  Hilbert space $H$, then 
  $\cS_c \minten B_{\bC}(K) = (\cS \minten B(H))_c$.   Similarly we saw that 
  $\cS_c \maxten B_{\bC}(K) = (\cS \maxten B(H))_c$.  Thus $\cS_c \minten B_{\bC}(K) = \cS_c \maxten B_{\bC}(K)$ for all complex  Hilbert spaces $K$, if and only if 
   $\cS \minten B(H) = \cS \maxten B(H)$ for all real Hilbert spaces $H$.
   Conversely, the latter condition implies $\cS \minten \cT = \cS \erten \cT$ for all real operator systems $\cT$, by the same simple argument as in the complex case \cite[Theorem 8.1]{KPTT2}. 

   It is similarly now clear by complexification that if $\cS_c$ has the SLLP  then
  $\cS \minten \cT = \cS \erten \cT$ for all real operator systems $\cT$. 
  The latter condition in turn implies that  $\cS \minten B(H) = \cS \maxten B(H)$, since $B(H)$ is injective.    
  
   We will be done if we can show that $\cS \erten \cT = \cS \minten \cT$  for all real operator systems $\cT$, implies $\cS$ has the SLLP. 
   This can be proved via the above and  the principle discussed in and after  the second paragraph of Section \ref{stabil}. We also give an alternative proof in which we check several other important results, and this may be valuable.  We observe that  the proof of (2) implies (1) in \cite[Theorem 8.5]{KPTT2} works in the real case too, providing that the real cases of  Theorem 5.1 and Proposition 5.3 in 
  \cite{KPTT2} hold.  
  The first of these we discussed in the real case 
  in Theorem \ref{mtenup}.

  The real case of Proposition 5.3 in 
  \cite{KPTT2} should state that $I \minten \cS$ is a completely biproximinal kernel in $A \minten \cS$, for any real operator system and  real $C^*$-algebra quotient $A/I$. 
  The proof of the complex case in \cite{KPTT2}   relies on Lemma 4.8 of \cite{KPTT2}, the real case of which we checked in Proposition \ref{neolp},
  and Lemma 2.4.7 in \cite{Pisbk} (which is 7.44 in \cite{P}).  The latter is true also in the real case as may be seen by complexifying all the spaces, and using the simple principle essentially seen in the proof of Corollary \ref{proxiff}.
  Namely that 
  if $q : E \to F$ is a surjective contraction such that $q_c$ is a 1-quotient map then $q$ is a 1-quotient map.   Indeed if $y \in F$ with $\| y \| = 1$, choose $x_1 + i x_2 \in E_c$ with $q_c(x_1 + i x_2) = y$ and $\| x_1 + i x_2 \| = 1.$  Here $x_i \in E$ so that $\| x_1 \| \leq \| x_1 + i x_2 \|  = 1$.
  Also $q(x_1) = y$ and so $1 \leq \| x_1 \|$.  So $\| x_1  \| = 1.$
 \end{proof}
 
 {\bf Remark.} We will need later that if $B$ and $C$ are  
 real C$^*$-algebras, or more generally real operator systems, with respectively the real WEP and real SLLP, then 
 we have 
 as in the complex case  that 
 $B \minten C = B \maxten C$.  
 Indeed since $C$ has the real SLLP we have 
 $$A \minten B = B \minten A = B \erten A = A \elten B = A \maxten B,$$
 the last equality because $B$ has real WEP.  

 \bigskip

We thank N. Ozawa for the following argument.

  \begin{theorem} \label{oz}  Suppose that  $V, W$ are real (resp.\ complex) operator systems with real (resp.\ complex) SLLP.
   Then $V \oplus^\infty W$ has real (resp.\ complex) SLLP.
\end{theorem}

\begin{proof}  Suppose that $u : V \oplus^\infty W \to A/I$ is ucp, and that $E$ is a finite dimensional subspace of $V \oplus^\infty W$.
We may assume that $E = V_0  \oplus^\infty W_0$ for unital subsystems $V_0 \subseteq V, W_0  \subseteq W$.
Consider $\varphi_1 = u_{|V_0}$.   This lifts to a completely positive map  $\psi_1 : V_0 \to A$ such that $q_I \circ \psi_1 = \varphi_1$.
Similarly $u_{|W_0}$  lifts to a completely positive map  $\psi_2 : V_0 \to A$ such that $q_I \circ \psi_2 = u_{|W_0}$.
Then $(v,w) \mapsto  \psi_1(v) + \psi_2(w)$ is completely positive and lifts $u_{|E}$. 
\end{proof}

 \subsection{Exactness} Let $\cS$  be a real operator system, $A$ a real  unital C$^*$-algebra and $I$  an ideal in $A$. Then by complexification it follows that $\cS \minten I$ is a kernel in $\cS \minten A$ (these are the completed minimal tensor products).
 By the functoriality of the minimal tensor product we get a ucp map 
 $\cS \minten A \to \cS \minten A/I$ with kernel containing $\cS \minten I$, and an induced ucp map
 $$(\cS \minten A)/(\cS \minten I) \to \cS \minten (A/I).$$
 We say that $\cS$ is {\em exact} if this map is bijective (which is true always if $\cS$ is finite dimensional, by complexifying) and is a complete order isomorphism.   (This is called 1-exactness in \cite{KPTT2}.

 \begin{theorem} \label{Comexa}  A real operator system $\cS$ is exact  if and only if
  $\cS_c$ is exact. In the definition of exactness we may take $A = \bB = B(\ell^2_{\bR})$ and $I$ the compact operators  $\bK$ here.   Exactness  is also equivalent to 
  $$\cS \minten \cT = \cS \elten \cT$$ for all real operator systems $\cT$.
\end{theorem}

\begin{proof}  Complexifying we have 
$$((\cS \minten A)/(\cS \minten I))_c \cong (\cS_c \minten A_c)/(\cS_c \minten I_c),$$
and $$(\cS \minten (A/I))_c \cong \cS_c \minten (A_c/I_c) .$$
Thus (looking at the appropriate subspace identifications) if $\cS_c$ is exact, 
so that $$(\cS_c \minten A_c)/(\cS_c \minten I_c) \cong \cS_c \minten (A_c/I_c),$$ then 
$(\cS \minten A)/(\cS \minten I) \to \cS \minten (A/I)$ and $\cS$ is exact. 
Conversely, if $\cS$ is exact then by complexification we have 
$$(\cS_c \minten A_c)/(\cS_c \minten I_c) \cong \cS_c \minten (A_c/I_c).$$
It is proved in \cite[Corollary 7.3]{Kavruk} that in the definition of (complex) exactness we may take $A = B_{\bC}(\ell^2_{\bC})$ and the ideal of compact operators.  These are both complexifications, with $A/I = (\bB/\bK)_c$.  Thus we see that $\cS_c$ is exact. 

As in previous proofs, if $\cS_c$ is exact then by the complex case 
$$\cS_c \minten \cT = \cS_c \elten \cT$$ for all complex operator systems $\cT$,
and so $\cS \minten \cT = \cS \elten \cT$ for all real operator systems $\cT$.
The converse can be proved via the above and  the principle discussed in and after  the second paragraph of Section \ref{stabil}. Alternatively, if $\cS \minten \cT = \cS \elten \cT$ for all real operator systems $\cT$,
then the proof of (2) $\Rightarrow$ (3) in \cite[Theorem 5.7]{KPTT2} works in the real case to show that $\cS$ is exact.
 \end{proof}

Thus as in the complex case, the CPFP (nuclearity) is the strongest of these five properties.   The WEP implies the DCEP, but the converse is false in general except for C$^*$-algebras (even false for subsystems of $M_n$, see examples below). Indeed there are no further  implications between these properties in general. The SLLP does not imply the  DCEP by the counterexample to Kirchberg's conjecture.  Neither the WEP nor DCEP implies the SLLP or exactness (Consider $\bB$).  (Even for subsystems of $M_n$), exactness does not imply the DCEP, see examples below), nor does it imply the LP 
(Lifting Property) or SLLP (\cite[Corollary 10.14]{Kavruk}).

As we saw, (in the real as in the complex case) $C^*$-nuclearity implies SLLP and DCEP. 
As in the complex system case (and similarly to the $C^*$-algebra and operator algebras cases see e.g.\ 
\cite{BD}) it is clear that exactness and DCEP together are equivalent to $\cS \minten \cT = \cS \cten \cT$ for all operator systems $\cT$, and hence  together are equivalent to C$^*$-nuclearity by  Corollary \ref{csnd}.  
$C^*$-nuclearity does not  imply WEP  (see below). It follows that none of the five properties except CPFP  implies $C^*$-nuclearity.
For a C$^*$-algebra WEP and DCEP agree, and imply QWEP.

  \begin{proposition} \label{410} Real exactness passes to real operator subsystems. That is, if $\cS$  is exact then every operator subsystem of $\cS$ is exact. Conversely, if every finite dimensional operator subsystem of $\cS$ is exact then $\cS$ is exact.
\end{proposition}

\begin{proof}  The first assertion follows by complexification of Proposition 4.10  in \cite{Kavruk} (or Corollary 5.8 in \cite{KPTT2}), or by the same proof.   Similarly for the second (if one complexifies one needs to choose the subspace $\cS_0$ in $\cS_c$ to be a complexification, but there is no difficulty doing that by enlarging it).
  \end{proof} 

\begin{corollary} \label{corduliex}
   An operator system is exact 
   as an operator system if and only if it is exact 
   as an operator space. 
\end{corollary}

\begin{proof} Complexify and use \cite[Proposition 5.5]{KPTT2} and \cite[Proposition 4.3 (2)]{RComp}.  
\end{proof}

  As is well known, exactness may be characterized in terms of a local embeddability into   subspaces of $M_n$ (and in the system case this can be done with completely positive maps, see e.g.\ \cite{Pisbk}). This is valid in the real case too with the same proofs 
  (see e.g.\ above Proposition 4.3 in \cite{RComp}).
  In the real case if $X$ is exact then $X$ is locally reflexive.  Indeed by that  Proposition 4.3  $X_c$ is complex exact, so complex locally reflexive.
  So $X$ is real locally reflexive.  We have not tried, but it seems that similar considerations should immediately derive from the known 
  complex case the real analogue of the existence
  of a universal separable 1-exact operator system.

\begin{corollary} \label{comeco}
    A finite dimensional real operator system $\cS$ has the lifting property if and only if $\cS^d$ is exact (and vice versa). 
\end{corollary}

\begin{proof} We saw that $\cS$ has the lifting property if and only if $\cS_c$ does. By Theorem 6.6 in \cite{Kavruk} this is equivalent to $(\cS_c)^d \cong (\cS^d)_c$ being
exact.  By Theorem \ref{Comexa}  this holds  if and only if $\cS^d$  is exact.  The `vice versa' is similar. 
\end{proof}

As in Theorem 6.8 and Proposition 7.4 in \cite{Kavruk} we have: 

\begin{corollary} \label{cornul}
    Let $\cS$ be a finite dimensional real operator system with null subspace $J$. If $\cS$  has the lifting property then  $\cS/J$ does too.   
\end{corollary}

\begin{corollary} \label{cordulif}
    Let $\cS$ be a finite dimensional real operator system. Then $\cS$  has the lifting property if and only if every ucp map $\cS \to \bB / \bK$ has a ucp lift to $\bB$.    
\end{corollary}

{\bf Remarks.}\ 1)\ One of the big problems in the operator system theory is if every 3 dimensional complex operator system is exact, or equivalently has the lifting property?  (This is false for 4 dimensions \cite{Kavruk}.)   A negative solution to this problem immediately follows from the above if we can find a 3 dimensional real operator system which is not exact, or equivalently does not have the lifting property.  This may be much easier, since computing with 3 real variables (which often means with 2 variables after scaling) is easier than with 6 (i.e.\ 3 complex variables). 

This is related to the real version of the Smith-Ward problem: namely if for every $T \in B_{\bR}(H)$ there is a compact real linear operator 
with $w_n(T+K) = w_n(\overset{\cdot}{T})$ for all $n \in \bN$.
Here $w_n(x) = \{ \varphi(x) : \varphi \in {\rm ucp}(\cS,M_n) \}$ where $x \in \cS$.
The methods and results of Paulsen and Kavruk (in e.g.\ \cite[Section 11]{Kavruk})
all work in the real case. 
In particular, the 
real version of the Smith-Ward problem is equivalent to every 3 dimensional real operator system being exact, or equivalently having the lifting property.  Indeed it is equivalent to the real case of Paulsen's criterion; i.e.\ to that for every 3 dimensional real operator 
subsystem $\cS$ of the Calkin algebra $\bQ = \bB_{\bR}(\ell^2)/\bK_{\bR}(\ell^2)$, the inclusion $\cS \subseteq \bQ$ has a ucp lift $\cS \to \bB_{\bR}(\ell^2)$.

Finally we believe that it is also not known if we can find a 3 dimensional real or complex operator system which is not C$^*$-nuclear. We plan to examine some 3 dimensional real examples where  calculations
are feasible.

\medskip

2)\ The real case of \cite[Theorem 6.10]{Kavruk} is the same as the complex.  That is, Todorov's  proof works to yield:
a  finite dimensional operator system $\cS$ has the LP  if and only if $C^*_u(\cS)$ has LLP. 

\bigskip

 {\bf Examples.}   
 We check the real case of many examples from the complex tensor product system literature.  We will sometimes silently be using results above, and note that the results in \cite[Proposition  5.2]{Kavrukw} and  \cite[Corollary 9.6]{KPTT2} are just the same in the real case, as is the helpful
\cite[Lemma  5.2]{Kavruk}.  

 Real $\cS_n$ is not exact for $n \geq 2$ in the real case (by appealing to the complex case in \cite[Corollary 10.13]{Kavruk}).  Hence it is not CPFP nor has WEP (although this is obvious from Corollary
 \ref{wepvn}).  It does have the real lifting property (see the proof of Theorem \ref{realK}, where this is used). Whether it has DCEP is settled in the negative by the solution to the Kirchberg conjectures (see below).  Thus it is not $C^*$-nuclear, hence neither is the dual of $\cS_n$.
 Note $S_n^d$ is
 computed explicitly in \cite{FP} in the complex case, and the same works in the real case to exhibit the  dual of $\cS_n$ as a real space 
 of very simple matrices.  In particular $S_2^d$ is the span in $M_4(\bR)$ of $I$ and $\{ E_{12},  E_{21}, E_{34}, E_{43} \}$,
 hence the latter is a matrix system that is not $C^*$-nuclear nor DCEP.  Also $S_n$ is identifiable as in the complex case \cite{FP} with the real case of the space $T_n/J_n$ studied there.
 As in \cite[Section 10]{Kavruk}, the real  tridiagonal matrices $T_n$ is real C$^*$-nuclear, so has real DCEP, but not real WEP by e.g.\ Corollary
 \ref{wepvn}.  Similarly for the space of $3 \times 3$ real matrices with zero in the $1$-$3$ and $3$-$1$ entries
 (the real version of the system considered in \cite[Theorem 5.16]{KPTT1}).  By complexifying and appealing to the complex case in the latter reference, this system is real C$^*$-nuclear, hence is exact, SLLP and DCEP, but it does not have real WEP.   Similarly the real version of the 
 chordal graph $C^*$-algebras $\cS_G$ considered after the just cited theorem, are real C$^*$-nuclear. 
 
 Similarly, the real version $C_{\bR}(S^1)^{(n)}$ of the Connes-van Suijlekom-Farenick Toeplitz system \cite{CvS,F} of real $n \times n$ Toeplitz matrices clearly has complexification $C_{\bC}(S^1)^{(n)}$, so it is exact and has DCEP (the latter by \cite[Theorem 6.26]{F}).
 Thus it and its dual are $C^*$-nuclear and SLLP (but not WEP).   Its dual is the real span $C_{\bR}(S^1)_{(n)}$ in $C(\Tdb)$ of the monomials $z^k$ for $k = 0, \cdots, n-1$.  Indeed the complexification of the last space  is $C_{\bC}(S^1)_{(n)}$,  and the canonical map $\rho : C_{\bR}(S^1)^{(n)}  \to (C_{\bR}(S^1)^{(n)})^d$ is the `restriction' of the duality pairing  
  $\phi : C_{\bC}(S^1)^{(n)} \to (C_{\bC}(S^1)^{(n)})^d$ considered in \cite[Proposition 4.6]{CvS}, which is a complete order isomorphism.  Indeed it is easy to see that $\rho_c = \phi$ if we view  $(C_{\bC}(S^1)^{(n)})^d = ((C_{\bR}(S^1)^{(n)})^d)_c$,  
  so that $\rho$ is a complete order isomorphism.

 If $I = \bR (1,1,-1,-1,-1)$ then 
 $\bR^5/I$  does not have the real DCEP, hence not the real CPFP nor WEP, and is not exact (c.f.\ \cite[Proposition 6.1]{Kavrukw}), but it has the real lifting property.   Similarly for the quotient of $M_n(\bR)$ by $J$, the diagonal matrices with trace zero.  
 The complexification of $J$ is
 the complex diagonal matrices with trace zero, which was shown in \cite{FP} to be a kernel.  Thus $J$ is a kernel.  The complexification of $M_n/J$ 
 is the quotient of $M_n(\bC)$   by the complex diagonal matrices with trace zero.  This is known to have the lifting property, hence $M_n(\bR)/J$ has the real lifting property.  However for $n = 3$ it does not have the DCEP \cite[Theorem 10.2]{Kavruk},  hence not the real CPFP nor WEP.  
 Note that  $M_n(\bR)/J$ may be identified with  the obvious real version $\cW_n^{\bR}$ of the space $\cW_n$ considered in \cite{FP}.   Indeed the real variant  on
 $M_n(\bR)$ of the map
 $\phi$ in \cite[Theorem 2.6]{FP} clearly has complexification $\phi$ and kernel $J$, so is a complete quotient map inducing 
 $M_n(\bR)/J \cong \cW_n^{\bR}$.  
 As in \cite{FP} it is clear that  $\cW_n^{\bR}$ contains $\cS_n$, so it is not exact.  
 
  Pop considers $\cS_{\ell^n_\infty(\bC)}$ for $n \geq 5$ \cite{Pop}, and shows it does not have DCEP, hence not the real CPFP nor WEP nor the LP (and that the LP does not imply WEP).   However it is exact.  The same will be true for 
 $\cS_{\ell^n_\infty(\bR)}$.  
 Similarly there should be a real version of the span of the generators of the Cuntz algebra $\cO_n$,  which is a subquotient of $M_n$.  These can be expected as in the complex case to have all five properties except WEP and CPFP.  We have not checked this.  

 The noncommutative cubes of \cite{FKPT} are complexifications of real systems ${\rm nc}_{\bR}(n)$.  
  By definition ${\rm nc}_{\bR}(n)$ is the real span of $1$ and the copies of the
   (selfadjoint) generator of $\bZ_2$ in $C^*_{\bR}(\star_{j=1}^n \, \bZ_2)$.
These generators generate $C^*_{\bR}(\star_{j=1}^n \, \bZ_2)$, so ${\rm nc}_{\bR}(n)$ `contains enough unitaries'.  
So ${\rm nc}_{\bR}(n)$ is $n+1$-dimensional and it has operator system complexification ${\rm nc}_{\bC}(n)$.  It also has the universal property
in  \cite[Proposition 4.1 (2)]{FKPT}, and is a quotient of $\cS_n$, and of a tridiagonal matrix system $T_{n+1}$ as in \cite[Proposition 4.2]{FKPT}.
It is also a quotient of $\bR^{2n}$ via the map in \cite[Theorem 5.5]{FKPT}.   Indeed the complexification of this real map $r$ is exactly 
the map  in \cite[Theorem 5.5]{FKPT}, so the real map is ucp and a real complete quotient map (it is an exercise to see
that if $u$ is a real ucp map with $u_c$ a  complete quotient map then $u$ is a  complete quotient map).    The kernel $J_n$ of this quotient map is the $n-1$ dimensional space span of $$\{ (1,1, -1,-1, 0 ,\cdots , 0), (1,1, 0,0, -1,-1, 0 ,\cdots , 0), \cdots ,  (1,1, 0, \cdots , 0, -1,-1) \},$$ 
which is a null subspace.    Thus the dual $(\bR^{2n}/J_n)^d$ is identifiable with $J_n^\perp \subseteq (\bR^{2n})^d$, which in turn 
is identifiable with $\cR_{\bR} \subseteq \bR^{2n}$. Here $\cR_{\bR}$ is the real version of the system $\cR$ in  \cite[Theorem 5.9]{FKPT}.  Clearly its complexification is $\cR$.   If $r : \bR^{2n} \to {\rm nc}_{\bR}(n)$ is the real complete quotient map above then 
$r^d : {\rm nc}_{\bR}(n)^d \to (\bR^{2n})^d = \bR^{2n}$ is a unital coi.   Moreover $(r^d)_c = r^d$. 
Also ${\rm nc}_{\bR}(3)$ is 
not WEP since it is not   a (4 dimensional) real $C^*$-algebra. 
We have ${\rm nc}_{\bR}(3) \minten {\rm nc}_{\bR}(3) = {\rm nc}_{\bR}(3) \erten {\rm nc}_{\bR}(3)$, but this is not ${\rm nc}_{\bR}(3) \cten {\rm nc}_{\bR}(3)$, or else the same would be true for ${\rm nc}_{\bC}(3)$ which is false
by \cite[Theorem 4.9]{FKPT}.  
This also shows  ${\rm nc}_{\bR}(3)$ does not  have DCEP.     
We also know 
that ${\rm nc}_{\bR}(3)$ is not $C^*$-nuclear, hence its dual is not either.  Hence its dual is not  DCEP.
We thank Sam Harris for informing us of Kavruk's result that 
${\rm nc}_{\bR}(3)$ is a WEP detector, and its dual is a nuclearity detector \cite{Kav3}.   Thus ${\rm nc}_{\bR}(3)^d \minten \bB \neq 
{\rm nc}_{\bR}(3)^d \maxten \bB$, so ${\rm nc}_{\bR}(3)^d$ does not have the lifting property, so ${\rm nc}_{\bR}(3)$ is not exact.

  All two dimensional real operator systems of course have all the properties, indeed are isomorphic to $\ell^\infty_2(\bR)$.  We leave this as an exercise analogous to the complex case.

\begin{corollary} \label{comcpwep} A complex operator system has the complex CPFP (resp.\ WEP, DCEP, exactness, SLLP) if and only if it has the real
CPFP (resp.\ WEP, DCEP, exactness, SLLP).  
\end{corollary}

\begin{proof} Suppose that $\cS$ is a complex exact operator system.  Then it is easy to see from the definition of exactness that $\bar{\cS}$ and 
$\cS_c \cong \cS \oplus \bar{\cS}$ are complex exact.   So $\cS$ is real exact. 

Similarly if $\cS$ is a complex operator system with SLLP.   It is easy to see that $\bar{\cS}$ has complex SLLP.    By Theorem \ref{oz} 
$\cS_c \cong \cS \oplus \bar{\cS}$  has complex SLLP.    So $\cS$ has real SLLP.   

A similar argument probably works for other properties here.  Alternatively, the first  is because complex ucp maps are real ucp. Also  
 recall that $M_n(\bC)$ may be viewed as a complemented real subsystem  of $M_{2n}(\bR)$.
 All of this then gives a real factorization diagram for $i_{\cS} : \cS \to \cS$ through 
 $M_{2n_t}(\bR)$ showing that $\cS$ has the real CPFP. 

 The second is because the real bidual equals the complex bidual,
complex ucp maps are real ucp, and the complex injective envelope is a real
injective envelope.

The third is similar to the second.  We 
observe that $B_{\bC}(H) \subset B_{\bR}(H)$.  It is easy to see that the real bicommutant of $\cS$ in the $B_{\bR}(H)$ is the complex bicommutant in $B_{\bC}(H).$

For the converses, let (P) be any one of the five properties.  If a complex operator system $\cS$ has the real (P)
then $\cS_c$ has complex (P). 
However since $\cS$ is complex complemented in 
$\cS_c$ we see by the tensorial characterization of (P), and functoriality, that $\cS$ has complex (P).
 \end{proof}

We leave the proof of the following as an exercise for the reader, using in part some of the principles we have used many times above.

 \begin{corollary} \label{rip}  Selfadjoint elements $(s_i)$ and $(t_i)$  in a real operator system $\cS$ have any one of the Riesz interpolation properties in {\rm \cite[Proposition 6.6--6.10]{Kavrukw}} if and only if they have the same property in $\cS_c$.
\end{corollary} 

The proof of \cite[Theorem 7.2]{Kavrukw} works in the real case for 
the real version of 
the $(k,m)$-tight Riesz interpolation property in $B$ in {\rm \cite[Definition 7.1]{Kavrukw}}.  It follows easily that a  real unital C$^*$-algebra $A$ has the $(k,m)$-tight Riesz interpolation property in $B$  if and only if 
$A_c$ has that property in $B_c$.   Since Kavruk showed that the latter is equivalent to  $B_c$ having complex WEP, we deduce: 

 \begin{corollary} \label{rip2}  A  real unital C$^*$-algebra $A$
 has the complete $(2,3)$-tight (resp.\ complete $(k,m)$-tight for all integers $k,m$) Riesz interpolation property in $B(H)$ 
 in {\rm \cite[Definition 7.1]{Kavrukw}}  if and only if $A$ has the real WEP. \end{corollary} 

For real and complex operator systems Kavruk's proof of the above works except in one place, and indeed Lupini extended the complex version
to systems.  We expect that the operator system version of the above and many of the other results in \cite{Lupini} have real variants.
However we have not had time yet to check more than a couple of these.

\section{The real Kirchberg conjectures} \label{kk}

One difficulty that presents itself at the outset with the real Kirchberg conjectures, is that Kirchberg (and hence many others later) often used that every unital C$^*$-algebra $A$ is a quotient of the C$^*$-algebra of a free group, or equivalently that the Russo-Dye theorem holds
for $A$, which is false in the real case.  We write $\cC_{\bR}$ for the full real group C$^*$-algebra of $\bF_\infty$.

 Lemma 4.3 in  \cite{KPTT2}, which we said earlier works in the real case,  shows that 
 $I = \bR (1,1,-1,-1,-1)$ is a kernel in $\bR^5$, so that $\bR^5/I$ is a 4 dimensional real operator system for a unique matrix ordering (Proposition \ref{kuq}). 

By the real Connes embedding problem we mean the statement that every real von Neumann algebra $M$ 
with a faithful normal tracial state and separable predual is  embeddable as a von Neumann algebra in a trace preserving way in $R^\omega$.  The complex variant of the latter statement is a common formulation of the complex Connes embedding problem (see e.g.\ \cite{P}). 
 
\begin{theorem} \label{realK}  {\rm (Real Kirchberg conjectures)}\ The following conjectures are all equivalent (and are all false!):
\begin{itemize} \item [{\rm (i)}] $\cC_{\bR}$ has the real WEP.
 \item [{\rm (ii)}] $\cC_{\bR} \maxten \cC_{\bR} = \cC_{\bR} \minten \cC_{\bR}$. 
 \item [{\rm (iii)}] $\cC_{\bR} \maxten \cC_{\bR}$ has a faithful tracial state. 
\item [{\rm (iv)}] If $G$ is a free group then $C^*_{\bR}(G)$ has real WEP. 
\item [{\rm (v)}]  Any unital real C$^*$-algebra 
which is densely spanned by its unitaries, 
 has real QWEP. 
 \item [{\rm (v')}]  Any unital real C$^*$-algebra 
 has real QWEP. 
\item [{\rm (vi)}] Any real von Neumann algebra has real QWEP. 
\item [{\rm (vii)}] Any unital real C$^*$-algebra with the real LLP has the real WEP.
\item [{\rm (viii)}] $\cS_2 \minten \cS_2 = \cS_2 \cten \cS_2$ (fully real case).
\item [{\rm (ix)}] $C^*_{\bR}(\bF_2)   \minten  C^*_{\bR}(\bF_2) = C^*_{\bR}(\bF_2)   \maxten  C^*_{\bR}(\bF_2)$.
\item [{\rm (x)}] Real $\cS_2$ has the real DCEP.
\item [{\rm (xi)}] Every finite dimensional real operator system that has the real lifting property (LP) has the real DCEP.
\item [{\rm (xii)}]  The 4 dimensional real operator system $\cS = \bR^5/I$ satisfies $\cS \minten \cS = \cS \cten \cS$, where $I = \bR (1,1,-1,-1,-1)$. 
\item [{\rm (xiii)}] $\bR^5/I$ has real DCEP. 
\item [{\rm (xiv)}]  The real Connes embedding problem. 
\item [{\rm (xv)}]  The complex Connes embedding problem. 
 \end{itemize} \end{theorem}
 
\begin{proof}  By complexification in the way that we have done very many times above, items (i), (ii), (iv), (viii), (ix), and  (x) are each equivalent to the matching complex statement in the standard account of the (complex) Kirchberg conjectures (see e.g.\ \cite[Theorem 13.1]{P}). So they are each equivalent to each other (see e.g.\ \cite{P,Kavrukthesis,Kavrukw}), and are all false by \cite{MIP}.   Similarly for (iii) once we recall from Section \ref{rcast} that $\cS$ has a faithful real state if and only if $\cS_c$ has a faithful complex state.  Similarly for  (xii) and (xiii) since $\bC^5/J = \bC^5/I_c  \cong (\bR^5/I)_c$ is a complexification.  Thus the result follows from the complex case \cite[Theorem 5.14]{Kavrukw}.

Similarly  (vii) follows from  its complex variant, which in turn is is equivalent to (i)--(iv).  And (vii) implies (i) since $\cC_{\bR}$ has the real LLP.  This follows from 
Theorem \ref{ComOLLP} because 
$\cC = (\cC_{\bR})_c$ has the complex LLP and SLLP.

Clearly (v) implies (vi) by the Russo-Dye theorem for real von Neumann algebras \cite{Li,MMPS}.  And (iv) implies (v) because as in the complex case if 
a unital real C$^*$-algebra $A$ is densely spanned by its unitary group $G$, then there exists a $*$-homomorphism $C^*_{\bR}(G) \to A$. 

(i) $\Leftrightarrow$ (xi)\  It follows from \cite[Theorem 9.1]{KPTT2} that the complex (hence the real) variant of (i) is equivalent  to:
every complex operator system with the SLLP has the DCEP.
This implies every finite dimensional real operator system $\cS$ with $\cS_c$  having the lifting property (LP) has the DCEP.  Hence we have (xi).

(xi) $\Rightarrow$ (x)\ 
Complex $S_2$ has the complex SLLP by \cite[Proposition 9.9]{KPTT2}.  Hence real $S_2$ has the real SLLP, and hence has the LP.  Thus (x) is implied by (xi).

(vi) implies (v)'\ Given a unital real C$^*$-algebra $A$ then $A^{**}$ has QWEP by(vi).
So (v)' holds by the real version of part of \cite[Theorem 9.66]{P}, namely that $A$ has QWEP if $A^{**}$ has QWEP.  In turn the proof of this part  uses Corollaries 7.27 and  9.65 there.
(All numbers in the rest of this paragraph are as in \cite{P}.) 
Corollary 7.27 (and Proposition 7.26) are true in the real case by the same arguments, and we have used this elsewhere.  Corollary 9.65 uses Proposition  9.64 and Remark 9.13.  Proposition  9.64 cites (7.6), which is our Lemma \ref{maxex}.  Remark 9.13 is true in the real case.  It somewhat loosely cites (7.10) in Remark 7.20,
but the principles here (consequences of what Pisier calls `max-injectivity') are precisely the same in the real case, using our versions of the properties of the tensor products used.

We say nothing about the
proof of the equivalence with 
item (xiv) and (xv), since it is already known that the usual (complex) Connes embedding problem (CEP) is equivalent to the real version \cite{BDKS,Oz},
and also to the complex Kirchberg conjectures.   (We were not aware of this, and an earlier version of our paper gave a different proof, which we may perhaps present elsewhere.) 
Hence it is equivalent to 
the other items above. 

It remains to prove that (vi) implies (vii).  However all steps in the usual (lengthy) proof of this 
appear to work in the real case.  We give some more details.  
The part of the proof in \cite{P} involving decomposable maps is only actually applied to the canonical injective $*$-homomorphism $i_B : B \to B^{**}$ on a C$^*$-algebra $B$, which is completely positive.  So one may avoid using facts from the theory of decomposable maps 
(see \cite{BPr} for that).

The idea of the proof of (vi) implying (vii), following \cite{P}:  If $B$ is a real C$^*$-algebra with LLP
then by (vi) we see that $B^{**}$ has real QWEP.  We now need the real case of (i) implying (ii) in \cite[Theorem 9.67]{P}.  Let us take this for granted for now.  This yields that the linear map $i_B$ above satisfies that 
$I_{\cC_{\bR}} \otimes i_B$ is contractive as a map $\cC_{\bR} \minten B \to \cC_{\bR} \maxten B^{**}$.  Since we have a faithful $*$-homomorphic embedding 
$\cC_{\bR} \maxten B \subseteq \cC_{\bR} \maxten B^{**}$ (this follows by the same proof in the complex case, e.g.\ \cite[Proposition 7.26]{P}), it follows that we have a faithful $*$-embedding 
$\cC_{\bR} \minten B \subseteq \cC_{\bR} \maxten B$.  
Taking complexifications we see that $\cC \minten B_c = \cC \maxten B_c$ so that 
$B_c$ and hence $B$ have the WEP.  So (vii) holds.

Claim: The real case of (i) implying (ii) in \cite[Theorem 9.67]{P}, and its proof, is valid for a faithful unital $*$-homomorphism or ucp map $u$, and with $C_1 = \cC_{\bR}$: In the complex case this proof uses \cite[Corollary 9.40]{P}, which follows in the real case by complexifying.
 The proof referred to in the Claim also uses \cite[Theorem 9.38]{P}, that $u$ is locally liftable.  In the real case Eq.\ (9.7) there simply follows for our special $u$ by  functoriality of $\maxten$.  The rest of the proof that $u$ is locally liftable is as in the complex case.  There is an appeal to 
\cite[Proposition  7.48]{P}, but (certainly the part we need of) this latter Proposition is the same in the real case. 
More particularly, the argument should proceed as follows: for the inclusion $i_E : E \to C$ for a finite dimensional subsystem $E$ of $C$, fix a complete order embedding $E^d \subseteq \bB$, and let $s_E \in (E^d \minten C)^+ \subseteq (\bB \minten C)_+$ be the associated tensor.
Then $(I_{\bB} \otimes u)(s_E) \geq 0$ in $(\bB  \minten A)/(\bB  \minten I)$.  By Theorem \ref{tenp51} we see
$(I_{E^d} \otimes u)(s_E) \geq 0$ in $(E^d  \minten A)/(E^d  \minten I)$.  As in the argument for (2) implies (1) of \cite{KPTT2}, which we  checked earlier in the real case, $E^d  \minten I$ is biproximinal, and we can lift to an element in $(E^d  \minten A)^+$, which gives the desired completely positive lift $E \to A$.
\end{proof}

 {\bf Remark.}  
 1)\ One may of course multiply the equivalences in the last theorem, e.g.\ replacing $\cS_2$ by $C^*_{\bR}(\bZ_2 * \bZ_3)$ (see \cite[Theorem 5.14]{Kavrukw}) or other such examples appearing in the literature that work in the complex case.   

 \medskip
 
2)\  In connection with the last lines of the proof we mention that many other theorems in the extensive  local lifting theory  for  complex C$^*$-algebras and operator spaces and systems work in the real case.

\begin{corollary} \label{qweps} A real C$^*$-algebra $A$ is real QWEP  if and only if $A_c$ is complex QWEP, and  if and only if $A^{**}$ is real QWEP. 
If $A$ is densely spanned by unitaries (e.g.\ if $A$ is a von Neumann algebra) then this is also equivalent to the existence of a free group $G$ and a surjective $*$-homomorphism $\pi : C_{\bR}^*(G) \to A$ for which $I \otimes \pi$ induces a 
 $*$-homomorphism $\cC_{\bR} \minten C_{\bR}^*(G) \to \cC_{\bR} \maxten A$.  
\end{corollary}

\begin{proof} It follows from the complexification identity for quotients that if $A$ is real QWEP then $A_c$ is complex QWEP.  To prove the converse, if $A_c$ is complex QWEP then it is also real QWEP by Corollary \ref{comcpwep}.  Since there is a projection onto $A$, $A$ is real QWEP 
 by the real case of \cite[Corollary 9.61]{P}.    Thus we have that 
$A$ is real QWEP if and only if $A_c$ is complex QWEP.   

We saw in the proof of (v)' above that 
$A$ is real QWEP if $A^{**}$ is real QWEP.   Conversely,
if $A_c$ is complex QWEP then $(A_c)^{**} = (A^{**})_c$ is
complex QWEP.  So $A^{**}$ is real QWEP by the above. 

If $A_c$ is complex QWEP then by \cite[Theorem 9.67]{P} there exists a free group $G$ such that the usual surjective $*$-homomorphism $\pi : C_{\bR}^*(G) \to A$,  $I_c \otimes \pi_c$  induces a 
 $*$-homomorphism $\cC \minten C_{\bC}^*(G) \to \cC  \maxten A_c$.   Then restrict to the copy of $\cC_{\bR} \minten C_{\bR}^*(G)$. 
 This argument can be reversed.  Suppose that we have a surjective $*$-homomorphism $\pi : C_{\bR}^*(G) \to A$ for which $I \otimes \pi$ induces a 
 $*$-homomorphism as stated.  Complexifying this latter $*$-homomorphism shows that $A_c$ is complex QWEP by \cite[Theorem 9.67]{P}, so 
 $A$ is QWEP.

 Every real von Neumann is densely spanned by unitaries by the Russo-Dye theorem for such algebras \cite{Li,MMPS}.
\end{proof}

{\bf Remarks.}  1)\ It now follows 
as in the proof of Corollary \ref{comcpwep} that a  complex C$^*$-algebra  has the complex QWEP if and only if it has the real QWEP.

\smallskip

2)\ Several of the other characterizations of QWEP in \cite{P} (namely 9.69--9.75 there) are valid in the real case, by complexification.  We leave these to the reader.

\bigskip

It is known that the Kirchberg conjectures are equivalent to that the predual of any  von Neumann algebra is finitely representable in the trace class $S_1$.  Indeed Kirchberg showed that a 
C$^*$-algebra is QWEP if and only if its dual is finitely representable in  $S^1$; and that a von Neumann algebra is QWEP if and only if it is Banach space isometric to 
a Banach space quotient of $B(H)$.

\begin{theorem} \label{frr}  
The dual of a real von Neumann algebra $M$ is 
completely finitely representable in $M_*$.  Also, $M$ is QWEP if and only if
 it is 
completely isometric to 
an operator space quotient of $B(H)$ for a real Hilbert space $H$, and if and only if $M_*$ is completely finitely representable in the trace class $S^1$.
Hence the real Kirchberg conjectures are equivalent to that the dual of any real 
C$^*$-algebra, or predual of any real von Neumann,  is completely finitely representable in the trace class $S_1$. 
\end{theorem}

\begin{proof} If a C$^*$-algebra $A$ is real QWEP then
$A_c$ is complex QWEP, hence completely finitely representable in the trace class $S^1(H_c)$.  For any finite dimensional subspace 
$E$ of $A^*$, $E_c$ is a subspace 
 of $(A_c)^*$, so that there is a subspace 
 $F$ of $S_1(H_c)$ which is $(1+\epsilon)$-cb-isomorphic to $E_c$.  Hence $E$ is real $(1+\epsilon)$-cb-isomorphic to
 a finite dimensional real subspace 
 of $S^1(H_c)$.  However $B(H_c)$ is a weak* closed complemented
 subspace of $B_{\bR}(H_c)$, and  $S^1(H_c)$
 may be viewed as a real subspace of $S^1_{\bR}(H_c)$.  Here $H = \ell^2$.

The proof of \cite[Theorem 15.5]{P} works in the real case to show that, first, if a real von Neumann algebra $M$ is QWEP then it is an operator space quotient of $B(H)$.   Second, if $M$  is an operator space quotient of $B(H)$ then $M^*$ (and hence $M_*$) is completely finitely representable in $S^1$.  For this we need to notice that the proof in \cite{EJR} that $M^*$ is 
completely finitely representable in $M_*$, works in the real case.  Indeed Theorem 6.6 there, and its proof, and its corollary the strong local reflexivity of $R_*$, 
work similarly in the real case. 

Finally, suppose that $M_*$ is 
completely finitely representable in $S^1$.
Suppose that $F$ is a finite dimensional subspace 
 of $(M_c)_* \cong (M_*)_c$.  Choose 
 a finite dimensional subspace $E \subseteq M_*$ with
 $F \subseteq E_c$. There is a subspace 
 $F$ of $S_1(H)$ which is $(1+\epsilon)$-cb-isomorphic to $E$.  Hence $E_c$ is complex $(1+\epsilon)$-cb-isomorphic to $F_c \subseteq S^1(H)_c \cong S_1(H_c)$.  Again $H = \ell^2$.
 Hence $M_c$ has complex QWEP, and hence $M$ has real QWEP. 
\end{proof}

{\bf Remark.} The converses of the statements in the last result are probably true at the Banach space level, as they are in the complex case \cite{P}.
However the route of the complex proof of the latter result in 
\cite[Section 15]{P} 
is obstructed 
by the Jordan algebra part of the argument  (due originally to Ozawa \cite{Ozqwep}).  
Unfortunately  the perhaps crucial Theorem 5.5 there on multiplicative domain of positive unital maps is false in the real case, even for maps on $M_2$. 
Note that \cite[Corollary  9.70]{P} (which is also used) in the real case follows by the complex case and complexification, via Remark 1 above.

Nonetheless, since $B_{\bR}(X_r,\bR) \cong B_{\bC}(X,\bC)_r$ isometrically for any complex Banach space $X$ (see \cite[Proposition 1.1.6]{Li}), it is clear that if the dual of any real 
C$^*$-algebra is finitely representable in the real trace class $S_1$  
then the complex dual of any complex 
C$^*$-algebra $N$ is {\em real} finitely representable in the real trace class $S_1$,
and hence in the complex trace class.  Indeed the desired result seems to be true with the completely bounded norm replaced by the norm at the level of $2 \times 2$ matrices.

This is related to a hard open problem in the subject mentioned to us by N. Ozawa: it is not known 
in the complex case if QWEP is implied by the dual $A^*$ being (complex) $C$-finitely representable (that is, 
finitely representable up to a fixed constant $C > 1$).   If this were the case then we can complete 
the argument at the end of in the last paragraph 
as follows.  If we complexify once more,  
 we see that $(N_*)_c \cong (N_c)_*$ (or $(N^*)_c \cong (N_c)^*$) is complex 
$C$-finitely representable in the complex trace class.
This implies that $N_c$ has complex QWEP, and hence $N$ has real QWEP.  

\section{Abstract projections and Tsirelson's problem} \label{aptp}

In this section, we consider the real case of two separate problems, which are related in the complex case by \cite{AR23} and \cite{ART23}. The first involves abstractly characterizing those elements of an operator system $V$ whose image in the C*-envelope $C^*_e(V)$ is a projection. The second involves describing Tsirelson's correlation sets in terms of operator systems and their projections. By considering the real case of these problems, we will show that real and complex quantum correlations sets are equal and the real and complex quantum commuting correlation sets are equal.

\begin{definition}
    Let $V$ be a real or complex operator system. A self-adjoint element $p \in V$ is an \textit{abstract projection} if whenever $x \in M_n(V)_{\rm sa}$ has the property that for every $\epsilon > 0$ there exists $t > 0$ such that \[ \begin{bmatrix} x & x \\ x & x \end{bmatrix} + \epsilon \begin{bmatrix} p \otimes I_n & 0 \\ 0 & (I-p) \otimes I_n \end{bmatrix} + t \begin{bmatrix} (I-p) \otimes I_n & 0 \\ 0 & p \otimes I_n \end{bmatrix} \geq 0 \] it follows that $x \geq 0$.
\end{definition}

The main result of \cite{AR23} is that whenever $V$ is a complex operator system and $p \in V$ is an abstract projection, then there exists a Hilbert space $H$ and a complete order embedding $\pi: V \to B(H)$ such that $\pi(p)$ is a projection operator (see Theorem 5.3 there). In particular, the image of $p$ in $C^*_e(V)$ is a projection whenever $p$ is an abstract projection.

\begin{theorem}
    Let $V$ be a real operator system and suppose that $p \in V$ is an abstract projection. Then $p$ is an abstract projection in $V_c$. Consequently $p$ is a projection in $C^*_e(V)$.
\end{theorem}

\begin{proof}
    Suppose $x + iy \in M_n(V_c)_{\rm sa}$ and that for every $\epsilon > 0$ there exists $t > 0$ such that \[ \begin{bmatrix} x + iy & x + iy \\ x + iy & x + iy \end{bmatrix} + \epsilon \begin{bmatrix} p \otimes I_n & 0 \\ 0 & (I-p) \otimes I_n \end{bmatrix} + t \begin{bmatrix} (I-p) \otimes I_n & 0 \\ 0 & p \otimes I_n \end{bmatrix} \in M_{2n}(V_c)^+. \] By the definition of $M_n(V_c)^+$, we see that \[ \begin{bmatrix} x & -y \\ y  & x  \end{bmatrix} \otimes J + \epsilon I_2 \otimes \begin{bmatrix} p \otimes I_n & 0 \\ 0 & (I-p) \otimes I_n \end{bmatrix} + t I_2 \otimes \begin{bmatrix} (I-p) \otimes I_n & 0 \\ 0 & p \otimes I_n \end{bmatrix} \in M_{4n}(V)^+ \] where $J = \begin{bmatrix} 1 & 1 \\ 1 & 1 \end{bmatrix}$. Conjugating by the permutation matrix \[ U = \begin{bmatrix} I_n & 0 & 0 & 0 \\ 0 & 0 & I_n & 0 \\ 0 & I_n & 0 & 0 \\ 0 & 0 & 0 & I_n \end{bmatrix}, \] we get that \[ J \otimes \begin{bmatrix} x & -y \\ y & x \end{bmatrix} + \epsilon \begin{bmatrix} p \otimes I_{2n} & 0 \\ 0 & (I-p) \otimes I_{2n} \end{bmatrix} + t \begin{bmatrix} (I-p) \otimes I_{2n} & 0 \\ 0 & p \otimes I_{2n} \end{bmatrix} \in M_{4n}(V)^+. \] Since $p$ is an abstract projection for the real operator system $V$, we conclude that $c(x,y) \in M_{2n}(V)^+$ and thus $x+iy \in M_n(V_c)^+$. Therefore $p$ is an abstract projection in $V_c$ and hence is a projection in $C^*_e(V_c)$. Since $C^*_e(V) \subseteq C^*_e(V)_c = C^*_e(V_c)$, $p$ is a projection in $C^*_e(V)$ as well.
\end{proof}

Fix integers $n,m,k \in \mathbb{N}$. A \textit{correlation} (in an $k$-partite 
scenario) consists of a tuple of positive real numbers $\{p(\vec{a}|\vec{b}) : \vec{a} \in [n]^k, \vec{b} \in [m]^k\}$, where $[N] := \{1,2,3,\dots,N\}$. The correlation is \textit{non-signalling} if it satisfies the non-signalling condition that
\[ p_i(a_i|b) := \sum_{a_j, j \neq i} p(a_1,a_2,\dots,a_k|b_1,\dots,b_{i-1},b,b_{i+1},\dots,b_k) \] is well-defined, i.e. is the same for every $\vec{b}$ with $b_i = b$.
A correlation is \textit{quantum commuting} if there exists a complex C*-algebra $A$, projection-valued measures 
$\{P_a^b(1)\}_{a=1}^n, \{P_a^b(2)\}_{a=1}^n, \dots, \{P_a^b(k)\}_{a=1}^n$ in $A$ for each $b \in [m]$ so that $[P_a^b(i), P_{a'}^{b'}(j)] = 0$ whenever $i \neq j$, and a state $\varphi: A \to \mathbb{C}$ such that \[ p(\vec{a}| \vec{b}) = \varphi( \prod_{i=1}^k P_{a_i}^{b_i}(i) ) \] for all $\vec{a}, \vec{b}$  above.
In the case when $A$ is finite dimensional, we call $p$ a \textit{quantum correlation}. We let $C_{qc}(n,m,k)$ and $C_q(n,m,k)$ denote the sets of quantum-commuting and quantum correlations, respectively. It is known that these sets are convex subsets of real Euclidean space and that $C_{qc}(n,m,k)$ is always closed (c.f. Proposition 5.3 of \cite{Fritz})

\bigskip 

{\bf Remark}: The sets $C_{qc}$ and $C_q$ are usually defined in the following spatial manner. A tuple $p(\vec{a}|\vec{b})$ is in $C_{qc}(n,m,k)$ if there exists a Hilbert space $H$ and projection-valued measures $\{P_a^b(1)\}_{a=1}^m, \{P_a^b(2)\}_{a=1}^m, \dots, \{P_a^b(k)\}_{a=1}^m \subseteq B(H)$ for each $b \in [n]$ so that $[P_a^b(i), P_{a'}^{b'}(j)] = 0$ whenever $i \neq j$, and a unit vector $h \in H$ such that \[ p(\vec{a}| \vec{b}) = \langle \prod_{i=1}^k P_{a_i}^{b_i}(i) h, h \rangle. \] Letting $A$ be the C*-algebra generated by the projection-valued measures and $\varphi: A \to \mathbb{C}$ be $\varphi(x)=\langle xh, h\rangle$ shows such a correlation is a correlation in the sense we have defined it. The other direction is an application of the GNS construction on the C*-algebra $A$. For $C_q$, the usual definition is similar except that we assume $H = H_1 \otimes H_2 \otimes \dots H_k$ with each $H_i$ finite-dimensional, and we require each projection $P_a^b(k)$ to be an elementary tensor product of projections on the $H_i's$ with the $i$-th projection equal to $I_{H_i}$ for $i \neq k$. Such projections clearly generate a finite-dimensional C*-algebra with $[P_a^b(i), P_{a'}^{b'}(j)] = 0$ whenever $i \neq j$. To go from our definition of $C_q$ to the spatial definition is more delicate: see Remark 2.5 of \cite{DPP}, or more directly Theorem 5.3 of \cite{PSSTW}, for the bipartite case. Adapting the proof of Theorem 5.3 of \cite{PSSTW} gives the multi-partite case. 

\begin{theorem}[c.f.\ Theorem 1 of \cite{JNP+}, Theorem 4.1 of \cite{Fritz}, and Theorem 36 of \cite{Oz}]
    Kirchberg's conjecture is equivalent to $C_{qc}(n,m,2) = \overline{C_q(n,m,2)}$ for all $n, m \in \mathbb{N}$.
\end{theorem}

The paper MIP*=RE \cite{MIP} proves that there exist $n$ an $m$ such that $C_{qc}(n,m,2) \neq \overline{C_q(n,m,2)}$. The proof is non-constructive, but gives estimates of $n,m \geq 10^{20}$. No sharp upper bound on the size of $n$ and $m$ needed for $C_{qc}(n,m,2) \neq \overline{C_q(n,m,2)}$ are known. For instance, it is open if $C_{qc}(3,2,2) = \overline{C_q(3,2,2)}$ or not.

If we replace the complex C*-algebra with a real C*-algebra in the definition of a quantum commuting correlation (respectively, a quantum correlation), then we get a potentially different set of tuples which we denote $C_{qc}^{\bR}(n,m,k)$ (or $C_q^{\bR}(n,m,k)$, respectively). By characterizing these sets in terms of operator systems and their states, we will see that there is no difference in the real and complex cases.

\begin{definition}
    Fix integers $n,m,k \in \mathbb{N}$. A \textit{quantum-commuting operator system} $\cS$ is an operator system with unit $e$ which is spanned by abstract projections  $$\{Q(\vec{a}|\vec{b}) : \vec{a} \in [m]^n, \vec{b} \in [k]^n\}$$ which satisfy the non-signaling condition that 
    \[ Q_i(a_i|b) := \sum_{a_j, j \neq i} Q(a_1,a_2,\dots,a_k|b_1,\dots,b_{i-1},b,b_{i+1},\dots,b_k) \] is well-defined in the sense above,  and the condition \[ \sum_{\vec{a} \in [m]^n} Q(\vec{a}|\vec{b}) = e \] for every $\vec{b} \in [k]^n$. If $\cS$ is $k$-minimal for some $k$, then we call $\cS$ a \textit{quantum operator system}.
\end{definition}

The above definition is more general than the one in \cite{AR23} and \cite{ART23}, where we only consider the case $k=2$. However, the argument is essentially unchanged. Because the sum over all $\vec{a}$ of $Q(\vec{a}|\vec{b})$ is the identity, $\{Q(\vec{a}|\vec{b}) \}_{\vec{a}}$ is a projection-valued measure for any fixed $\vec{b}$. Hence $Q(\vec{a}| \vec{b}) Q(\vec{a}' | \vec{b}) = 0$ whenever $\vec{a} \neq \vec{a}'$. It is now easily seen that \[ \prod_{i=1}^k Q_i(a_i|\vec{b}) = Q(\vec{a}|\vec{b}). \] 

\begin{proposition}
    The complexification of a quantum (resp.\ quantum-commuting)
     operator system $\cS$ is a quantum (resp.\ quantum-commuting) operator system.
\end{proposition}

\begin{proof}
    Since each generator $Q(\vec{a}|\vec{b})$ of $\cS$ is an abstract projection in $\cS$, it is an abstract projection in $\cS_c$ as well. If $\cS$ is $k$-minimal for some $k$, then $\cS$ is also $k$-minimal. The non-signalling conditions still hold in $\cS_c$ since the embedding $\cS \to \cS_c$ is a linear isomorphism.
\end{proof}

The next result is known in the (finite dimensional) quantum setting (see \cite{Nav}). 
We do not know if it was known in the general setting. For the readers convenience we supply a quick 
proof. 

\begin{corollary}
    The real and complex Tsirelson correlation sets coincide.
\end{corollary}

\begin{proof}
    It is shown in \cite{AR23} that each 
    $p \in C_{qc}(n,m,k)$ if and only if there exists a complex quantum commuting operator system $\cS = \text{span} \{Q(\vec{a}|\vec{b})\}$ and a state $\varphi: \cS \to \mathbb{C}$ such that $p(\vec{a}|\vec{b}) = \varphi(Q(\vec{a}|\vec{b}))$. Suppose $p \in C_{qc}^{\bR}(n,m,k)$. Then there exists a real C*-algebra, projection valued measures $\{P_a^b(k)\} \subseteq A$, and a state $\varphi: A \to \mathbb{R}$ such that \[ p(\vec{a}| \vec{b}) = \varphi( \prod_{i=1}^k P_{a_i}^{b_i}(i) ). \] As the linear span of $Q(\vec{a}|\vec{b}) := \prod_{i=1}^k P_{a_i}^{b_i}(i)$ defines a real quantum-commuting operator system, the same is true of its complexification, and the state $\varphi_c$ agrees with $\varphi$ on the generators $Q(\vec{a}|\vec{b})$. Thus $C_{qc}^{\bR}(n,m,k) \subseteq C_{qc}^{\bR}(n,m,k)$. On the other hand, given a complex quantum commuting operator system $\cS$ with generators $Q(\vec{a}|\vec{b})$ and state $\varphi: \cS \to \bC$, we obtain a real quantum-commuting operator system by forgetting the complex structure, and a real state by setting $\varphi'(x) = \Re(\varphi(x))$. Since $\varphi(Q(\vec{a}|\vec{b})) \geq 0$, $p(\vec{a}|\vec{b}) = \varphi(Q(\vec{a}|\vec{b})) = \varphi'(Q(\vec{a}|\vec{b})).$ So $C_{qc}(n,m,k) = C_{qc}^{\bR}(n,m,k)$. For quantum correlations, the arguments are the same except we note that when $A$ is finite-dimensional, $A$ is $d$-minimal for some $d$. This follows easily from Corollary \ref{cor: k-minimal characterization} 
    and \cite[Theorem 5.7.1]{Li}).  
    It is shown in \cite{ART23} that each 
    $p \in C_{q}(n,m,k)$ if and only if there exists a ($d$-minimal) quantum operator system $\cS = \text{span} \{Q(\vec{a}|\vec{b})\}$ and a state $\varphi: \cS \to \mathbb{C}$ such that $p(\vec{a}|\vec{b}) = \varphi(Q(\vec{a}|\vec{b}))$. From these facts we see that $C_{q}(n,m,k) = C_{q}^{\bR}(n,m,k)$.
\end{proof}

{\bf Acknowledgements.} This work was presented in two lectures at 
"Operator Systems and their Applications" BIRS workshop 25w5405.  We thank the organizers, and the 
BIRS institute, and thank many of the attendees there for their comments, particularly Vern Paulsen and Taka Ozawa.   We thank Caleb McClure for spotting some typos in the first ArXiV version, B. Xhabli for an interchange related to Theorem \ref{fdom}, and the referee for several useful comments. 
D.P. Blecher acknowledges support by NSF Grant DMS-2154903.  T. Russell  is supported by grants from  the Binational Science Foundation (BSF-2024161) and the TCU Research and Creative Activities Fund.

\end{document}